\documentclass[oneside,12pt]{amsart}

\usepackage{amsfonts,amssymb,amsmath,amsthm,graphicx,fancyhdr,color}
\usepackage[hidelinks]{hyperref}
\usepackage{array}
\usepackage{orcidlink}
\usepackage{float}
\usepackage{placeins}

\usepackage[ddmmyyyy,hhmmss]{datetime}

\usepackage{float,caption}
\usepackage{tikz}
\usepackage[mode=tex]{standalone}

\usepackage{bbm}

\usepackage{mathtools}

\usepackage[T1]{fontenc}
\usepackage[utf8]{inputenc}

\usepackage{geometry} 
\usepackage{amsthm,amsmath,amsfonts,amssymb,amstext} 
\usepackage{enumitem,float,pgfplots}
\usepackage{tikz,standalone}
\usepackage{mathtools}
\usepackage{MnSymbol} 
\usepackage{stmaryrd}
\usepackage{relsize}
\usepackage{array,multirow,makecell}
\usepackage{verbatim}
\usepackage{bbm}
\usepackage{accents}
\usepackage{fancyhdr}
\usepackage{graphicx}
\usepackage{subcaption}
\usepackage[ddmmyyyy,hhmmss]{datetime}
\usepackage{pdfpages}
\usepackage{lipsum}

\def\card#1{ | #1 | } 
\def\id{\textnormal{id}} 

\def\NN{\mathbb{N}} 
\def\NNp{\NN_{+}} 
\def\ZZ{\mathbb{Z}} 
\def\RR{\mathbb{R}} 
\def\CC{\mathbb{C}} 
\def\smallO#1{\mathrm{o} \left( #1 \right)} 
\def\lfrac#1#2{#1 / #2} 
\def\diff#1{\textnormal{d} #1} 
\def\complexAdjoint#1{ \overline{#1} }

\def\supp{\textnormal{supp} \, }
\def\const{\textnormal{const}}
\def\sgn{\textnormal{sgn}\,}
\def\principalLogarithm{\textnormal{Log}\,}

\def\principalArgument{\textnormal{Arg}\,}

\def\weaklyConvergesTo{\overset{d}{\to}} 

\def\part{\mathcal{P}} 
\def\partTree#1{ \mathcal{T}_{#1} }
\def\emptyWord{ () } 
\def\fWords#1{\mathrm{w}(#1)} 

\def\vWords#1#2{\mathcal{V}_{#1}(#2)} 
\def\vChar{\mathbbm{1}_{\mathrm{V}}} 

\def\vDiscreteFockSpace{ \mathcal{F}_{\mathrm{vm}}^{\mathrm{d}} }

\def\part{\mathcal{P}} 
\def\oPart{\mathcal{OP}} 

\def\noncr{\mathcal{NC}} 
\def\noncrk#1{\mathcal{NC}^{#1}}
\def\oNoncr{\mathcal{ONC}} 
\def\oNoncrk#1{\mathcal{ONC}^{#1}}

\def\vNoncr{\mathcal{OV}} 
\def\vNoncrk#1{\mathcal{OV}^{#1}}
\def\vlNoncr#1{\mathcal{LV}_{#1}} 
\def\vlNoncrk#1#2{\mathcal{LV}_{#2}^{#1}}

\def\momentExact#1#2#3{ m_{#1}(#2, #3) } 
\def\momentLimit#1#2{ m_{#1}(#2) } 
\def\mgfExact#1#2{ M_{#1, #2} } 
\def\mgfLimit#1{ M_{#1} } 

\numberwithin{equation}{section}

\theoremstyle{plain}
\newtheorem{theorem}{Theorem}[section]
\newtheorem{lemma}[theorem]{Lemma}
\newtheorem{proposition}[theorem]{Proposition}
\newtheorem{corollary}[theorem]{Corollary}
\newtheorem*{theorem*}{Theorem}
\newtheorem*{lemma*}{Lemma}
\newtheorem*{proposition*}{Proposition}
\newtheorem*{corollary*}{Corollary}

\theoremstyle{definition}
\newtheorem{definition}[theorem]{Definition}
\newtheorem{remark}[theorem]{Remark}
\newtheorem{example}[theorem]{Example}
\newtheorem*{definition*}{Definition}
\newtheorem*{remark*}{Remark}
\newtheorem*{example*}{Example}

\def\complexAdjoint#1{ \overline{#1} }
\def\supp{\textnormal{supp} \,}

\newcommand{\tp}{\mathbin{\hbox{$\bigcirc$\hbox to 0pt{\hspace{-0.81em}$\scriptstyle\top$\hfil}}}}
\def\partOpDiscA#1#2{b_{#2}^{#1}}
\def\partOpDiscB#1#2{ b_{\left( {#1}, {#2} \right)} }

\newcommand{\ml}[1]{\mathfrak{ml}\left(#1\right)}
\newcommand{\legs}[1]{n \left( #1 \right)}
\def\lfrac#1#2{#1 / #2}

\DeclareRobustCommand{\ncPartitionRecurrenceB}[2]
{
   	\vcenter
   	{	
   		\hbox
   		{
   			\tikz
   			{
				\tikzstyle{leg} = [circle, draw=black, fill=black!100, text=black!100, thin, inner sep=0pt, minimum size=2.0]

                \pgfmathsetmacro {\x}{0}
                \pgfmathsetmacro {\y}{0}

                \pgfmathsetmacro {\dx}{0.75}
                \pgfmathsetmacro {\dy}{0.5}

				\draw[-] (\x+0*\dx,\y+0*\dy) -- (\x+0*\dx,\y+1*\dy);
				\node () at (\x+0*\dx,\y+0*\dy) [leg] {.};
				\draw[-] (\x+1*\dx,\y+0*\dy) -- (\x+1*\dx,\y+1*\dy);
				\node () at (\x+1*\dx,\y+0*\dy) [leg] {.};
				\draw[-] (\x+0*\dx,\y+1*\dy) -- (\x+1*\dx,\y+1*\dy);

				\node () at (\x+0.5*\dx,\y+0*\dy) [label={[label distance = -1.5mm]above:\small{$#1$}}] {};
				\node () at (\x+1.5*\dx,\y+0*\dy) [label={[label distance = -1.5mm]above:\small{$#2$}}] {};
			}
		}
	} \! \! \! \!
}

\DeclareRobustCommand{\ncPartitionRecurrenceD}[4]
{
   	\vcenter
   	{	
   		\hbox
   		{
   			\tikz
   			{
				\tikzstyle{leg} = [circle, draw=black, fill=black!100, text=black!100, thin, inner sep=0pt, minimum size=2.0]

                \pgfmathsetmacro {\dx}{0.75}
                \pgfmathsetmacro {\dy}{0.65}

                \pgfmathsetmacro {\x}{0}
                \pgfmathsetmacro {\y}{0}
				\draw[-] (\x+0*\dx,\y+0*\dy) -- (\x+0*\dx,\y+1*\dy);
				\node () at (\x+0*\dx,\y+0*\dy) [leg] {.};
                \foreach \p in {#1, #2, #3}
                {
                    \draw[-] (\x+0*\dx,\y+1*\dy) -- (\x+\dx,\y+1*\dy);

				    \draw[-] (\x+\dx,\y+0*\dy) -- (\x+\dx,\y+1*\dy);
				    \node () at (\x+\dx,\y+0*\dy) [leg] {.};

                    \node () at (\x+0.5*\dx,\y+0*\dy) [label={[label distance = -1.5mm]above:\small{$\p$}}] {};
                    \xdef\x{\x+\dx}
                }
				\node () at (\x+0.5*\dx,\y+0*\dy) [label={[label distance = -1.5mm]above:\small{$#4$}}] {};
			}
		}
	} \! \! \! \!
}

\usepackage{fancyhdr}
\allowdisplaybreaks
\begin{document}
\title{Distribution for nonsymmetric V-monotone position operators}
\author[A. Dacko]{Adrian Dacko$^{1, \orcidlink{0009-0004-9386-4445}}$}
\address{$^1$ Wroc\l{}aw University of Environmental and Life Sciences,
Faculty of Environmental Engineering and Geodesy, Adrian.Piotr.Dacko@gmail.com, Adrian.Dacko@upwr.edu.pl}
\author[L. Oussi]{Lahcen Oussi$^{2, \orcidlink{0000-0001-9804-0761}}$}
\address{$^2$ Wroclaw University of Science and Technology, Faculty of Pure and Applied Mathematics, Lahcen.Oussi@pwr.edu.pl, oussimaths@gmail.com}

\begin{abstract}
We investigate the vacuum distribution of a family of partial sums of nonsymmetric position operators, depending on a real parameter $\lambda$, and acting on the discrete Fock space in the framework of V-monotone independence. We analyze the combinatorics of the moments of this distribution, and using its Cauchy--Stieltjes transform, we determine its exact form, consisting of a unique atom and an absolutely continuous part. Finally, we present computer-generated graphs that illustrate the distribution for several values of the intensity parameter $\lambda$.
\end{abstract}

\maketitle

\section{Introduction}
\label{section:introduction}
In this paper, we study the vacuum distribution of a sum of noncommutative random variables, namely
\begin{equation}
\label{eq:GeneralSumPLT}
S_{N}(\lambda) \coloneqq \sum_{i=1}^{N} \left( \frac{ A_{i}^{+}+A_{i}^{-} } {\sqrt{N} } + \lambda A_{i}^{\circ} \right) \text{,} \qquad N \in \NNp \text{, } \lambda \in \RR \text{.}
\end{equation}
where $A_{i}^{+}$, $A_{i}^{-}$ and $A_{i}^{\circ}$, for $i \in \NNp = \{ 1, 2, 3, \ldots \}$, are creation, annihilation, and conservation operators, respectively, acting on a V-monotone analog of the full Fock space, called the \emph{V-monotone Fock space}. 

A \emph{noncommutative (or quantum) probability space} is a pair $\left( \mathcal{A}, \varphi \right)$, where $\mathcal{A}$ is a unital $C^*$-algebra with unit $\mathbf{1}$, and $\varphi \colon \mathcal{A} \to \CC$ is a \emph{state}, that is, a linear functional such that $\varphi(\mathbf{1})=1$ and $\varphi(aa^*) \geq 0$. Self-adjoint elements $a \in \mathcal{A}$ are called \emph{noncommutative random variables}.
A \emph{distribution} of a noncommutative random variable $a \in \mathcal{A}$ is a probability measure $\mu_a$ on $\RR$ such that
\begin{equation*}
\varphi(a^n) = \int_{-\infty}^{\infty} x^n \mu_a(\diff{x}) \text{}
\end{equation*}
for each positive integer $n$. One can show that such a measure exists and is uniquely determined by the moment sequence $\left( \varphi(a^n) \right)_{n=1}^{\infty}$.
Let $(\mathcal{A}_N, \varphi_N)_{n=1}^{\infty}$ and $(\mathcal{A}, \varphi)$ be noncommutative probability spaces and let $(a_N)_{n=1}^\infty$ be a family of random variables such that $a_N \in \mathcal{A}_N$ for each $N \in \NNp$. We say that $a_N$ \emph{converges in distribution} to a random variable $a \in \mathcal{A}$ as $N \to \infty$, if we have
\begin{equation*}
\lim_{N \to \infty} \varphi_N(a_N^n) = \varphi(a^n) \text{}
\end{equation*}
for all positive integers $n$, and we write $a_N \weaklyConvergesTo a$. Since all random variables are compactly supported here, the above condition is equivalent to the fact that
\begin{equation*}
\lim_{N \to \infty} \int_{-\infty}^{\infty} \psi(x) \mu_{a_N}(\diff{x}) = \int_{-\infty}^{\infty} \psi(x) \mu_a(\diff{x}) \text{}
\end{equation*}
for each continuous bounded function $\psi \colon \RR \to \RR$.
For more general tools and background on noncommutative probability and related topics, we refer the reader to \cite{HO, JOW2025, Mey, ANRS2006, Par, VDN1992} and references therein.

There are several notions of independence in this framework, the universal ones being: free independence of Voiculescu~\cite{DV1985}, monotone independence of Muraki~\cite{Mur1}, and Boolean independence~\cite{RSRW1997}. There are other notions of noncommutative independence, where some of them are mixtures of these natural ones. In the present paper, we develop and focus on the theory of the notion of independence called \emph{V-monotone}, introduced in~\cite{AD2020}. It can be considered as a combination of monotone and antimonotone independence, which is best seen on a Fock space level. We now recall the definition of this independence.

Let $(\mathcal{A}, \varphi)$ be a noncommutative probability space and let $(\mathcal{A}_{i})_{i\in I}$ be a family of subalgebras of $\mathcal{A}$, where $(I, \leq)$ is a totally ordered set. We assume that there exists a family $(\mathbf{1}_i)_{i\in I}$ such that for each $i \in I$, the random variable $\mathbf{1}_i \in \mathcal{A}_i$ is an \emph{inner unit} of $\mathcal{A}_i$, i.e., $\mathbf{1}_i a = a \mathbf{1}_i = a$ for all $a \in \mathcal{A}_i$. 

Let $\fWords{I}$ be the set of all finite words over the (possible infinite) alphabet $I$, i.e.,
\begin{equation*}
\fWords{I} \coloneqq \bigcup_{n = 0}^{\infty} I^n \text{,}
\end{equation*}
where $I^0 \coloneqq \{ \emptyWord \}$ contains only the empty word. Define
\begin{equation*}
\vWords{n, m}{I} \coloneqq \{ (i_1, \ldots, i_n )\in I^n: i_1 > \cdots > i_m< \cdots < i_n \} \text{,} \quad n \in \NNp \text{, } m \in [n] \text{,}
\end{equation*}
where $[n] = \{ 1, 2, \ldots, n \}$, with the convention that $\vWords{1, 1}{I} = I$, and for $n \in \NN$, set
\begin{equation*}
\vWords{n}{I} \coloneqq
\begin{cases}
\left\{ \emptyWord \right\} & \text{if $n=0$,} \\
\bigcup_{m=1}^{n} \vWords{n, m}{I} & \text{if $n > 0$.}
\end{cases}
\end{equation*}
For $i\in I$ and $(i_1, \ldots, i_n) \in \vWords{n}{I}$, we use the notation
\begin{equation*}
i \sim (i_1, \ldots, i_n) \text{ if } (i, i_1, \ldots, i_n) \in \vWords{n+1}{I} \text{ and } (i_1, \ldots, i_n) \sim i \text{ if } (i_1, \ldots, i_n, i) \in \vWords{n+1}{I} \text{.}
\end{equation*}

\begin{definition}
\label{definition:vMonotoneIndependence}
We say that the family $(\mathcal{A}_i)_{i \in I}$ is \emph{V-monotonically independent} with respect to $\varphi$, if for any $i \in I$, we have $\varphi(\mathbf{1}_i) = 1$, and for any sequence $(i_1, \ldots, i_n) \in I^n$ such that $i_k \neq i_{k+1}$ for $k \in [n-1]$ and for any $a_1 \in \mathcal{A}_{i_1}, \ldots, a_n \in \mathcal{A}_{i_n}$, the following conditions are satisfied:
\begin{itemize}
	\item[ (i)] $\varphi(a_1 \ldots a_n) = 0$ if $\varphi(a_1) = \cdots = \varphi(a_n) = 0$,
	\item[(ii)] for any $j \in [n]$ we have
	$$
	\varphi(a_1 \ldots a_{j-1} \mathbf{1}_{i_j} a_{j+1} \ldots a_n) = 
	\begin{cases}
	\varphi(a_1 \ldots a_{j-1} a_{j+1} \ldots a_n) & \text{if $j \leq r$,} \\
	0 & \text{otherwise,}
	\end{cases}
	$$
	whenever $\varphi(a_1) = \cdots = \varphi(a_{j-1}) = 0$ and $r \in [n]$ is such that $\vWords{r}{I} \ni (i_1, \ldots, i_r) \sim i_{r+1}$ (without the condition with `$\sim$' for $r=n$).
\end{itemize}
We say that random variables from a family $(a_i)_{i \in I}$ are V-monotone independent if there exists a V-monotone independent family of algebras $(\mathcal{A}_i)_{i \in I}$ such that $a_i \in \mathcal{A}_i$ for each $i \in I$.
\end{definition}
For more details on V-monotone independence, see~\cite{AD2020}. Throughout this paper, we only consider $I = \NN$ with the canonical total order.

In many models of noncommutative independence, one can define an associated Poisson distribution with the intensity parameter $\lambda > 0$. A Poisson limit theorem, also known as the law of small numbers, can be formulated as follows:
\begin{theorem}
\label{theorem:PLT}
Let $\{ a_{N,i} : N \in \NNp, i \in [N] \}$ be a family of operators with the following properties:
\begin{itemize}
    \item[ (i)] $a_{N, 1} \ldots, a_{N, N}$ are self-adjoint, \emph{independent}, and identically distributed random variables from a noncommutative probability space $(\mathcal{A}_N, \varphi_N)$ for each $N \in \NNp$,
    \item[(ii)] there exist constants $a, b \in \RR$ such that for each $i \in \NNp$ and any positive integer $n$, we have
\begin{equation*}
\lim \limits_{N \to \infty} N \varphi_N \left( a_{N, i}^n \right) = a^n b \text{.}
\end{equation*}
\end{itemize}
Then
\begin{equation*}
\frac{a_1 + \cdots + a_N}{\sqrt{N}} \weaklyConvergesTo a \cdot \gamma(\lambda) \text{,}
\end{equation*}
where $\gamma(\lambda)$ is a random variable with the Poisson distribution with intensity $\lambda > 0$.

\end{theorem}
In noncommutative probability, this result was established, among other things, for free~\cite[Theorem~4]{Sp}, Boolean~\cite[Theorem~3.5]{RSRW1997}, and monotone independence~\cite[Theorem~4.1]{Mur1}. There are other noncommutative analogs of the law of small numbers, for example~\cite{{BLS1996},{BW2001},{FL1999},{RLRS2006},{bmPoisson}}. Basically, in this type of Poisson limit theorem, one determines the weak limit (as $N \to \infty$) of the $N$th convolution power of the binomial distributions for a single trial (i.e., Bernoulli distributions) with the probability of success $\lfrac{\lambda}{N}$ (certain generalization in~\cite{bmPoisson}).

The combinatorics of limit moments of the Poisson distribution, which are expressed by means of certain class of partitions related to the notion of independence, is also studied. For instance, in the case of free independence, we have the class of noncrossing partitions (called admissible in~\cite{Sp}) and in monotone probability we have monotone noncrossing partitions (Muraki~\cite{Mur1} mentions a kind of noncrossing partitions, but for a more precise description, see~\cite[Theorem~8.1]{RLRS2006} for $m=1$).

The second approach to the Poisson limit theorem, which we take here, is motivated by~\cite[Theorem on p.~74]{Mey}. The construction using operator sum of type~\eqref{eq:GeneralSumPLT} was first established in the noncommutative case by Muraki~\cite{Mur0} for monotone independence (see also Crismale et al.~\cite{CGW2021} for weakly monotone Fock space), by Oussi and Wysocza\'{n}ski~\cite{LOJW2024} for bm-independence, and by Oussi for free independence~\cite{LO2023}. In this paper, we extend the construction to the setting of noncommutative V-monotone independence by constructing an appropriate model on the associated discrete V-monotone Fock space. This leads to a new family of Poisson-type distributions reflecting the combinatorial and algebraic features of V-monotone independence.

The paper is structured as follows. In Section~\ref{section:combinatorics}, we review some elementary information about noncrossing partitions of a finite set.

In Section~\ref{section:Fock}, we recall the definition of V-monotone Fock space and certain operators that act on it. Some properties of these operators are given, and their mixed moments are studied. These mixed moments are closely related to partitions, with a focus on V-monotonically labeled noncrossing partitions.
Section~\ref{section:Poisson} contains our main results, the V-monotone Poisson-type limit theorem on the V-monotone Fock space. We study the vacuum distribution of the operator $S_{N}(\lambda)$. Its exact and limit moments are given combinatorially by means of V-monotone labeled noncrossing partitions.
In Section~\ref{section:MGF}, we investigate  the moment generating function of the operator $S_{N}(\lambda)$ with its limit distribution as $N\to\infty$.
In  Section~\ref{section:limitMeasure}, using the Stieltjes inversion formula, we obtain the limit measure $\mu_{\lambda}$ of the nonsymmetric V-monotone operator $S_{N}(\lambda)$. Finally, using the \texttt{matplotlib.pyplot} module in Python~3, we generate some graphs illustrating the limit distribution for several values of intensity $\lambda$, and the position and weight of its only atom as functions of $\lambda$.

\section{Combinatorics of noncrossing partitions and V-monotone labelings}
\label{section:combinatorics}

In this section, we review some foundational combinatorial concepts related to noncrossing partitions, with a particular focus on V-monotone labelings. These tools play a key role in the results discussed later in the paper. For further details, we refer the reader to~\cite{ANRS2006}.

By a {\it partition} of the set $[n] \coloneqq \{ 1, 2, \ldots , n \}$, we mean a collection $\pi = \{ B_1, B_2, \ldots , B_k \}$ of its nonempty and disjoint subsets whose union is $[n]$. Elements of a partition are called {\it blocks}. The set of all partitions of $[n]$ will be denoted by $\part(n)$. A partition $\pi = \{ B_1, \ldots, B_k \}$ is called \emph{noncrossing} if there are no elements $a < c < b < d$ such that $\{ a, b \}\subset B_i$ and $\{c,d\}\subset B_j$ for some $i\neq j$. The set of all noncrossing partitions of $[n]$ will be denoted by $\noncr(n)$ (sometimes we write $\noncr(n, k)$ to denote its subset containing all partitions with exactly $k$ blocks). A block $B \in \pi$ is called \emph{a singleton} if it contains exactly one element, i.e., $\card{B} = 1$, where $\card{.}$ denotes cardinality.

The elements of a block $B \in \pi$ are called \emph{legs}. For a block $B \in \pi$ that is not a singleton, the rightmost leg of $B$ is defined as the greatest element in $B$, the leftmost leg is the least element, and the remaining legs are called middle legs. We denote by $\ml{B}$ the number of middle legs in the block $B$ and by $\ml{\pi}$ the total number of middle legs in the partition $\pi$.
A noncrossing partition $\pi = \{ B_1, \ldots, B_k \} \in \noncr(n, k)$ in which every block satisfies $\card{B_i} \geq 2$ for $i=1, \ldots, k$ will be denoted by $\noncrk{2+}(n)$. We write $\noncrk{2+}(n,k)$ when specifying that the partition has exactly $k$ blocks.

We say that a block $B$ of a partition $\pi$ is \emph{inner} with respect to a block $B' \in \pi$ and we write $B' \prec_{\pi} B$ if $l_{q-1} < \min B \leq \max B < l_{q}$ for some $q \in [p]$, where $B' = \{ l_0 < \cdots < l_p \}$ for $p \in \NN$. In that case, we say that $B'$ is \emph{outer} with respect to $B$. If a block $B \in \pi$ has no outer blocks, then it is called an \emph{outer block} in $\pi$.
There is a natural partial order on the blocks of $\pi$, induced by the above relation, i.e., $B' \preceq_{\pi} B$ if and only if $B' \prec_{\pi} B$ or $B' = B$.

\section{Discrete V-monotone Fock space and related operators}
\label{section:Fock}
In this section, we recall the definition of the discrete V-monotone Fock space~\cite[Definition~6.6]{AD2020}, which extends monotone and antimonotone Fock space constructions by incorporating ``V-shaped'' sequences. We also introduce the associated creation, annihilation and conservation operators acting on this space, and establish some of their fundamental properties that are essential for the analysis of the distribution of~\eqref{eq:GeneralSumPLT}.

Let $\mathcal{H}$ be a Hilbert space with an orthonormal basis $\{\Omega\} \cup \{e_i : i \in I\}$, where $\Omega$ is the vacuum vector. Let
\begin{equation*}
e_{\mathbf{i}} = \begin{cases}
\Omega & \text{for $\mathbf{i} = \emptyWord$,} \\
\vChar(\mathbf{i}) \cdot e_{i_1} \otimes \cdots \otimes e_{i_n} & \text{for $\mathbf{i} = (i_1, \ldots, i_n) \in I^n$, $n \in \NNp$,}
\end{cases}
\end{equation*}
where $\vChar \colon \fWords{I} \to \{ 0, 1 \}$ is the characteristic function of the set
\begin{equation*}
\vWords{}{I} \coloneqq \bigcup_{n=0}^\infty \vWords{n}{I} \text{,}
\end{equation*}
where $\vWords{}{I}$ is the set of all finite words over the alphabet $I$.

\begin{definition}[Discrete V-monotone Fock space]
The \emph{discrete V-monotone Fock space}, denoted by $\vDiscreteFockSpace$, is defined as the direct sum of one-dimensional Hilbert spaces
\begin{equation*}
\vDiscreteFockSpace \coloneqq \bigoplus_{\mathbf{i} \in \vWords{}{I}} \CC \, e_{\mathbf{i}} \text{,}
\end{equation*}
equipped with the canonical inner product. Recall that the set of indices $\vWords{n}{I}$ consists of all multi-indices $(i_1, \ldots, i_n) \in I$ satisfying the V-monotone ordering condition. The vacuum state $\varphi$ on $\vDiscreteFockSpace$ is given by
\[
\varphi(a) := \langle a \, \Omega, \Omega \rangle, \quad a \in \mathcal{B}(\vDiscreteFockSpace).
\]
\end{definition}

For each $i \in I$, we define the \emph{V-monotone (left) creation operator} $A_i^+$ associated with $e_i$ by the continuous linear extension of
\begin{equation*}
A_i^{+} \, e_{ \mathbf{i} } = e_{ (i, \mathbf{i}) } \text{,} \quad \mathbf{i} \in \vWords{}{I} \text{,}
\end{equation*}
where $(i, \mathbf{i})$ is the concatenation of the words $(i)$ and $\mathbf{i}$.

Its adjoint, the \emph{V-monotone (left) annihilation operator} $A_i^-$ associated with $e_i$, is given by the continuous linear extension of
\begin{equation*}
A^{-}_i \, e_{\mathbf{i}} = \begin{cases}
0 & \text{for $\mathbf{i} = \emptyWord$,} \\
\delta_{i, j} e_{\mathbf{i}_0} & \text{for $\mathbf{i} = (j, \mathbf{i}_0) \in \vWords{}{I}$,}
\end{cases}
\end{equation*}
where $\mathbf{i}_0 \in \vWords{}{I}$ and $\delta_{i, j}$ is the Kronecker delta. The \emph{V-monotone conservation operator} $A_i^{\circ}$ associated with $e_i$ is defined by its action on the vacuum and tensor basis vectors as the continuous linear extension of
\begin{equation*}
A^{\circ}_i \, e_{\mathbf{i}} = \begin{cases}
0 & \text{for $\mathbf{i} = \emptyWord$,} \\
\delta_{i, j} e_{\mathbf{i}} & \text{for $\mathbf{i} = (j, \mathbf{i}_0) \in \vWords{}{I}$.}
\end{cases}
\end{equation*}
The creation and annihilation operators are mutually adjoint, i.e., 
\[
(A_i^+)^* = A_i^-,
\]
and the conservation operator $A_i^{\circ} \coloneqq A^{+}_i A^{-}_i$ is self-adjoint. Using the above definitions of operators, one can establish some of their basic properties.
\begin{lemma}
\label{cr}
The creation, annihilation and conservation operators satisfy the following relations
\begin{itemize}
\item $A_{i}^{-}A_{j}^{+}=A_{i}^{\circ}A_{j}^{+}=A_{i}^{-}A_{j}^{\circ}=A_{i}^{\circ}A_{j}^{\circ}=0, i\neq j$,
\item $A_{i}^{+}A_{i}^{+}=A_{i}^{-}A_{i}^{-}=A_{i}^{\circ}A_{i}^{-}=A_{i}^{+}A_{i}^{\circ}=0$,
\item 

$(A_{i}^{\circ})^{m}A_{i}^{+}=A_{i}^{+}$ for any $m\in\mathbb{N}$.
\end{itemize}
\end{lemma}

Using the same idea as for the bm-case~\cite{LOJW2024} and~\cite[Proposition~6.7]{AD2020}, one can prove the following proposition.
\begin{proposition}
\label{proposition:Vindependence}
The algebras $\mathcal{A}_i \coloneqq \mathrm{alg} \{A_{i}^{+}, A_{i}^{-}\}$ generated by the creation and annihilation operators are V-monotone independent in the probability space $(\mathcal{A}, \varphi)$, where $\mathcal{A}$ is the $C^{*}$-algebra of all bounded operators in $\mathcal{F}_{\mathrm{vm}}^{d}(\mathcal{H})$.
\end{proposition}

In this paper, we want to study the V-monotone Poisson-type limit theorem for a sequence of operators~\eqref{eq:GeneralSumPLT}. In order to compute their exact and limit moments, we now study mixed moments of the form
\begin{equation*}
\label{eq:mixedMoment}
\varphi \left( A^{\varepsilon_1}_{i_1} \cdot \ldots \cdot A^{\varepsilon_n}_{i_n} \right) \text{,}
\end{equation*}
where $\mathbf{i} = \left( i_1, \ldots, i_n \right) \in I^n$ and $\pmb{\varepsilon} = \left( \varepsilon_1, \ldots, \varepsilon_n \right) \in \{ -, 0, + \}^n$, $n \in \NNp$.

\def\riordan#1{\mathcal{R}_{#1}}
\def\seps#1{\sgn(\varepsilon_{#1})}
First, we provide a necessary condition for sequences $\pmb{\varepsilon}$ for which a mixed moment of the above form is nonzero.
\begin{definition}
For $n \in \NN$, we denote by $\riordan{n}$ the set of all sequences $(\varepsilon_1, \ldots, \varepsilon_n) \in \{ -, \circ, + \}^n$ that meet the following conditions:
\begin{itemize}
    \item[(i)]\label{riordan1} $\seps{1} + \cdots + \seps{n} = 0$,
    \item[(ii)]\label{riordan2} $\seps{1} + \cdots + \seps{l} \leq 0$, for all $l \in [n]$,

    \item[(iii)]\label{riordan3} $\varepsilon_{l} = -$ whenever $\seps{1} + \cdots + \seps{l-1} = 0$, for all $l \in [n]$ (for $l=1$, the empty sum equals $0$ by convention),
\end{itemize}
where
\begin{equation*}
\seps{l} = \begin{cases}
-1 & \text{if $\varepsilon_l = -$,} \\
0 & \text{if $\varepsilon_l = \circ$,} \\
1 & \text{if $\varepsilon_l = +$,}
\end{cases}
\end{equation*}
for $l \in [n]$. For $n=0$, we put $\riordan{0} = \{ \emptyWord \}$, i.e., the singleton of the empty word.
\end{definition}

\begin{proposition}
\label{proposition:Riordan}
For any $n \in \NN$, if $\varphi \left( A^{\varepsilon_1}_{i_1} \cdot \ldots \cdot A^{\varepsilon_n}_{i_n} \right) \neq 0$ then $(\varepsilon_1, \ldots, \varepsilon_n) \in \riordan{n}$.
\end{proposition}

\begin{proof}
Let $\mathcal{H}_k$, $k \in \ZZ$, be the closed linear subspace of $\vDiscreteFockSpace$ spanned by all tensors of length $k$ (we put $\mathcal{H}_0 = \CC \, \Omega$ and $\mathcal{H}_k = \{ 0 \}$ for all negative $k$). Consider an operator $b = A^{\varepsilon_1}_{i_1} \cdot \ldots \cdot A^{\varepsilon_n}_{i_n}$. Let us first observe that the image of $\mathcal{H}_k$ under $b$ is a subspace of
\begin{equation*}
\mathcal{H}_{k + \seps{1} + \cdots + \seps{n}} \text{,}
\end{equation*}
which can be easily proved by induction on $n$ using only the definitions of creation, conservation, and annihilation operators.

Assume that $(\varepsilon_1, \ldots, \varepsilon_n) \notin \riordan{n}$. We consider three cases. If Condition~(i) does not hold, then $b \, \Omega \perp \Omega$ and the moment vanishes. If Condition~(i) holds but Condition~(ii) fails, then $r = \seps{l+1} + \cdots + \seps{n} < 0$ for some $l \in [0, n-1]$ and $A^{\varepsilon_{l+1}}_{i_{l+1}} \cdots A^{\varepsilon_n}_{i_n} \, \Omega \in \mathcal{H}_r = \{ 0 \}$. This means $b \, \Omega = 0$ and the moment also vanishes. If Conditions~(i) and~(ii) hold, but Condition~(iii) fails, then $\seps{1} + \cdots + \seps{l-1} = 0$ and $\varepsilon_l \neq -$ for some $l \in [n]$, which means $\varepsilon_l = \circ$. Since $\seps{l} + \cdots + \seps{n} = 0$, we have
\begin{equation*}
A^{\varepsilon_l}_{i_l} A^{\varepsilon_{l+1}}_{i_{l+1}} \ldots A^{\varepsilon_1}_{i_1} \, \Omega = A^{\varepsilon_l}_{i_l} \, (\alpha \Omega) = \alpha A^{\circ}_{i_l} \, \Omega = 0 \text{,}
\end{equation*}
where $\alpha \in \CC$, and $b \, \Omega = 0$. Then the mixed moment vanishes, which proves the proposition.
\end{proof}

\begin{remark}
The set $\riordan{n}$ corresponds to a certain subset of the set of all Motzkin paths of length $n$, namely the set of all so-called Riordan paths of length $n$ (see, for instance,~\cite[Section~2.3]{CGW2021}).

\end{remark}

\begin{proposition}
\label{proposition:bijRiordan}
There is a natural bijection between sets $\riordan{n}$ and $\noncrk{2+}(n)$.
\end{proposition}
\begin{proof}
For $n = 0$, we have $\emptyWord \mapsto \emptyset$. For $n=1$, we have the empty bijection between $\riordan{1} = \emptyset$ and $\noncrk{2+}(1) = \emptyset$. Let $(\varepsilon_1, \ldots, \varepsilon_n) \in \riordan{n}$ and let $m \in [n]$ be the smallest number such that $\seps{1} + \cdots + \seps{m} = 0$. Note that $\varepsilon_1 = -$ and $\varepsilon_m \neq \circ$ from~(iii). Then $\varepsilon_m = +$ from~(iii) and $\varepsilon \neq -$ from~(ii). For that reason, the sequence $(\varepsilon_2, \ldots, \varepsilon_{m-1})$ (possibly empty) satisfies the properties~(i) and~(ii), but not necessarily~(iii). Let $\{ l_1 < \cdots < l_{p-1} \} \subseteq [2, m-1]$ for some $p \in [m]$ ($\emptyset$ for $p=1$) be the set of all integers $l$ for which $\varepsilon_l = \circ$ although $\seps{2} + \cdots + \seps{l-1} = 0$ (by convention, an empty sum equals $0$). We put $l_0 = 1$ and $l_p = m$, and define the set $B = \{ l_0 < l_1 < \cdots < l_{p-1} < l_p \}$. For each $q \in [p]$, we have $(\varepsilon_{l_{q-1} + 1}, \ldots, \varepsilon_{l_q - 1}) \in \riordan{l_q - 1 - l_{q-1}}$ and $(\varepsilon_{m+1}, \ldots, \varepsilon_{n}) \in \riordan{n-m}$. These sequences correspond to unique partitions $\pi'_q \in \noncrk{2+}(l_q - 1 - l_{q-1})$, $q \in [p]$, and $\pi'' \in \noncrk{2+}(n-m)$, respectively. Let $\pi = \ncPartitionRecurrenceD{\pi'_1}{\cdots}{\pi'_p}{\pi''}$ be such that $B \in \pi$. We then assign the partition $\pi \in \noncrk{2+}(n)$ to the sequence $(\varepsilon_1, \ldots, \varepsilon_n)$. It is clear from induction that this correspondence is a bijection between $\riordan{n}$ and $\noncrk{2+}(n)$.
\end{proof}

\begin{remark}
There is a canonical bijection between the sets $\noncrk{2+}(n)$ and $\mathcal{NC}_{2}^{1, i}(n)$, where the latter denotes the set of all noncrossing partitions $\pi \in \noncr(n)$ with blocks of cardinality at most $2$ such that all singleton blocks must not be outer. This kind of noncrossing partition was used in an analogous context in~\cite{LO2023, LOJW2024}. Indeed, for $\pi \in \noncrk{2+}(n)$, it suffices to convert all middle legs into singletons. For the inverse, we take $\pi \in \noncr_{2}^{1, i}(n)$ and connect all singleton blocks with their nearest outer block.

\end{remark}

Let us first introduce some additional tools and preliminary results.
\begin{definition}
By a \emph{labeling} of a partition $\pi$, we mean a function $\mathrm{L} \colon \pi \mapsto I$. The pair $(\pi, \mathrm{L})$ is called a \emph{labeled partition}. For such a partition, consider a set
\begin{equation*}
\partTree{(\pi, \mathrm{L})} \coloneqq \left\{ \left( \mathrm{L}(B_r), \ldots, \mathrm{L}(B_1) \right) : \{ B_1 \prec_{\pi} \cdots \prec_{\pi} B_r \} \text{ is a maximal chain in } (\pi, \prec_{\pi}) \right\} \text{.}
\end{equation*}
We say that $(\pi, \mathrm{L})$ is \emph{V-monotonically labeled} if $\partTree{(\pi, \mathrm{L})} \subseteq \vWords{}{I}$.
\end{definition}

\begin{example}
The labeled partition $(\pi, \mathrm{L}_1)$ in Fig.~\ref{imageadaptedseq} (left) is V-monotonically labeled, while the labeled partition $(\pi, \mathrm{L}_2)$ (right) is not. Indeed
\begin{equation*}
\partTree{(\pi, \mathrm{L}_1)} = \left\{ (6, 5, 7, 8), (6, 5, 6) \right\} \subseteq \vWords{}{ \NNp } \text{,}
\end{equation*}
whereas $(4, 5, 7, 2)$ is not ``V-shaped''.
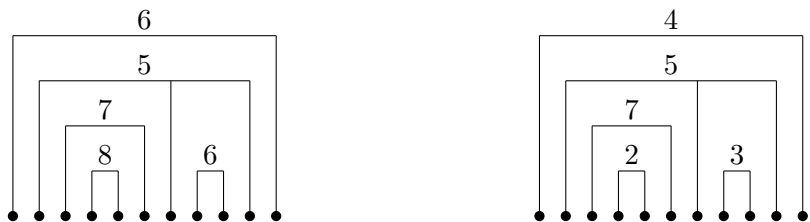
\begin{figure}[ht]
\centering



\begin{tikzpicture}
    \tikzstyle{leg} = [circle, draw=black, fill=black!100, text=black!100, thin, inner sep=0pt, minimum size=2.0]

    \pgfmathsetmacro {\dx}{0.35}
    \pgfmathsetmacro {\dy}{0.60}
    
    \pgfmathsetmacro {\x}{0*\dx}
    \pgfmathsetmacro {\y}{-6*\dy}

    \draw[-] (\x+0*\dx,\y+0*\dy) -- (\x+0*\dx,\y+4*\dy);
    \node () at (\x+0*\dx,\y+0*\dy) [leg] {.};
    \draw[-] (\x+10*\dx,\y+0*\dy) -- (\x+10*\dx,\y+4*\dy);
    \node () at (\x+10*\dx,\y+0*\dy) [leg] {.};
    \draw[-] (\x+0*\dx,\y+4*\dy) -- (\x+10*\dx,\y+4*\dy);
    \node () at (\x+5*\dx,\y+4*\dy) [label={[label distance = -2mm]above:\small{$6$}}] {};

    \draw[-] (\x+1*\dx,\y+0*\dy) -- (\x+1*\dx,\y+3*\dy);
    \node () at (\x+1*\dx,\y+0*\dy) [leg] {.};
     \draw[-] (\x+6*\dx,\y+0*\dy) -- (\x+6*\dx,\y+3*\dy);
    \node () at (\x+6*\dx,\y+0*\dy) [leg] {.};
    \draw[-] (\x+9*\dx,\y+0*\dy) -- (\x+9*\dx,\y+3*\dy);
    \node () at (\x+9*\dx,\y+0*\dy) [leg] {.};
     \draw[-] (\x+1*\dx,\y+3*\dy) -- (\x+9*\dx,\y+3*\dy);
    \node () at (\x+5*\dx,\y+3*\dy) [label={[label distance = -2mm]above:\small{$5$}}] {};
    
    \draw[-] (\x+2*\dx,\y+0*\dy) -- (\x+2*\dx,\y+2*\dy);
    \node () at (\x+2*\dx,\y+0*\dy) [leg] {.};
    \draw[-] (\x+5*\dx,\y+0*\dy) -- (\x+5*\dx,\y+2*\dy);
    \node () at (\x+5*\dx,\y+0*\dy) [leg] {.};
    \draw[-] (\x+2*\dx,\y+2*\dy) -- (\x+5*\dx,\y+2*\dy);
    \node () at (\x+3.5*\dx,\y+2*\dy) [label={[label distance = -2mm]above:\small{$7$}}] {};
    
    \draw[-] (\x+3*\dx,\y+0*\dy) -- (\x+3*\dx,\y+1*\dy);
    \node () at (\x+3*\dx,\y+0*\dy) [leg] {.};
    \draw[-] (\x+4*\dx,\y+0*\dy) -- (\x+4*\dx,\y+1*\dy);
    \node () at (\x+4*\dx,\y+0*\dy) [leg] {.};
    \draw[-] (\x+3*\dx,\y+1*\dy) -- (\x+4*\dx,\y+1*\dy);
    \node () at (\x+3.5*\dx,\y+1*\dy) [label={[label distance = -2mm]above:\small{$8$}}] {};
    
    \draw[-] (\x+7*\dx,\y+0*\dy) -- (\x+7*\dx,\y+1*\dy);
    \node () at (\x+7*\dx,\y+0*\dy) [leg] {.};
    \draw[-] (\x+8*\dx,\y+0*\dy) -- (\x+8*\dx,\y+1*\dy);
    \node () at (\x+8*\dx,\y+0*\dy) [leg] {.};
    \draw[-] (\x+7*\dx,\y+1*\dy) -- (\x+8*\dx,\y+1*\dy);
    \node () at (\x+7.5*\dx,\y+1*\dy) [label={[label distance = -2mm]above:\small{$6$}}] {};
    
    \pgfmathsetmacro {\x}{20*\dx}
    \pgfmathsetmacro {\y}{-6*\dy}

    \draw[-] (\x+0*\dx,\y+0*\dy) -- (\x+0*\dx,\y+4*\dy);
    \node () at (\x+0*\dx,\y+0*\dy) [leg] {.};
    \draw[-] (\x+10*\dx,\y+0*\dy) -- (\x+10*\dx,\y+4*\dy);
    \node () at (\x+10*\dx,\y+0*\dy) [leg] {.};
    \draw[-] (\x+0*\dx,\y+4*\dy) -- (\x+10*\dx,\y+4*\dy);
    \node () at (\x+5*\dx,\y+4*\dy) [label={[label distance = -2mm]above:\small{$4$}}] {};

    \draw[-] (\x+1*\dx,\y+0*\dy) -- (\x+1*\dx,\y+3*\dy);
    \node () at (\x+1*\dx,\y+0*\dy) [leg] {.};
     \draw[-] (\x+6*\dx,\y+0*\dy) -- (\x+6*\dx,\y+3*\dy);
    \node () at (\x+6*\dx,\y+0*\dy) [leg] {.};
    \draw[-] (\x+9*\dx,\y+0*\dy) -- (\x+9*\dx,\y+3*\dy);
    \node () at (\x+9*\dx,\y+0*\dy) [leg] {.};
     \draw[-] (\x+1*\dx,\y+3*\dy) -- (\x+9*\dx,\y+3*\dy);
    \node () at (\x+5*\dx,\y+3*\dy) [label={[label distance = -2mm]above:\small{$5$}}] {};
    
    \draw[-] (\x+2*\dx,\y+0*\dy) -- (\x+2*\dx,\y+2*\dy);
    \node () at (\x+2*\dx,\y+0*\dy) [leg] {.};
    \draw[-] (\x+5*\dx,\y+0*\dy) -- (\x+5*\dx,\y+2*\dy);
    \node () at (\x+5*\dx,\y+0*\dy) [leg] {.};
    \draw[-] (\x+2*\dx,\y+2*\dy) -- (\x+5*\dx,\y+2*\dy);
    \node () at (\x+3.5*\dx,\y+2*\dy) [label={[label distance = -2mm]above:\small{$7$}}] {};
    
    \draw[-] (\x+3*\dx,\y+0*\dy) -- (\x+3*\dx,\y+1*\dy);
    \node () at (\x+3*\dx,\y+0*\dy) [leg] {.};
    \draw[-] (\x+4*\dx,\y+0*\dy) -- (\x+4*\dx,\y+1*\dy);
    \node () at (\x+4*\dx,\y+0*\dy) [leg] {.};
    \draw[-] (\x+3*\dx,\y+1*\dy) -- (\x+4*\dx,\y+1*\dy);
    \node () at (\x+3.5*\dx,\y+1*\dy) [label={[label distance = -2mm]above:\small{$2$}}] {};
    
    \draw[-] (\x+7*\dx,\y+0*\dy) -- (\x+7*\dx,\y+1*\dy);
    \node () at (\x+7*\dx,\y+0*\dy) [leg] {.};
    \draw[-] (\x+8*\dx,\y+0*\dy) -- (\x+8*\dx,\y+1*\dy);
    \node () at (\x+8*\dx,\y+0*\dy) [leg] {.};
    \draw[-] (\x+7*\dx,\y+1*\dy) -- (\x+8*\dx,\y+1*\dy);
    \node () at (\x+7.5*\dx,\y+1*\dy) [label={[label distance = -2mm]above:\small{$3$}}] {};
    
\end{tikzpicture}
\caption{Labeled noncrossing partitions.}
\label{imageadaptedseq}
\end{figure}
\end{example}

\begin{definition}
\label{definition:partitionOperatorDiscrete}
For $\pi \in \noncrk{2+}(n)$ and $\mathbf{i} = (i_1, \ldots, i_n) \in I^n$, let
\begin{equation*}
\partOpDiscA{\pi}{\mathbf{i}} = b_1 \cdot \ldots \cdot b_n \text{,}
\end{equation*}
where, for $k \in [n]$, we put
\begin{equation*}
b_k =
\begin{cases}
A_{i_k}^{+} & \text{if $k = \max B$ for some $B \in \pi$,} \\
A_{i_k}^{-} & \text{if $k = \min B$ for some $B \in \pi$,} \\
A_{i_k}^{\circ} & \text{otherwise.}
\end{cases}
\end{equation*}
We adopt the convention that the empty product equals $\id$ (the identity operator) for $\pi = \emptyset$ and $\mathbf{i} = \emptyWord$.

That is, on the $k$th place in $b_1 \cdot \ldots \cdot b_n$, we have

\begin{itemize}
\item a creation operator if $k$ is the rightmost leg of some $B \in \pi$, 
\item an annihilation operator if $k$ is the leftmost leg of some $B \in \pi$, 
\item a conservation operator $k$ is a middle leg of some $B \in \pi$.
\end{itemize}

\def\j{\mathbf{j}}
For a labeled partition $(\pi, \mathrm{L})$, we put $\partOpDiscB{\pi}{\mathrm{L}} = \partOpDiscA{\pi}{\j}$, where the sequence $\j = \left( j_1, \ldots, j_n \right)$ is such that $j_l = \mathrm{L}(B)$ with $B$ being a unique block that contains $l \in [n]$.

\let\j\undefined
\end{definition}

\begin{example}
To illustrate the above definition, consider a partition $\pi \in \noncrk{2+}(12)$ and an operator $\partOpDiscA{\pi}{\mathbf{i}}$, presented in Fig.~\ref{figure:partition_operator}, for $\mathbf{i} \in \{ \mathbf{i}_1, \mathbf{i}_2, \mathbf{i}_3 \} \subseteq I^{12}$, where
\begin{itemize}
    \item $\mathbf{i}_1 = (3,2,1,3,4,3,3,3,3,2,4,4)$,
    \item $\mathbf{i}_2 = (2,1,1,2,3,4,3,3,3,4,3,3)$,
    \item $\mathbf{i}_3 = (2,1,1,2,4,3,5,5,5,3,4,4)$.
\end{itemize}
There is no labeling $\mathrm{L}_1$ such that $\partOpDiscA{\pi}{\mathbf{i}_1} = \partOpDiscB{\pi}{\mathrm{L}_1}$. However, such labelings exist:
\begin{itemize}
    \item $\mathrm{L}_2$, for $\mathbf{i}_2$, given by $B_1 \mapsto 2, B_2 \mapsto 1, B_3 \mapsto 3, B_4 \mapsto 4, B_5 \mapsto 3$,
    \item $\mathrm{L}_3$, for $\mathbf{i}_3$, given by $B_1 \mapsto 2, B_2 \mapsto 1, B_3 \mapsto 4, B_4 \mapsto 3, B_5 \mapsto 5$.
\end{itemize}
The partition $(\pi, \mathrm{L}_3)$ is V-monotonically labeled whereas $(\pi, \mathrm{L}_2)$ is not.

On the other hand, there are no noncrossing partitions corresponding to the operators $b_1 = A^{-}_1 A^{+}_1 A^{+}_2$, $b_2 = A^{-}_2 A^{\circ}_2 A^{+}_2   A^{+}_1 A^{-}_1   A^{-}_4 A^{-}_3 A^{\circ}_3 A^{+}_3 A^{+}_4$, and $b_3 = A^{-}_2 A^{-}_1 A^{+}_1 A^{+}_2 A^{\circ}_4 A^{-}_2 A^{+}_2$.
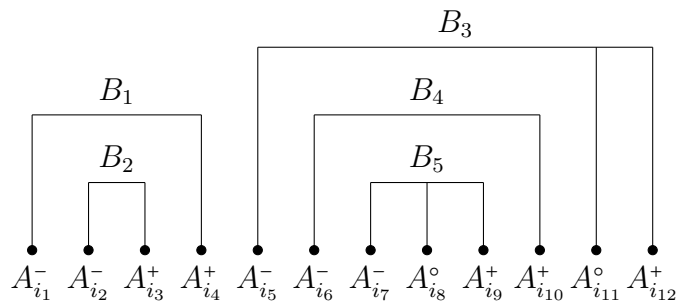
\begin{figure}[ht]
\centering



\begin{tikzpicture}
    \tikzstyle{leg} = [circle, draw=black, fill=black!100, text=black!100, thin, inner sep=0pt, minimum size=2.0]

    \pgfmathsetmacro {\dx}{0.75}
    \pgfmathsetmacro {\dy}{0.9}

    \pgfmathsetmacro {\x}{13*\dx}
    \pgfmathsetmacro {\y}{-6*\dy}

    \draw[-] (\x+0*\dx,\y+0*\dy) -- (\x+0*\dx,\y+2*\dy);
    \node () at (\x+0*\dx,\y+0*\dy) [leg, label={below:$A^{-}_{i_1}$}] {.};
    \draw[-] (\x+3*\dx,\y+0*\dy) -- (\x+3*\dx,\y+2*\dy);
    \node () at (\x+3*\dx,\y+0*\dy) [leg, label={below:$A^{+}_{i_4}$}] {.};
    \draw[-] (\x+0*\dx,\y+2*\dy) -- (\x+3*\dx,\y+2*\dy);
    \node () at (\x+1.5*\dx,\y+2*\dy)  [label={[label distance = -1.5mm]above:$B_1$}] {};

    \draw[-] (\x+1*\dx,\y+0*\dy) -- (\x+1*\dx,\y+1*\dy);
    \node () at (\x+1*\dx,\y+0*\dy) [leg, label={below:$A^{-}_{i_2}$}] {.};
    \draw[-] (\x+2*\dx,\y+0*\dy) -- (\x+2*\dx,\y+1*\dy);
    \node () at (\x+2*\dx,\y+0*\dy) [leg, label={below:$A^{+}_{i_3}$}] {.};
    \draw[-] (\x+1*\dx,\y+1*\dy) -- (\x+2*\dx,\y+1*\dy);
    \node () at (\x+1.5*\dx,\y+1*\dy)  [label={[label distance = -1.5mm]above:$B_2$}] {};

    \draw[-] (\x+4*\dx,\y+0*\dy) -- (\x+4*\dx,\y+3*\dy);
    \node () at (\x+4*\dx,\y+0*\dy) [leg, label={below:$A^{-}_{i_5}$}] {.};
    \draw[-] (\x+10*\dx,\y+0*\dy) -- (\x+10*\dx,\y+3*\dy);
    \node () at (\x+10*\dx,\y+0*\dy) [leg, label={below:$A^{\circ}_{i_{11}}$}] {.};
    \draw[-] (\x+11*\dx,\y+0*\dy) -- (\x+11*\dx,\y+3*\dy);
    \node () at (\x+11*\dx,\y+0*\dy) [leg, label={below:$A^{+}_{i_{12}}$}] {.};
    \draw[-] (\x+4*\dx,\y+3*\dy) -- (\x+11*\dx,\y+3*\dy);
    \node () at (\x+7.5*\dx,\y+3*\dy)  [label={[label distance = -1.5mm]above:$B_3$}] {};

    \draw[-] (\x+5*\dx,\y+0*\dy) -- (\x+5*\dx,\y+2*\dy);
    \node () at (\x+5*\dx,\y+0*\dy) [leg, label={below:$A^{-}_{i_6}$}] {.};
    \draw[-] (\x+9*\dx,\y+0*\dy) -- (\x+9*\dx,\y+2*\dy);
    \node () at (\x+9*\dx,\y+0*\dy) [leg, label={below:$A^{+}_{i_{10}}$}] {.};
    \draw[-] (\x+5*\dx,\y+2*\dy) -- (\x+9*\dx,\y+2*\dy);
    \node () at (\x+7*\dx,\y+2*\dy)  [label={[label distance = -1.5mm]above:$B_4$}] {};

    \draw[-] (\x+6*\dx,\y+0*\dy) -- (\x+6*\dx,\y+1*\dy);
    \node () at (\x+6*\dx,\y+0*\dy) [leg, label={below:$A^{-}_{i_7}$}] {.};
    \draw[-] (\x+7*\dx,\y+0*\dy) -- (\x+7*\dx,\y+1*\dy);
    \node () at (\x+7*\dx,\y+0*\dy) [leg, label={below:$A^{\circ}_{i_{8}}$}] {.};
    \draw[-] (\x+8*\dx,\y+0*\dy) -- (\x+8*\dx,\y+1*\dy);
    \node () at (\x+8*\dx,\y+0*\dy) [leg, label={below:$A^{+}_{i_{9}}$}] {.};
    \draw[-] (\x+6*\dx,\y+1*\dy) -- (\x+8*\dx,\y+1*\dy);
    \node () at (\x+7*\dx,\y+1*\dy)  [label={[label distance = -1.5mm]above:$B_5$}] {};

\end{tikzpicture}
\caption{Operator $\partOpDiscA{\pi}{\mathbf{i}} = A^{-}_{i_1} A^{-}_{i_2} A^{+}_{i_3} A^{+}_{i_4} A^{-}_{i_5} A^{-}_{i_6} A^{-}_{i_7} A^{\circ}_{i_8} A^{+}_{i_9} A^{+}_{i_{10}} A^{\circ}_{i_{11}} A^{+}_{i_{12}}$ associated to $\pi \in \noncrk{2+}(12)$.}
\label{figure:partition_operator}
\end{figure}
\end{example}
\begin{lemma}
\label{lemLS}
Let $b = A^{\varepsilon_1}_{i_1} \cdot \ldots \cdot A^{\varepsilon_n}_{i_n}$ for $\mathbf{i} = \left( i_1, \ldots, i_n \right) \in I^n$ and $\varepsilon = \left( \varepsilon_1, \ldots, \varepsilon_n \right) \in \{ -, 0, + \}^n$, $n \in \NNp$. If there exists a noncrossing V-monotonically labeled partition $(\pi, \mathrm{L})$ such that $b = \partOpDiscB{\pi}{\mathrm{L}}$, then $\varphi(b) = 1$. Otherwise, the mixed moment equals $0$.
\end{lemma}

\def\const#1#2#3{ C^{#1}_{#2, #3} }
\def\constL#1#2#3{ C_{#1}^{(#2, #3)} }
\def\operator#1#2{ A^{#1}_{ i_{l_{#2}} } }

\def\vCharB{ \vChar(\mathbf{j}') }
\def\constB#1{ \const{\pi'_{#1}}{\mathbf{i}'_{#1}}{\mathbf{j}'} }

\def\constC{ \const{\pi''}{\mathbf{i}''}{\mathbf{j}} }

\def\constDelta#1{ \delta_{i_{l_{#1}}, i_{l_p}} }

\begin{proof}
We focus only on the case where there exists $\pi \in \noncrk{2+}(n)$ such that $b = \partOpDiscA{\pi}{\mathbf{i}}$ since otherwise, according to Propositions~\ref{proposition:Riordan} and~\ref{proposition:bijRiordan}, we have $\varphi(b) = 0$ (see also the proof of Lemma~5.1 from~\cite{RLRS2008}). To establish the lemma, we show that for every $\mathbf{j} \in \fWords{I}$, we have
\begin{equation*}
\partOpDiscA{\pi}{\mathbf{i}} \, e_{\mathbf{j}} =  \const{\pi}{\mathbf{i}}{\mathbf{j}} \cdot e_{\mathbf{j}} \text{,}
\end{equation*}
for some $\const{\pi}{\mathbf{i}}{\mathbf{j}} \in \{ 0, 1 \}$, which will be specified later. We use induction on $n$. The base case is obvious. In the step case, for some $p \in \NNp$, we have
\begin{equation*}
\pi = \ncPartitionRecurrenceD{\pi'_1}{\cdots}{\pi'_p}{\pi''}
\quad \text{ and } \quad
\partOpDiscA{\pi}{\mathbf{i}} = \operator{-}{0} \cdot \partOpDiscA{\pi'_1}{\mathbf{i}'_1} \cdot \operator{0}{1} \cdot \ldots \cdot \operator{0}{p-1} \cdot \partOpDiscA{\pi'_p}{\mathbf{i}'_p} \cdot \operator{+}{p} \cdot \partOpDiscA{\pi''}{\mathbf{i}''} \text{,}
\end{equation*}
where
\begin{itemize}
    \item $\{ 1= l_0 < l_1 < \cdots < l_p \} \in \pi$,
    \item $\pi'_q \in \noncrk{2+}([l_{q-1}+1, l_q-1])$, for all $q \in [p]$,
    \item $\mathbf{i}'_q = \left( i_{l_{q-1}+1}, \ldots, i_{l_q-1} \right)$, for all $q \in [p]$,
    \item $\pi'' \in \noncrk{2+}([l_p+1, n])$,
    \item $\mathbf{i}'' = \left(i_{l_p+1}, \ldots, i_n \right)$.
\end{itemize}
In this case
\begin{equation*}
\begin{split}
\partOpDiscA{\pi}{\mathbf{i}} \, e_{\mathbf{j}}
& = \left( \operator{-}{0} \cdot \partOpDiscA{\pi'_1}{\mathbf{i}'_1} \cdot \operator{0}{1} \cdot \ldots \cdot \operator{0}{p-1} \cdot \partOpDiscA{\pi'_p}{\mathbf{i}'_p} \right) \, \vCharB \cdot \constC \cdot e_{\mathbf{j}'} \\
& = \operator{-}{0} \, \left( \constB{1} \cdot \constDelta{1} \cdot \ldots \cdot \constDelta{p-1} \cdot \constB{p} \cdot \vCharB \cdot \constC \cdot e_{\mathbf{j}'} \right) \\
& = \constDelta{0} \cdot \constB{1} \cdot \constDelta{1} \cdot \ldots \cdot \constDelta{p-1} \cdot \constB{p} \cdot \vCharB \cdot \constC \cdot e_{\mathbf{j}} \text{,}
\end{split}
\end{equation*}
where $\mathbf{j}' = \left( i_{l_p}, \mathbf{j} \right)$ is the concatenation of $(i_{l_p})$ and $\mathbf{j}$.

It follows that
\begin{equation*}
\const{\pi}{\mathbf{i}}{\mathbf{j}} = \constDelta{0} \cdot \constB{1} \cdot \constDelta{1} \cdot \ldots \cdot \constDelta{p-1} \cdot \constB{p} \cdot \vCharB \cdot \constC \text{.}
\end{equation*}
Using induction, one can check that
\begin{equation*}
\const{\pi}{\mathbf{i}}{\mathbf{j}} = \prod_{\tilde{\mathbf{i}} \in \partTree{}} \vChar(\tilde{\mathbf{i}}, \mathbf{j}) \cdot \prod_{B \in \pi} \left( \prod_{l, l' \in B} \delta_{i_{l}, i_{l'}} \right) \text{,}
\end{equation*}
where
\begin{equation*}
\partTree{} \coloneqq \left\{ \left( i_{\max B_r}, \ldots, i_{\max B_1} \right) : \{ B_1 \prec_{\pi} \cdots \prec_{\pi} B_r \} \text{ is a maximal chain in } (\pi, \prec_{\pi}) \right\} \text{.}
\end{equation*}

That means $b\, \Omega = \partOpDiscA{\pi}{\mathbf{i}} \, \Omega = \Omega$ if and only if there exists a V-monotonically labeled partition $(\pi, \mathrm{L})$ such that $b = \partOpDiscB{\pi}{\mathrm{L}}$ and $b \, \Omega$ is the zero vector otherwise, which leads to the desired conclusion.
\end{proof}
\let\constDelta\undefined
\let\constC\undefined
\let\constB\undefined
\let\vCharB\undefined
\let\opearator\undefined
\let\constL\undefined
\let\const\undefined

\begin{corollary}
\label{corollary:labeledPartOpInterval}
We have
\begin{equation*}
\varphi \left( \partOpDiscB{\pi}{\mathrm{L}} \right) = 1 \text{}
\end{equation*}
for any labeled partition $(\pi, \mathrm{L})$ such that $\pi \in \part(n)$ is an \emph{interval partition}, i.e., each block of $\pi$ is an interval block. A block $B \in \pi$ is called an \emph{interval block} if it consists of consecutive integers, that is, $B = [l_1, l_2] = \{ l_1, l_1+1, \ldots, l_2 \}$ for some $1 \leq l_1 \leq l_2 \leq n$.

\end{corollary}

\section{V-monotone Poisson type limit theorem}
\label{section:Poisson}
In this section, we study the V-monotone Poisson-type limit theorem for the sequence of operators~\eqref{eq:GeneralSumPLT}, that is,
\begin{equation*}
S_{N}(\lambda) \coloneqq \sum_{i=1}^{N} \left( \frac{G_i}{\sqrt{N}} + \lambda A_{i}^{\circ} \right) \text{,} \qquad N \in \NNp \text{,} \quad \lambda \in \RR \text{,}
\end{equation*}
where $G_{i}:=A_{i}^{+}+A_{i}^{-}$ is a Gaussian operator. We investigate its exact and limit moments
\begin{equation*}
\momentExact{n}{\lambda}{N} \coloneqq \varphi \left( S_N^n(\lambda) \right) \qquad \text{and} \qquad \momentLimit{n}{\lambda} \coloneqq \lim_{N \to \infty} \momentExact{n}{\lambda}{N} \text{,}
\end{equation*}
respectively, for $n \in \NN$.

\begin{remark}
The limit measure for $\lambda = 0$ is the standard V-monotone Gaussian distribution (see~\cite[Corollary~6.8]{AD2020}). The exact form of this measure was obtained in~\cite{AD2025}.
\end{remark}

The moments of $S_{N}(\lambda)$ can be expressed combinatorially. For $n \in \NN$ and $N \in \NNp$, let $\vlNoncrk{2+}{N}(n)$ be the set of all V-monotonically labeled partition $(\pi, \mathrm{L})$ such that $\pi \in \noncrk{2+}(n)$ and $\mathrm{L} \colon \pi \to [N]$. We denote by $\vlNoncrk{2+}{N}(n, k)$ the set of all $(\pi, \mathrm{L}) \in \vlNoncrk{2+}{N}(n)$ such that $\pi$ has exactly $k \in \NN$ blocks. Moreover, we put $\vlNoncrk{2+}{N}(0) = \vlNoncrk{2+}{N}(0, 0) = \{ \emptyset \}$. Note that $\vlNoncrk{2+}{N}(n, k) = \emptyset$ if $n > k = 0$, $n < 2k$, or $N = 0$ (with $(n, k) \neq (0, 0)$).

\begin{example}
\label{example:VL_two-plus}
The set $\noncrk{2+}(6)$ consists of $15$ elements, presented in Fig.~\ref{picture:NC_two-plus}. We have
\begin{itemize}
    \item $\noncrk{2+}(6, 1) = \{ \pi_{15} \}$,
    \item $\noncrk{2+}(6, 2) = \{ \pi_3, \pi_7, \ldots, \pi_{14} \}$,
    \item $\noncrk{2+}(6, 3) = \noncrk{2}(6) = \{ \pi_1, \pi_2, \pi_4, \pi_5, \pi_6 \}$.
\end{itemize}
Let us compute the cardinality of $\vlNoncrk{2+}{N}(6, k)$, $N \in \NNp$, for $k = 1, 2, 3$. Of course, for the remaining values of $k$, the cardinality is $0$. We have
\begin{itemize}
    \item $\card{ \vlNoncrk{2+}{N}(6, 1) } = N$,
    \item $\card{ \vlNoncrk{2+}{N}(6, 2) } = 9 N^2 - 6 N$,
    \item $\card{ \vlNoncrk{2+}{N}(6, 3) } = \lfrac{14 N^3}{3} - \lfrac{11 N^2}{2} + \lfrac{11 N}{6}$.
\end{itemize}
In fact,
\begin{equation*}
\card{ \vlNoncrk{2+}{N}(6, 2) } = N^2 + N (N-1) + N^2 + 2 N (N-1) + N^2 + 3 N (N-1) = 9 N^2 - 6 N \text{}
\end{equation*}
and
\begin{equation*}
\card{ \vlNoncrk{2+}{N}(6, 3) } = N^3 + 2 N^2 (N-1) + N (N-1)^2 + \left( 4 \cdot {{N}\choose{3}} + {{N}\choose{2}} \right) = \frac{14}{3} N^3 - \frac{11}{2} N^2 + \frac{11}{6} N \text{.}
\end{equation*}
For example, consider V-monotonically labeled partitions $(\pi_6, \mathrm{L})$. Let $i_1 = \mathrm{L}\left( \{ 1, 6 \} \right)$, $i_2 = \mathrm{L}\left( \{ 2, 4 \} \right)$ and $i_3 = \mathrm{L}\left( \{ 3, 4 \} \right)$. There are three cases: $i_1 < i_2 < i_3$, $i_1 > i_2 < i_3$, and $i_1 > i_2 > i_3$. The second case consists of three subcases: $i_1 < i_3$, $i_1 = i_3$ and $i_1 > i_3$. The case in which $i_1 = i_3$ consists of $N \choose 2$ possibilities, each of the remaining $4$ cases consists of $N \choose 3$ labelings. The number of labelings for the other partitions $\pi \in \vlNoncrk{2+}{N}(6)$ can be computed directly.

Finally, the total number of labeled partitions in the set $\vlNoncrk{2+}{N}(6)$ equals
\begin{equation*}
\frac{14}{3} N^3 + \frac{7}{2} N^2 - \frac{19}{6} N \text{.}
\end{equation*}



\def\partitionA#1#2#3#4#5{
    \draw[-] (#3+0*#1,#4+0*#2) -- (#3+0*#1,#4+1*#2);
    \node () at (#3+0*#1,#4+0*#2) [leg] {.};
    \draw[-] (#3+1*#1,#4+0*#2) -- (#3+1*#1,#4+1*#2);
    \node () at (#3+1*#1,#4+0*#2) [leg] {.};
    \draw[-] (#3+2*#1,#4+0*#2) -- (#3+2*#1,#4+1*#2);
    \node () at (#3+2*#1,#4+0*#2) [leg] {.};
    \draw[-] (#3+3*#1,#4+0*#2) -- (#3+3*#1,#4+1*#2);
    \node () at (#3+3*#1,#4+0*#2) [leg] {.};
    \draw[-] (#3+4*#1,#4+0*#2) -- (#3+4*#1,#4+1*#2);
    \node () at (#3+4*#1,#4+0*#2) [leg] {.};
    \draw[-] (#3+5*#1,#4+0*#2) -- (#3+5*#1,#4+1*#2);
    \node () at (#3+5*#1,#4+0*#2) [leg] {.};
    \draw[-] (#3+0*#1,#4+1*#2) -- (#3+5*#1,#4+1*#2);

    \node () at (#3+2.5*#1,#4-0.5*#2) [label={[label distance = -2mm]below:#5}] {};
}

\def\partitionB#1#2#3#4#5{
    \draw[-] (#3+0*#1,#4+0*#2) -- (#3+0*#1,#4+1*#2);
    \node () at (#3+0*#1,#4+0*#2) [leg] {.};
    \draw[-] (#3+1*#1,#4+0*#2) -- (#3+1*#1,#4+1*#2);
    \node () at (#3+1*#1,#4+0*#2) [leg] {.};
    \draw[-] (#3+2*#1,#4+0*#2) -- (#3+2*#1,#4+1*#2);
    \node () at (#3+2*#1,#4+0*#2) [leg] {.};
    \draw[-] (#3+3*#1,#4+0*#2) -- (#3+3*#1,#4+1*#2);
    \node () at (#3+3*#1,#4+0*#2) [leg] {.};
    \draw[-] (#3+4*#1,#4+0*#2) -- (#3+4*#1,#4+1*#2);
    \node () at (#3+4*#1,#4+0*#2) [leg] {.};
    \draw[-] (#3+5*#1,#4+0*#2) -- (#3+5*#1,#4+1*#2);
    \node () at (#3+5*#1,#4+0*#2) [leg] {.};
    \draw[-] (#3+0*#1,#4+1*#2) -- (#3+3*#1,#4+1*#2);
    \draw[-] (#3+4*#1,#4+1*#2) -- (#3+5*#1,#4+1*#2);
     
    \node () at (#3+2.5*#1,#4-0.5*#2) [label={[label distance = -2mm]below:#5}] {};
}

\def\partitionC#1#2#3#4#5{
    \draw[-] (#3+0*#1,#4+0*#2) -- (#3+0*#1,#4+2*#2);
    \node () at (#3+0*#1,#4+0*#2) [leg] {.};
    \draw[-] (#3+1*#1,#4+0*#2) -- (#3+1*#1,#4+2*#2);
    \node () at (#3+1*#1,#4+0*#2) [leg] {.};
    \draw[-] (#3+2*#1,#4+0*#2) -- (#3+2*#1,#4+2*#2);
    \node () at (#3+2*#1,#4+0*#2) [leg] {.};
    \draw[-] (#3+3*#1,#4+0*#2) -- (#3+3*#1,#4+1*#2);
    \node () at (#3+3*#1,#4+0*#2) [leg] {.};
    \draw[-] (#3+4*#1,#4+0*#2) -- (#3+4*#1,#4+1*#2);
    \node () at (#3+4*#1,#4+0*#2) [leg] {.};
    \draw[-] (#3+5*#1,#4+0*#2) -- (#3+5*#1,#4+2*#2);
    \node () at (#3+5*#1,#4+0*#2) [leg] {.};
    \draw[-] (#3+0*#1,#4+2*#2) -- (#3+5*#1,#4+2*#2);
    \draw[-] (#3+3*#1,#4+1*#2) -- (#3+4*#1,#4+1*#2);
     
    \node () at (#3+2.5*#1,#4-0.5*#2) [label={[label distance = -2mm]below:#5}] {};
}

\def\partitionD#1#2#3#4#5{
    \draw[-] (#3+0*#1,#4+0*#2) -- (#3+0*#1,#4+1*#2);
    \node () at (#3+0*#1,#4+0*#2) [leg] {.};
    \draw[-] (#3+1*#1,#4+0*#2) -- (#3+1*#1,#4+1*#2);
    \node () at (#3+1*#1,#4+0*#2) [leg] {.};
    \draw[-] (#3+2*#1,#4+0*#2) -- (#3+2*#1,#4+1*#2);
    \node () at (#3+2*#1,#4+0*#2) [leg] {.};
    \draw[-] (#3+3*#1,#4+0*#2) -- (#3+3*#1,#4+1*#2);
    \node () at (#3+3*#1,#4+0*#2) [leg] {.};
    \draw[-] (#3+4*#1,#4+0*#2) -- (#3+4*#1,#4+1*#2);
    \node () at (#3+4*#1,#4+0*#2) [leg] {.};
    \draw[-] (#3+5*#1,#4+0*#2) -- (#3+5*#1,#4+1*#2);
    \node () at (#3+5*#1,#4+0*#2) [leg] {.};
    \draw[-] (#3+0*#1,#4+1*#2) -- (#3+2*#1,#4+1*#2);
    \draw[-] (#3+3*#1,#4+1*#2) -- (#3+5*#1,#4+1*#2);
     
    \node () at (#3+2.5*#1,#4-0.5*#2) [label={[label distance = -2mm]below:#5}] {};
}

\def\partitionE#1#2#3#4#5{
    \draw[-] (#3+0*#1,#4+0*#2) -- (#3+0*#1,#4+2*#2);
    \node () at (#3+0*#1,#4+0*#2) [leg] {.};
    \draw[-] (#3+1*#1,#4+0*#2) -- (#3+1*#1,#4+2*#2);
    \node () at (#3+1*#1,#4+0*#2) [leg] {.};
    \draw[-] (#3+2*#1,#4+0*#2) -- (#3+2*#1,#4+1*#2);
    \node () at (#3+2*#1,#4+0*#2) [leg] {.};
    \draw[-] (#3+3*#1,#4+0*#2) -- (#3+3*#1,#4+1*#2);
    \node () at (#3+3*#1,#4+0*#2) [leg] {.};
    \draw[-] (#3+4*#1,#4+0*#2) -- (#3+4*#1,#4+2*#2);
    \node () at (#3+4*#1,#4+0*#2) [leg] {.};
    \draw[-] (#3+5*#1,#4+0*#2) -- (#3+5*#1,#4+2*#2);
    \node () at (#3+5*#1,#4+0*#2) [leg] {.};
    \draw[-] (#3+0*#1,#4+2*#2) -- (#3+5*#1,#4+2*#2);
    \draw[-] (#3+2*#1,#4+1*#2) -- (#3+3*#1,#4+1*#2);
     
    \node () at (#3+2.5*#1,#4-0.5*#2) [label={[label distance = -2mm]below:#5}] {};
}

\def\partitionF#1#2#3#4#5{
    \draw[-] (#3+0*#1,#4+0*#2) -- (#3+0*#1,#4+2*#2);
    \node () at (#3+0*#1,#4+0*#2) [leg] {.};
    \draw[-] (#3+1*#1,#4+0*#2) -- (#3+1*#1,#4+2*#2);
    \node () at (#3+1*#1,#4+0*#2) [leg] {.};
    \draw[-] (#3+2*#1,#4+0*#2) -- (#3+2*#1,#4+1*#2);
    \node () at (#3+2*#1,#4+0*#2) [leg] {.};
    \draw[-] (#3+3*#1,#4+0*#2) -- (#3+3*#1,#4+1*#2);
    \node () at (#3+3*#1,#4+0*#2) [leg] {.};
    \draw[-] (#3+4*#1,#4+0*#2) -- (#3+4*#1,#4+1*#2);
    \node () at (#3+4*#1,#4+0*#2) [leg] {.};
    \draw[-] (#3+5*#1,#4+0*#2) -- (#3+5*#1,#4+2*#2);
    \node () at (#3+5*#1,#4+0*#2) [leg] {.};
    \draw[-] (#3+0*#1,#4+2*#2) -- (#3+5*#1,#4+2*#2);
    \draw[-] (#3+2*#1,#4+1*#2) -- (#3+4*#1,#4+1*#2);
     
    \node () at (#3+2.5*#1,#4-0.5*#2) [label={[label distance = -2mm]below:#5}] {};
}

\def\partitionG#1#2#3#4#5{
    \draw[-] (#3+0*#1,#4+0*#2) -- (#3+0*#1,#4+1*#2);
    \node () at (#3+0*#1,#4+0*#2) [leg] {.};
    \draw[-] (#3+1*#1,#4+0*#2) -- (#3+1*#1,#4+1*#2);
    \node () at (#3+1*#1,#4+0*#2) [leg] {.};
    \draw[-] (#3+2*#1,#4+0*#2) -- (#3+2*#1,#4+1*#2);
    \node () at (#3+2*#1,#4+0*#2) [leg] {.};
    \draw[-] (#3+3*#1,#4+0*#2) -- (#3+3*#1,#4+1*#2);
    \node () at (#3+3*#1,#4+0*#2) [leg] {.};
    \draw[-] (#3+4*#1,#4+0*#2) -- (#3+4*#1,#4+1*#2);
    \node () at (#3+4*#1,#4+0*#2) [leg] {.};
    \draw[-] (#3+5*#1,#4+0*#2) -- (#3+5*#1,#4+1*#2);
    \node () at (#3+5*#1,#4+0*#2) [leg] {.};
    \draw[-] (#3+0*#1,#4+1*#2) -- (#3+1*#1,#4+1*#2);
    \draw[-] (#3+2*#1,#4+1*#2) -- (#3+5*#1,#4+1*#2);
     
    \node () at (#3+2.5*#1,#4-0.5*#2) [label={[label distance = -2mm]below:#5}] {};
}

\def\partitionH#1#2#3#4#5{
    \draw[-] (#3+0*#1,#4+0*#2) -- (#3+0*#1,#4+1*#2);
    \node () at (#3+0*#1,#4+0*#2) [leg] {.};
    \draw[-] (#3+1*#1,#4+0*#2) -- (#3+1*#1,#4+1*#2);
    \node () at (#3+1*#1,#4+0*#2) [leg] {.};
    \draw[-] (#3+2*#1,#4+0*#2) -- (#3+2*#1,#4+1*#2);
    \node () at (#3+2*#1,#4+0*#2) [leg] {.};
    \draw[-] (#3+3*#1,#4+0*#2) -- (#3+3*#1,#4+1*#2);
    \node () at (#3+3*#1,#4+0*#2) [leg] {.};
    \draw[-] (#3+4*#1,#4+0*#2) -- (#3+4*#1,#4+1*#2);
    \node () at (#3+4*#1,#4+0*#2) [leg] {.};
    \draw[-] (#3+5*#1,#4+0*#2) -- (#3+5*#1,#4+1*#2);
    \node () at (#3+5*#1,#4+0*#2) [leg] {.};
    \draw[-] (#3+0*#1,#4+1*#2) -- (#3+1*#1,#4+1*#2);
    \draw[-] (#3+2*#1,#4+1*#2) -- (#3+3*#1,#4+1*#2);
    \draw[-] (#3+4*#1,#4+1*#2) -- (#3+5*#1,#4+1*#2);
     
    \node () at (#3+2.5*#1,#4-0.5*#2) [label={[label distance = -2mm]below:#5}] {};
}

\def\partitionI#1#2#3#4#5{
    \draw[-] (#3+0*#1,#4+0*#2) -- (#3+0*#1,#4+1*#2);
    \node () at (#3+0*#1,#4+0*#2) [leg] {.};
    \draw[-] (#3+1*#1,#4+0*#2) -- (#3+1*#1,#4+1*#2);
    \node () at (#3+1*#1,#4+0*#2) [leg] {.};
    \draw[-] (#3+2*#1,#4+0*#2) -- (#3+2*#1,#4+2*#2);
    \node () at (#3+2*#1,#4+0*#2) [leg] {.};
    \draw[-] (#3+3*#1,#4+0*#2) -- (#3+3*#1,#4+1*#2);
    \node () at (#3+3*#1,#4+0*#2) [leg] {.};
    \draw[-] (#3+4*#1,#4+0*#2) -- (#3+4*#1,#4+1*#2);
    \node () at (#3+4*#1,#4+0*#2) [leg] {.};
    \draw[-] (#3+5*#1,#4+0*#2) -- (#3+5*#1,#4+2*#2);
    \node () at (#3+5*#1,#4+0*#2) [leg] {.};
    \draw[-] (#3+0*#1,#4+1*#2) -- (#3+1*#1,#4+1*#2);
    \draw[-] (#3+2*#1,#4+2*#2) -- (#3+5*#1,#4+2*#2);
    \draw[-] (#3+3*#1,#4+1*#2) -- (#3+4*#1,#4+1*#2);
     
    \node () at (#3+2.5*#1,#4-0.5*#2) [label={[label distance = -2mm]below:#5}] {};
}

\def\partitionJ#1#2#3#4#5{
    \draw[-] (#3+0*#1,#4+0*#2) -- (#3+0*#1,#4+2*#2);
    \node () at (#3+0*#1,#4+0*#2) [leg] {.};
    \draw[-] (#3+1*#1,#4+0*#2) -- (#3+1*#1,#4+1*#2);
    \node () at (#3+1*#1,#4+0*#2) [leg] {.};
    \draw[-] (#3+2*#1,#4+0*#2) -- (#3+2*#1,#4+1*#2);
    \node () at (#3+2*#1,#4+0*#2) [leg] {.};
    \draw[-] (#3+3*#1,#4+0*#2) -- (#3+3*#1,#4+2*#2);
    \node () at (#3+3*#1,#4+0*#2) [leg] {.};
    \draw[-] (#3+4*#1,#4+0*#2) -- (#3+4*#1,#4+2*#2);
    \node () at (#3+4*#1,#4+0*#2) [leg] {.};
    \draw[-] (#3+5*#1,#4+0*#2) -- (#3+5*#1,#4+2*#2);
    \node () at (#3+5*#1,#4+0*#2) [leg] {.};
    \draw[-] (#3+0*#1,#4+2*#2) -- (#3+5*#1,#4+2*#2);
    \draw[-] (#3+1*#1,#4+1*#2) -- (#3+2*#1,#4+1*#2);
     
    \node () at (#3+2.5*#1,#4-0.5*#2) [label={[label distance = -2mm]below:#5}] {};
}

\def\partitionK#1#2#3#4#5{
    \draw[-] (#3+0*#1,#4+0*#2) -- (#3+0*#1,#4+2*#2);
    \node () at (#3+0*#1,#4+0*#2) [leg] {.};
    \draw[-] (#3+1*#1,#4+0*#2) -- (#3+1*#1,#4+1*#2);
    \node () at (#3+1*#1,#4+0*#2) [leg] {.};
    \draw[-] (#3+2*#1,#4+0*#2) -- (#3+2*#1,#4+1*#2);
    \node () at (#3+2*#1,#4+0*#2) [leg] {.};
    \draw[-] (#3+3*#1,#4+0*#2) -- (#3+3*#1,#4+2*#2);
    \node () at (#3+3*#1,#4+0*#2) [leg] {.};
    \draw[-] (#3+4*#1,#4+0*#2) -- (#3+4*#1,#4+1*#2);
    \node () at (#3+4*#1,#4+0*#2) [leg] {.};
    \draw[-] (#3+5*#1,#4+0*#2) -- (#3+5*#1,#4+1*#2);
    \node () at (#3+5*#1,#4+0*#2) [leg] {.};
    \draw[-] (#3+0*#1,#4+2*#2) -- (#3+3*#1,#4+2*#2);
    \draw[-] (#3+1*#1,#4+1*#2) -- (#3+2*#1,#4+1*#2);
    \draw[-] (#3+4*#1,#4+1*#2) -- (#3+5*#1,#4+1*#2);
     
    \node () at (#3+2.5*#1,#4-0.5*#2) [label={[label distance = -2mm]below:#5}] {};
}

\def\partitionL#1#2#3#4#5{
    \draw[-] (#3+0*#1,#4+0*#2) -- (#3+0*#1,#4+2*#2);
    \node () at (#3+0*#1,#4+0*#2) [leg] {.};
    \draw[-] (#3+1*#1,#4+0*#2) -- (#3+1*#1,#4+1*#2);
    \node () at (#3+1*#1,#4+0*#2) [leg] {.};
    \draw[-] (#3+2*#1,#4+0*#2) -- (#3+2*#1,#4+1*#2);
    \node () at (#3+2*#1,#4+0*#2) [leg] {.};
    \draw[-] (#3+3*#1,#4+0*#2) -- (#3+3*#1,#4+1*#2);
    \node () at (#3+3*#1,#4+0*#2) [leg] {.};
    \draw[-] (#3+4*#1,#4+0*#2) -- (#3+4*#1,#4+2*#2);
    \node () at (#3+4*#1,#4+0*#2) [leg] {.};
    \draw[-] (#3+5*#1,#4+0*#2) -- (#3+5*#1,#4+2*#2);
    \node () at (#3+5*#1,#4+0*#2) [leg] {.};
    \draw[-] (#3+0*#1,#4+2*#2) -- (#3+5*#1,#4+2*#2);
    \draw[-] (#3+1*#1,#4+1*#2) -- (#3+3*#1,#4+1*#2);
     
    \node () at (#3+2.5*#1,#4-0.5*#2) [label={[label distance = -2mm]below:#5}] {};
}

\def\partitionM#1#2#3#4#5{
    \draw[-] (#3+0*#1,#4+0*#2) -- (#3+0*#1,#4+2*#2);
    \node () at (#3+0*#1,#4+0*#2) [leg] {.};
    \draw[-] (#3+1*#1,#4+0*#2) -- (#3+1*#1,#4+1*#2);
    \node () at (#3+1*#1,#4+0*#2) [leg] {.};
    \draw[-] (#3+2*#1,#4+0*#2) -- (#3+2*#1,#4+1*#2);
    \node () at (#3+2*#1,#4+0*#2) [leg] {.};
    \draw[-] (#3+3*#1,#4+0*#2) -- (#3+3*#1,#4+1*#2);
    \node () at (#3+3*#1,#4+0*#2) [leg] {.};
    \draw[-] (#3+4*#1,#4+0*#2) -- (#3+4*#1,#4+1*#2);
    \node () at (#3+4*#1,#4+0*#2) [leg] {.};
    \draw[-] (#3+5*#1,#4+0*#2) -- (#3+5*#1,#4+2*#2);
    \node () at (#3+5*#1,#4+0*#2) [leg] {.};
    \draw[-] (#3+0*#1,#4+2*#2) -- (#3+5*#1,#4+2*#2);
    \draw[-] (#3+1*#1,#4+1*#2) -- (#3+4*#1,#4+1*#2);
     
    \node () at (#3+2.5*#1,#4-0.5*#2) [label={[label distance = -2mm]below:#5}] {};
}

\def\partitionN#1#2#3#4#5{
    \draw[-] (#3+0*#1,#4+0*#2) -- (#3+0*#1,#4+2*#2);
    \node () at (#3+0*#1,#4+0*#2) [leg] {.};
    \draw[-] (#3+1*#1,#4+0*#2) -- (#3+1*#1,#4+1*#2);
    \node () at (#3+1*#1,#4+0*#2) [leg] {.};
    \draw[-] (#3+2*#1,#4+0*#2) -- (#3+2*#1,#4+1*#2);
    \node () at (#3+2*#1,#4+0*#2) [leg] {.};
    \draw[-] (#3+3*#1,#4+0*#2) -- (#3+3*#1,#4+1*#2);
    \node () at (#3+3*#1,#4+0*#2) [leg] {.};
    \draw[-] (#3+4*#1,#4+0*#2) -- (#3+4*#1,#4+1*#2);
    \node () at (#3+4*#1,#4+0*#2) [leg] {.};
    \draw[-] (#3+5*#1,#4+0*#2) -- (#3+5*#1,#4+2*#2);
    \node () at (#3+5*#1,#4+0*#2) [leg] {.};
    \draw[-] (#3+0*#1,#4+2*#2) -- (#3+5*#1,#4+2*#2);
    \draw[-] (#3+1*#1,#4+1*#2) -- (#3+2*#1,#4+1*#2);
    \draw[-] (#3+3*#1,#4+1*#2) -- (#3+4*#1,#4+1*#2);
     
    \node () at (#3+2.5*#1,#4-0.5*#2) [label={[label distance = -2mm]below:#5}] {};
}

\def\partitionO#1#2#3#4#5{
    \draw[-] (#3+0*#1,#4+0*#2) -- (#3+0*#1,#4+3*#2);
    \node () at (#3+0*#1,#4+0*#2) [leg] {.};
    \draw[-] (#3+1*#1,#4+0*#2) -- (#3+1*#1,#4+2*#2);
    \node () at (#3+1*#1,#4+0*#2) [leg] {.};
    \draw[-] (#3+2*#1,#4+0*#2) -- (#3+2*#1,#4+1*#2);
    \node () at (#3+2*#1,#4+0*#2) [leg] {.};
    \draw[-] (#3+3*#1,#4+0*#2) -- (#3+3*#1,#4+1*#2);
    \node () at (#3+3*#1,#4+0*#2) [leg] {.};
    \draw[-] (#3+4*#1,#4+0*#2) -- (#3+4*#1,#4+2*#2);
    \node () at (#3+4*#1,#4+0*#2) [leg] {.};
    \draw[-] (#3+5*#1,#4+0*#2) -- (#3+5*#1,#4+3*#2);
    \node () at (#3+5*#1,#4+0*#2) [leg] {.};
    \draw[-] (#3+0*#1,#4+3*#2) -- (#3+5*#1,#4+3*#2);
    \draw[-] (#3+1*#1,#4+2*#2) -- (#3+4*#1,#4+2*#2);
    \draw[-] (#3+2*#1,#4+1*#2) -- (#3+3*#1,#4+1*#2);
     
    \node () at (#3+2.5*#1,#4-0.5*#2) [label={[label distance = -2mm]below:#5}] {};
}

\begin{figure}
\centering
\begin{tikzpicture}
   \pgfmathsetmacro {\dx}{0.25}
   \pgfmathsetmacro {\Dx}{12*\dx} 
   \pgfmathsetmacro {\dy}{0.35}
   \pgfmathsetmacro {\x}{0}
   \pgfmathsetmacro {\y}{0}


   \tikzstyle{leg} = [circle, draw=black, fill=black!100, text=black!100, thin, inner sep=0pt, minimum size=2.0]

    \partitionH{\dx}{\dy}{\x+0*\Dx}{\y+0*\dy}{$\pi_{1}$}
    \partitionI{\dx}{\dy}{\x+1*\Dx}{\y+0*\dy}{$\pi_{2}$}
    \partitionG{\dx}{\dy}{\x+2*\Dx}{\y+0*\dy}{$\pi_{3}$}
    \partitionK{\dx}{\dy}{\x+3*\Dx}{\y+0*\dy}{$\pi_{4}$}
    \partitionN{\dx}{\dy}{\x+4*\Dx}{\y+0*\dy}{$\pi_{5}$}
    \partitionO{\dx}{\dy}{\x+0*\Dx}{\y-5*\dy}{$\pi_{6}$}
    \partitionM{\dx}{\dy}{\x+1*\Dx}{\y-5*\dy}{$\pi_{7}$}

    \partitionD{\dx}{\dy}{\x+2*\Dx}{\y-5*\dy}{$\pi_{8}$}
    \partitionF{\dx}{\dy}{\x+3*\Dx}{\y-5*\dy}{$\pi_{9}$}
    \partitionL{\dx}{\dy}{\x+4*\Dx}{\y-5*\dy}{$\pi_{10}$}
    \partitionB{\dx}{\dy}{\x+0*\Dx}{\y-10*\dy}{$\pi_{11}$}
    \partitionC{\dx}{\dy}{\x+1*\Dx}{\y-10*\dy}{$\pi_{12}$}
    \partitionE{\dx}{\dy}{\x+2*\Dx}{\y-10*\dy}{$\pi_{13}$}
    \partitionJ{\dx}{\dy}{\x+3*\Dx}{\y-10*\dy}{$\pi_{14}$}

    \partitionA{\dx}{\dy}{\x+4*\Dx}{\y-10*\dy}{$\pi_{15}$}    
\end{tikzpicture}
\caption{All noncrossing partitions $\pi \in \noncrk{2+}(6)$.}
\label{picture:NC_two-plus}
\end{figure}
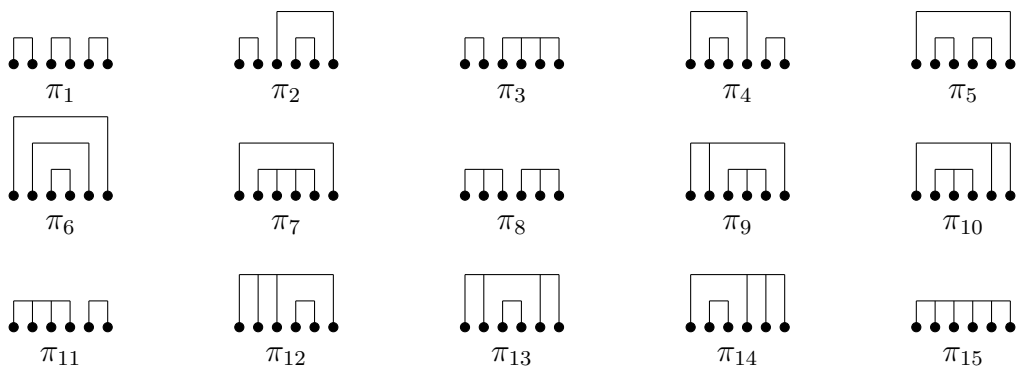

\end{example}

We first establish the formula for the exact moments.
\begin{theorem}
\label{thm:exactMomentFormula}
For $N \in \NNp$ and $\lambda \in \RR$, the moments of $S_{N}(\lambda)$ with respect to $\varphi$ are given by
\begin{equation}
\label{eq:exactMomentFormula}
\momentExact{n}{\lambda}{N} = \sum_{k=0}^{\lfloor \lfrac{n}{2} \rfloor} \frac{ \card{ \vlNoncrk{2+}{N}(n, k) } }{N^k} \cdot \lambda^{n-2k} \text{,}
\end{equation}
where $n \in \NN$.
\end{theorem}

\begin{proof}
Fix $n \in \NN$, $\lambda \in \RR$ and $N \in \NNp$. We have
\begin{equation*}
\momentExact{n}{\lambda}{N} = N^{-\lfrac{n}{2}} \sum_{ \pmb{\varepsilon} \in \{ -, \circ, + \}^n } \sum_{ \mathbf{i} \in [N]^n } \left( \sqrt{N} \lambda \right)^{\card{ \{ l \in [n] : \varepsilon_l = \circ \} }} \varphi \left( A^{\varepsilon_1}_{i_1} \cdot \ldots \cdot A^{\varepsilon_n}_{i_n} \right) \text{.}
\end{equation*}
with $\pmb{\varepsilon} = (\varepsilon_1, \ldots, \varepsilon_n)$ and $\mathbf{i} = (i_1, \ldots, i_n)$. For fixed $\mathbf{i} \in [N]^n$ and each $\pmb{\varepsilon} \in \{ -, \circ, + \}^n$, there exists at most one $\pi \in \noncrk{2+}(n)$ such that $A^{\varepsilon_1}_{i_1} \cdot \ldots \cdot A^{\varepsilon_n}_{i_n} = \partOpDiscA{\pi}{\mathbf{i}}$. Moreover, the total number of middle legs in partition $\pi \in \noncrk{2+}$ is $\ml{\pi} = n - 2 \cdot \card{\pi}$. Combining these facts with Propositions~\ref{proposition:Riordan},~\ref{proposition:bijRiordan} and Lemma~\ref{lemLS}, we get
\begin{equation*}
\label{eq:exactMomentFormula_piL}
\momentExact{n}{\lambda}{N} = N^{-\lfrac{n}{2}} \sum_{\pi \in \noncrk{2+}(n)} \left( \sqrt{N} \lambda \right)^{\ml{\pi}} \sum_{\mathbf{i} \in [N]^n} \varphi \left( \partOpDiscA{\pi}{\mathbf{i}} \right) = \sum_{(\pi, \mathrm{L}) \in \vlNoncrk{2+}{N}(n)} \frac{\lambda^{\ml{\pi}}}{N^{\card{\pi}}} \text{,}
\end{equation*}
which yields~\eqref{eq:exactMomentFormula}.
\end{proof}

\begin{example}
Let us now give the exact formula for the moments of lowest order for $S_{N}(\lambda)$. We have
\begin{itemize}
    \item $\momentExact{0}{\lambda}{N} = 1 \text{,} \quad \momentExact{1}{\lambda}{N} = 0 \text{,} \quad \momentExact{2}{\lambda}{N} = 1 \text{,} \quad \momentExact{3}{\lambda}{N} = \lambda$,

    \item $\momentExact{4}{\lambda}{N} = \lambda^2 + \left( 2 - \tfrac{1}{N} \right)$,
    \item $\momentExact{5}{\lambda}{N} = \lambda^3 + \left( 5 - \tfrac{3}{N} \right) \lambda$,
    \item $\momentExact{6}{\lambda}{N} = \lambda^4 + \left( 9 - \tfrac{6}{N} \right) \lambda ^2 + \left( \tfrac{14}{3} - \tfrac{11}{2 N} + \tfrac{11}{6 N^2} \right)$.
\end{itemize}
The calculation of $\momentExact{6}{\lambda}{N}$ is based on Example~\ref{example:VL_two-plus}. The remaining moments can be computed by analyzing labelings of partitions $\pi \in \noncrk{2+}{(n, k)}$, $1 \leq k \leq n \leq 5$.
\end{example}

\begin{corollary}
\label{corollary:exactMomentGaussian}
For $\lambda = 0$ and nonnegative integers $n$, $N$, we have
\begin{equation*}
\momentExact{2n}{0}{N} = \frac{\card{ \vlNoncrk{2}{N}(2n) }}{N^n} \text{.}
\end{equation*}
All odd moments vanish.
\end{corollary}

\begin{example}
Consider a single operator $A_{i}^{+} + A_{i}^{-} + \lambda A_{i}^{\circ}$ for $\lambda \in \RR$ and $i \in \NNp$, which has the same distribution as $S_{1}(\lambda)$. Its $n$th moment has the form
\begin{equation}
\label{eq:exactMomentFormulaSingle}
\momentExact{n}{\lambda}{1} = \sum_{k=0}^{\lfloor \lfrac{n}{2} \rfloor} \card{ \mathcal{I}^{2+}(n, k) } \cdot \lambda^{n - 2k} \text{,}
\end{equation}
where $\mathcal{I}^{2+}(n, k)$ for $n, k \in \NN$, is the set of all interval partitions $\pi \in \noncrk{2+}(n, k)$ (see Corollary~\ref{corollary:labeledPartOpInterval}). It is worth mentioning that the distribution of $S_1(\lambda)$ is the same as that obtained for monotone independence~\cite[Section~3]{Mur0}, free independence~\cite[Theorem~3.7]{LO2023}, and bm-independence~\cite[Appendix~A]{LOJW2024}.

Moreover, the moment sequence satisfies the following recursive formula
\begin{equation*}
\momentExact{0}{\lambda}{1} = 1 \text{,} \quad
\momentExact{1}{\lambda}{1} = 0 \text{,} \qquad
\momentExact{n+2}{\lambda}{1} = \lambda \momentExact{n+1}{\lambda}{1} + \momentExact{n}{\lambda}{1} \text{,}
\quad n \in \NN \text{,}
\end{equation*}
and, starting from $0$, the sequence equals
\begin{equation*}
\left( 1, 0, 1, \lambda, \lambda^2 + 1, \lambda^3 + 2 \lambda, \lambda^4 + 3 \lambda^2 + 1, \ldots \right) \text{.}
\end{equation*}
Indeed, let $N_{n, k} = \card{\mathcal{I}^{2+}(n, k)}$ for $n, k \in \NN$. The recurrence follows from the fact that
\begin{equation*}
N_{n+2, k+1} = N_{n+1, k+1} + N_{n, k} \text{,}
\end{equation*}
for each $n, k \in \NN$. For $\lambda=1$ and $n \in \NNp$, we have $\momentExact{n}{1}{1} = F_{n-1}$, where $F_n$ is the $n$th Fibonacci number.

Using the above recurrence for moments, one can show that the moment generating function has the following form:
\begin{equation*}
\sum_{n=0}^\infty \momentExact{n}{\lambda}{1} z^n = \frac{\lambda z - 1}{z^2 + \lambda z - 1} \text{.}
\end{equation*}
The probability distribution $\mu_\lambda$ of $S_{1}(\lambda)$ with respect to the vacuum state $\varphi$ is the Bernoulli distribution
\begin{equation*}
\mu_\lambda = p_1 \delta_{x_1} + p_2 \delta_{x_2} \text{,}
\end{equation*}
where $x_{i} = \dfrac{ \lambda + (-1)^i \cdot \sqrt{\lambda^2+4} }{2}$ and $p_{i} = \dfrac{1}{2} - \dfrac{(-1)^i \cdot \lambda}{2\sqrt{\lambda^2+4}}$, $i = 1, 2$ (see also~\cite[Proposition~5.1]{CGW2021}).
\end{example}

The distribution of each summand in $S_{N}(\lambda)$ is described in the following.
\begin{corollary}
\label{corollary:summand}
Let $a_i^{(N)} = \lfrac{(A_i^{+} + A_i^{-})}{\sqrt{N}} + \lambda A_i^{\circ}$, $i \in \NNp$. The operator $a_i^{(N)}$ has the distribution
\begin{equation*}
\label{eq:summand}
\left( \frac{1}{2} + \frac{\lambda}{2 \sqrt{\lambda^2 + \lfrac{4}{N}}} \right) \delta_{\frac{\lambda - \sqrt{\lambda^2 + \lfrac{4}{N}}}{2}} + \left( \frac{1}{2} - \frac{\lambda}{2 \sqrt{\lambda^2 + \lfrac{4}{N}}} \right) \delta_{\frac{\lambda + \sqrt{\lambda^2 + \lfrac{4}{N}}}{2}} \text{,}
\end{equation*}
for any $N \in \NNp$.
\end{corollary}

Our limit theorem is almost identical to the law of small numbers (Poisson limit theorem). The only difference is that in our case the first moment vanishes, whereas in the law of small numbers it is asymptotically equal to $\lfrac{\lambda}{N}$.
\begin{proposition}
\label{proposition:partialPLT}
The family of operators $\{ a_{N,i} : N \in \NNp, i \in [N] \}$ has the following properties:
\begin{itemize}
    \item[ (i)] $a_{N, 1} \ldots, a_{N, N}$ are self-adjoint, V-monotonically independent, and identically distributed random variables with mean $0$,
    \item[(ii)] there exist constants $a, b \in \RR$ such that for each $i \in \NNp$ and any natural $n \geq 2$, we have
\begin{equation*}
\lim \limits_{N \to \infty} N \varphi\left( a_{N, i}^n \right) = a^n b \text{.}
\end{equation*}
\end{itemize}
\end{proposition}

\begin{proof}
The V-monotonically independence follows directly from Proposition~\ref{proposition:Vindependence}. The conditions for moments, with $a = \lambda$ and $b = \lfrac{1}{\lambda^2}$, are implied by~\eqref{eq:exactMomentFormulaSingle}, in which we replace $\lambda$ by $\sqrt{N} \lambda$.
\end{proof}

Now, let us proceed to find the limit moments of $S_{N}(\lambda)$. For that, we use the notion of ordered partitions to study their combinatorics, which we now recall.
\begin{definition}
\label{def:orderedPartitions}

By an \emph{ordered partition} of a set $X \subseteq \ZZ$, we mean a sequence of blocks $\pi = (B_1, \ldots, B_k)$ such that $\widehat{\pi} = \{ B_1, \ldots, B_k \} \in \part(X)$. Each block $B_j$ has its unique \emph{position} in $\pi$, i.e., the \emph{order} of $B_j$ which is the number $j \in [k]$, hence the name `ordered partition'. Note that for every partition $\{ B_1, \ldots, B_k \}$ there exist exactly $k!$ ordered partitions that have the same blocks. Note also that the ordered partition $\pi$ can be seen as a labeled partition $(\widehat{\pi}, \mathrm{L})$, where $\mathrm{L} \colon \widehat{\pi} \to [k]$ assigns to each block its position in $\pi$.

The set of all ordered partitions of $X$ will be denoted by $\oPart(X)$. For noncrossing ordered partitions, we use the same notation as for noncrossing partitions, with the prefix $\mathcal{O}$: $\oNoncr(n)$, $\oNoncrk{2}(n)$, $\oNoncrk{2+}(n)$, $\oNoncrk{2+}(n, k)$, etc., for $k, n \in \NN$. For example, $\oNoncrk{2+}(n, k)$ stands for the set of all noncrossing ordered partitions of the set $[n]$ that have exactly $k$ blocks, each of them having at least two legs.

For $n, k \in \NN$, we denote by $\vNoncrk{2+}(n)$ and $\vNoncrk{2+}(n, k)$ the set of all ordered partitions $\pi$ of $\oNoncrk{2+}(n)$ and $\oNoncrk{2+}(n, k)$, respectively, which are V-monotonically labeled as labeled partitions $(\widehat{\pi}, \mathrm{L})$. This kind of ordered partitions will be called \emph{V-monotone partitions}.
\end{definition}

\begin{example}
\label{example:OV_two-plus}
The set $\oNoncrk{2+}(6) \setminus \vNoncrk{2+}(6)$ consists of $2$ elements, presented in Fig.~\ref{picture:OV_two-plus} in the second row --- the middle and the right ones. We have
\begin{itemize}
    \item $\card{ \vNoncrk{2+}(6, 1) } = 1$,
    \item $\card{ \vNoncrk{2+}(6, 2) } = 18$,
    \item $\card{ \vNoncrk{2+}(6, 3) } = 28$.
\end{itemize}
Indeed, $\card{ \vNoncrk{2+}(6, 2) } = 9 \cdot 2! = 18$ and $\card{ \vNoncrk{2+}(6, 3) } = 5 \cdot 3! - 2 = 28$, see Fig.~\ref{picture:NC_two-plus}. Eventually, the cardinality of $\vNoncrk{2+}(6)$ is $47$.
\begin{figure}
    \centering



\begin{tikzpicture}
    \tikzstyle{leg} = [circle, draw=black, fill=black!100, text=black!100, thin, inner sep=0pt, minimum size=2.0]

    \pgfmathsetmacro {\dx}{0.35}
    \pgfmathsetmacro {\dy}{0.60}

    \pgfmathsetmacro {\x}{0*\dx}
    \pgfmathsetmacro {\y}{0}

    \draw[-] (\x+0*\dx,\y+0*\dy) -- (\x+0*\dx,\y+3*\dy);
    \node () at (\x+0*\dx,\y+0*\dy) [leg] {.};
    \draw[-] (\x+5*\dx,\y+0*\dy) -- (\x+5*\dx,\y+3*\dy);
    \node () at (\x+5*\dx,\y+0*\dy) [leg] {.};
    \draw[-] (\x+0*\dx,\y+3*\dy) -- (\x+5*\dx,\y+3*\dy);
    \node () at (\x+2.5*\dx,\y+3*\dy) [label={[label distance = -2mm]above:\small{1}}] {};
    
    \draw[-] (\x+1*\dx,\y+0*\dy) -- (\x+1*\dx,\y+2*\dy);
    \node () at (\x+1*\dx,\y+0*\dy) [leg] {.};
    \draw[-] (\x+4*\dx,\y+0*\dy) -- (\x+4*\dx,\y+2*\dy);
    \node () at (\x+4*\dx,\y+0*\dy) [leg] {.};
    \draw[-] (\x+1*\dx,\y+2*\dy) -- (\x+4*\dx,\y+2*\dy);
    \node () at (\x+2.5*\dx,\y+2*\dy) [label={[label distance = -2mm]above:\small{2}}] {};
    
    \draw[-] (\x+2*\dx,\y+0*\dy) -- (\x+2*\dx,\y+1*\dy);
    \node () at (\x+2*\dx,\y+0*\dy) [leg] {.};
    \draw[-] (\x+3*\dx,\y+0*\dy) -- (\x+3*\dx,\y+1*\dy);
    \node () at (\x+3*\dx,\y+0*\dy) [leg] {.};
    \draw[-] (\x+2*\dx,\y+1*\dy) -- (\x+3*\dx,\y+1*\dy);
    \node () at (\x+2.5*\dx,\y+1*\dy) [label={[label distance = -2mm]above:\small{3}}] {};

    \pgfmathsetmacro {\x}{13*\dx}
    \pgfmathsetmacro {\y}{0}
\draw[-] (\x+0*\dx,\y+0*\dy) -- (\x+0*\dx,\y+3*\dy);
    \node () at (\x+0*\dx,\y+0*\dy) [leg] {.};
    \draw[-] (\x+5*\dx,\y+0*\dy) -- (\x+5*\dx,\y+3*\dy);
    \node () at (\x+5*\dx,\y+0*\dy) [leg] {.};
    \draw[-] (\x+0*\dx,\y+3*\dy) -- (\x+5*\dx,\y+3*\dy);
    \node () at (\x+2.5*\dx,\y+3*\dy) [label={[label distance = -2mm]above:\small{3}}] {};
    
    \draw[-] (\x+1*\dx,\y+0*\dy) -- (\x+1*\dx,\y+2*\dy);
    \node () at (\x+1*\dx,\y+0*\dy) [leg] {.};
    \draw[-] (\x+4*\dx,\y+0*\dy) -- (\x+4*\dx,\y+2*\dy);
    \node () at (\x+4*\dx,\y+0*\dy) [leg] {.};
    \draw[-] (\x+1*\dx,\y+2*\dy) -- (\x+4*\dx,\y+2*\dy);
    \node () at (\x+2.5*\dx,\y+2*\dy) [label={[label distance = -2mm]above:\small{2}}] {};
    
    \draw[-] (\x+2*\dx,\y+0*\dy) -- (\x+2*\dx,\y+1*\dy);
    \node () at (\x+2*\dx,\y+0*\dy) [leg] {.};
    \draw[-] (\x+3*\dx,\y+0*\dy) -- (\x+3*\dx,\y+1*\dy);
    \node () at (\x+3*\dx,\y+0*\dy) [leg] {.};
    \draw[-] (\x+2*\dx,\y+1*\dy) -- (\x+3*\dx,\y+1*\dy);
    \node () at (\x+2.5*\dx,\y+1*\dy) [label={[label distance = -2mm]above:\small{1}}] {};

    \pgfmathsetmacro {\x}{26*\dx}
    \pgfmathsetmacro {\y}{0}

    \draw[-] (\x+0*\dx,\y+0*\dy) -- (\x+0*\dx,\y+3*\dy);
    \node () at (\x+0*\dx,\y+0*\dy) [leg] {.};
    \draw[-] (\x+5*\dx,\y+0*\dy) -- (\x+5*\dx,\y+3*\dy);
    \node () at (\x+5*\dx,\y+0*\dy) [leg] {.};
    \draw[-] (\x+0*\dx,\y+3*\dy) -- (\x+5*\dx,\y+3*\dy);
    \node () at (\x+2.5*\dx,\y+3*\dy) [label={[label distance = -2mm]above:\small{2}}] {};
    
    \draw[-] (\x+1*\dx,\y+0*\dy) -- (\x+1*\dx,\y+2*\dy);
    \node () at (\x+1*\dx,\y+0*\dy) [leg] {.};
    \draw[-] (\x+4*\dx,\y+0*\dy) -- (\x+4*\dx,\y+2*\dy);
    \node () at (\x+4*\dx,\y+0*\dy) [leg] {.};
    \draw[-] (\x+1*\dx,\y+2*\dy) -- (\x+4*\dx,\y+2*\dy);
    \node () at (\x+2.5*\dx,\y+2*\dy) [label={[label distance = -2mm]above:\small{1}}] {};
    
    \draw[-] (\x+2*\dx,\y+0*\dy) -- (\x+2*\dx,\y+1*\dy);
    \node () at (\x+2*\dx,\y+0*\dy) [leg] {.};
    \draw[-] (\x+3*\dx,\y+0*\dy) -- (\x+3*\dx,\y+1*\dy);
    \node () at (\x+3*\dx,\y+0*\dy) [leg] {.};
    \draw[-] (\x+2*\dx,\y+1*\dy) -- (\x+3*\dx,\y+1*\dy);
    \node () at (\x+2.5*\dx,\y+1*\dy) [label={[label distance = -2mm]above:\small{3}}] {};
    
    \pgfmathsetmacro {\x}{0}
    \pgfmathsetmacro {\y}{-6*\dy}

    \draw[-] (\x+0*\dx,\y+0*\dy) -- (\x+0*\dx,\y+3*\dy);
    \node () at (\x+0*\dx,\y+0*\dy) [leg] {.};
    \draw[-] (\x+5*\dx,\y+0*\dy) -- (\x+5*\dx,\y+3*\dy);
    \node () at (\x+5*\dx,\y+0*\dy) [leg] {.};
    \draw[-] (\x+0*\dx,\y+3*\dy) -- (\x+5*\dx,\y+3*\dy);
    \node () at (\x+2.5*\dx,\y+3*\dy) [label={[label distance = -2mm]above:\small{3}}] {};
    
    \draw[-] (\x+1*\dx,\y+0*\dy) -- (\x+1*\dx,\y+2*\dy);
    \node () at (\x+1*\dx,\y+0*\dy) [leg] {.};
    \draw[-] (\x+4*\dx,\y+0*\dy) -- (\x+4*\dx,\y+2*\dy);
    \node () at (\x+4*\dx,\y+0*\dy) [leg] {.};
    \draw[-] (\x+1*\dx,\y+2*\dy) -- (\x+4*\dx,\y+2*\dy);
    \node () at (\x+2.5*\dx,\y+2*\dy) [label={[label distance = -2mm]above:\small{2}}] {};
    
    \draw[-] (\x+2*\dx,\y+0*\dy) -- (\x+2*\dx,\y+1*\dy);
    \node () at (\x+2*\dx,\y+0*\dy) [leg] {.};
    \draw[-] (\x+3*\dx,\y+0*\dy) -- (\x+3*\dx,\y+1*\dy);
    \node () at (\x+3*\dx,\y+0*\dy) [leg] {.};
    \draw[-] (\x+2*\dx,\y+1*\dy) -- (\x+3*\dx,\y+1*\dy);
    \node () at (\x+2.5*\dx,\y+1*\dy) [label={[label distance = -2mm]above:\small{1}}] {};
    
    \pgfmathsetmacro {\x}{13*\dx}
    \pgfmathsetmacro {\y}{-6*\dy}
    \draw[-] (\x+0*\dx,\y+0*\dy) -- (\x+0*\dx,\y+3*\dy);
    \node () at (\x+0*\dx,\y+0*\dy) [leg] {.};
    \draw[-] (\x+5*\dx,\y+0*\dy) -- (\x+5*\dx,\y+3*\dy);
    \node () at (\x+5*\dx,\y+0*\dy) [leg] {.};
    \draw[-] (\x+0*\dx,\y+3*\dy) -- (\x+5*\dx,\y+3*\dy);
    \node () at (\x+2.5*\dx,\y+3*\dy) [label={[label distance = -2mm]above:\small{1}}] {};
    
    \draw[-] (\x+1*\dx,\y+0*\dy) -- (\x+1*\dx,\y+2*\dy);
    \node () at (\x+1*\dx,\y+0*\dy) [leg] {.};
    \draw[-] (\x+4*\dx,\y+0*\dy) -- (\x+4*\dx,\y+2*\dy);
    \node () at (\x+4*\dx,\y+0*\dy) [leg] {.};
    \draw[-] (\x+1*\dx,\y+2*\dy) -- (\x+4*\dx,\y+2*\dy);
    \node () at (\x+2.5*\dx,\y+2*\dy) [label={[label distance = -2mm]above:\small{3}}] {};
    
    \draw[-] (\x+2*\dx,\y+0*\dy) -- (\x+2*\dx,\y+1*\dy);
    \node () at (\x+2*\dx,\y+0*\dy) [leg] {.};
    \draw[-] (\x+3*\dx,\y+0*\dy) -- (\x+3*\dx,\y+1*\dy);
    \node () at (\x+3*\dx,\y+0*\dy) [leg] {.};
    \draw[-] (\x+2*\dx,\y+1*\dy) -- (\x+3*\dx,\y+1*\dy);
    \node () at (\x+2.5*\dx,\y+1*\dy) [label={[label distance = -2mm]above:\small{2}}] {};
    \pgfmathsetmacro {\x}{26*\dx}
    \pgfmathsetmacro {\y}{-6*\dy}

    \draw[-] (\x+0*\dx,\y+0*\dy) -- (\x+0*\dx,\y+3*\dy);
    \node () at (\x+0*\dx,\y+0*\dy) [leg] {.};
    \draw[-] (\x+5*\dx,\y+0*\dy) -- (\x+5*\dx,\y+3*\dy);
    \node () at (\x+5*\dx,\y+0*\dy) [leg] {.};
    \draw[-] (\x+0*\dx,\y+3*\dy) -- (\x+5*\dx,\y+3*\dy);
    \node () at (\x+2.5*\dx,\y+3*\dy) [label={[label distance = -2mm]above:\small{2}}] {};
    
    \draw[-] (\x+1*\dx,\y+0*\dy) -- (\x+1*\dx,\y+2*\dy);
    \node () at (\x+1*\dx,\y+0*\dy) [leg] {.};
    \draw[-] (\x+4*\dx,\y+0*\dy) -- (\x+4*\dx,\y+2*\dy);
    \node () at (\x+4*\dx,\y+0*\dy) [leg] {.};
    \draw[-] (\x+1*\dx,\y+2*\dy) -- (\x+4*\dx,\y+2*\dy);
    \node () at (\x+2.5*\dx,\y+2*\dy) [label={[label distance = -2mm]above:\small{3}}] {};
    
    \draw[-] (\x+2*\dx,\y+0*\dy) -- (\x+2*\dx,\y+1*\dy);
    \node () at (\x+2*\dx,\y+0*\dy) [leg] {.};
    \draw[-] (\x+3*\dx,\y+0*\dy) -- (\x+3*\dx,\y+1*\dy);
    \node () at (\x+3*\dx,\y+0*\dy) [leg] {.};
    \draw[-] (\x+2*\dx,\y+1*\dy) -- (\x+3*\dx,\y+1*\dy);
    \node () at (\x+2.5*\dx,\y+1*\dy) [label={[label distance = -2mm]above:\small{1}}] {};
\end{tikzpicture}
    \caption{Ordered noncrossing partitions.}
    \label{picture:OV_two-plus}
\end{figure}
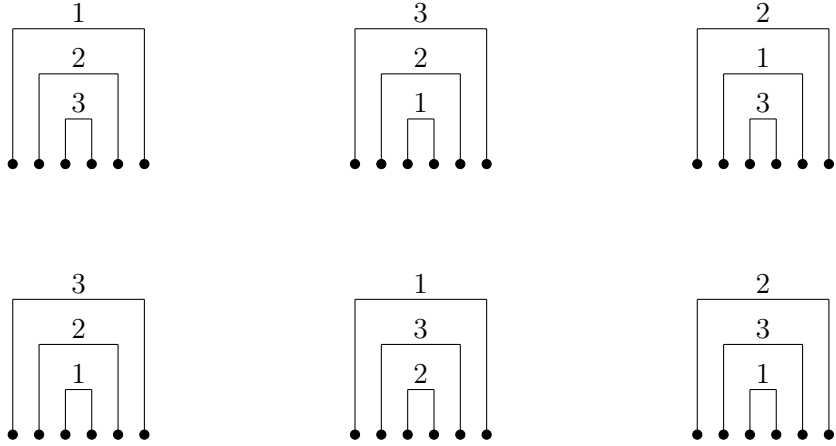
\end{example}

We are now in a position to state the main theorem of this section.
\begin{theorem}
\label{thm:limitMomentFormula}
For $n \in \NN$, the $n$th limit moment of $S_{N}(\lambda)$ with respect to $\varphi$ is given by
\begin{equation}
\label{eq:limitMomentFormula}
\momentLimit{n}{\lambda} = \sum_{k=0}^{\lfloor \lfrac{n}{2} \rfloor} \frac{ \card{ \vNoncrk{2+}(n, k) } }{k!} \cdot \lambda^{n-2k} \text{,}
\end{equation}
where $\lambda \in \RR$.
\end{theorem}

\begin{proof}
Fix $n \in \NN$, $\lambda \in \RR$ and $N \in \NNp$. For a partition $\pi \in \noncrk{2+}(n)$, we denote by $\vlNoncr{N}(\pi)$ and $\vNoncr(\pi)$ the set of all $(\pi', \mathrm{L}) \in \vlNoncrk{2+}{N}(n)$ such that $\pi' = \pi$ and the set of all $\pi_0 \in \vNoncrk{2+}(n)$ such that $\widehat{\pi}_0 = \pi$, respectively. We rewrite~\eqref{eq:exactMomentFormula} as
\begin{equation}
\label{eq:exactMomentFormula_pi}
\momentExact{n}{\lambda}{N} = \sum_{ \pi \in \noncrk{2+}(n) } \frac{\card{ \vlNoncr{N}(\pi) }}{N^{\card{\pi}}} \cdot \lambda^{n - 2 \card{\pi}} \text{.}
\end{equation}
It suffices to prove that for each $\pi \in \noncrk{2+}$, we have
\begin{equation}
\label{equation:vlNoncr_asymptotics}
\lim_{N \to \infty} \frac{\card{ \vlNoncr{N}(\pi) }}{N^{\card{\pi}}} = \frac{\card{ \vNoncr(\pi) }}{\card{\pi}!} \text{.}
\end{equation}
If $\mathrm{L} \colon \pi \to [N]$ is not injective, then it has at most $\card{\pi}-1$ different values and therefore there are at most $\smallO{N^{\card{\pi}}}$ non-injective labelings.

Otherwise, there exists a unique ordered partition $\pi_0 \in \oNoncrk{2+}(n)$ with the same blocks as $\pi$, such that the block $B \in \pi$ is on $r$th position in $\pi_0$ if and only if $\mathrm{L}(B)$ is the $r$th smallest number in the range of $\mathrm{L}$, for some $r \in [\card{\pi}]$. Since $(\pi, \mathrm{L}) \in \vlNoncrk{2+}{N}(n)$ if and only if $\pi_0 \in \vNoncrk{2+}(n)$, there are exactly $\card{ \vNoncr(\pi) } \cdot {N \choose \card{\pi}}$ injective V-monotone labelings $L \colon \pi \to [N]$. Therefore, we have
\begin{equation*}
\card{\vlNoncr{N}(\pi)} = \card{ \vNoncr(\pi) } \cdot {N \choose \card{\pi}} + \smallO{N^{\card{\pi}}} \text{,}
\end{equation*}
which leads to
\begin{equation}
\label{eq:limitMomentFormula_pi1}
\momentLimit{n}{\lambda} = \sum_{ \pi \in \noncrk{2+}(n) } \frac{\card{ \vNoncr(\pi) }}{\card{\pi}!} \cdot \lambda^{n - 2 \card{\pi}} \text{,}
\end{equation}
and, consequently,~\eqref{equation:vlNoncr_asymptotics} holds.
\end{proof}

\begin{example}
For illustration, we present the lowest-order moment of V-monotone Poisson distribution compared with their monotone~\cite{Mur0}, free~\cite{LO2023} and some bm~\cite{LOJW2024} counterparts. In all cases, we have $m_1(\lambda) = 0$, $m_2(\lambda) = 1$ and $m_3(\lambda) = \lambda$. Table~\ref{table1} contains the remaining moments of orders~up~to~$6$.

\begin{table}[H]
\begin{center}
\renewcommand{\arraystretch}{1.5}
\begin{tabular}{ | m{3em} | m{2cm} | m{2cm} | m{2cm} | m{2.5cm} | m{2.5cm} | }
  \hline
  & free & monotone & bm $(\mathbb{R}_{+}^{2})$ & bm $(\mathbb{R}_{+}^{3})$ & V-monotone \\

  \hline
  $\momentLimit{4}{\lambda}$ & $\lambda^2 + 2$ & $\lambda^2 + \frac{3}{2}$ &  $\lambda^2 + \frac{5}{4}$ & $\lambda^2 + \frac{9}{8}$ & $\lambda^2 + 2$ \\
  \hline
  $\momentLimit{5}{\lambda}$ & $\lambda^3 + 5 \lambda$ & $\lambda^3 + \frac{7}{2} \lambda$ & $\lambda^3 + \frac{11}{4} \lambda$ & $\lambda^3 + \frac{19}{8} \lambda$ & $\lambda^3 + 5 \lambda$ \\
  \hline
  $\momentLimit{6}{\lambda}$ & $\lambda^4+9\lambda^2+5$ & $\lambda^4+6\lambda^2+\frac{5}{2}$ & $\lambda^4+\frac{9}{2}\lambda^2+\frac{59}{36}$ & $\lambda^4+\frac{15}{4}\lambda^2+\frac{31}{24}$ & $\lambda^4+9\lambda^2+\frac{14}{3}$ \\
  \hline
\end{tabular}
\caption{The moments of orders $4$, $5$ and $6$ for V-monotone, free, monotone and bm (for the cones $\mathbb{R}^2_+$ and $\mathbb{R}^3_+$) Poisson distributions.}
\label{table1}
\renewcommand{\arraystretch}{1}
\end{center}
\end{table}
\end{example}

\section{Moment generating function}
\label{section:MGF}
In this section, we derive the moment generating functions of the exact and limit distributions of $S_{N}(\lambda)$, for $\lambda \in \RR$, that is,
\begin{equation*}
\mgfExact{\lambda}{N}(z) = \sum_{n=0}^{\infty} \momentExact{n}{\lambda}{N} z^n \text{, } N \in \NNp \quad \text{and} \quad \mgfLimit{\lambda}(z) = \sum_{n=0}^{\infty}\momentLimit{n}{\lambda} z^n \text{, } N \to \infty \text{.}
\end{equation*}

We start with $S_{N}(\lambda)$ for fixed $N \in \NNp$. The moment generating function is given recursively.

\def\cardV#1#2#3{ c_{#1, #2}(#3) }
\def\cardM#1#2{ \cardV{#1}{0}{#2} }
\def\mgfV#1#2#3#4#5{ P^{#5}_{#1, #2}(#3, #4) }
\def\mgfM#1#2#3#4{ \mgfV{#1}{0}{#2}{#3}{#4} }

\begin{theorem}
\label{thm:MGFexact}
For $N \in \NNp$ and $\lambda \in \RR$, the moment generating function of the operator $S_N(\lambda)$ has the form
\begin{equation}
\label{eq:MGFexact}
\mgfExact{\lambda}{N}(z) = \mgfV{N}{N+1}{z}{\lambda}{} \text{,} \quad z \in \CC \text{, $\lvert z \rvert$ small enough,}
\end{equation}

where $\left( P_{N, l} : N \in \NN \text{, } 0 \leq l \leq N+1 \right)$ is a family of functions
\begin{equation*}
P_{N, l} \colon \left\{ (z, \lambda) \in \CC^2 : \lvert z \rvert < \min \left( \frac{1}{4}, \frac{1}{4 \lvert \lambda \rvert} \right) \right\} \to \CC \text{}
\end{equation*}
satisfying the following recurrence
\begin{equation}
\label{eq:mgfExactRecurrence}
\begin{split}
& \frac{1}{\mgfV{N}{l}{z}{\lambda}{}} = \\
& \begin{cases}
1 - \dfrac{z^2}{N} \cdot \left( \dsum \limits_{i=1}^{l-1} \mgfV{N}{i}{z}{0}{} + \dsum \limits_{i=l+1}^N\mgfM{N-i}{z}{0}{} \right) & \text{for $\lambda = 0$,} \\
1 - \dfrac{z}{\lambda N} \cdot \left( \dsum \limits_{i=1}^{l-1} \left( \dfrac{1}{1 - \lambda z \mgfV{N}{i}{z}{\lambda}{}} - 1 \right) + \dsum \limits_{i=l+1}^N \left( \dfrac{1}{1 - \lambda z \mgfM{N-i}{z}{\lambda}{}} - 1 \right) \right) & \text{for $\lambda \neq 0$,}
\end{cases}
\end{split}
\end{equation}
for $N \in \NN$ and $l \in [0, N+1]$. We adopt the convention that, for $l \in \{ 0, 1, N, N+1 \}$, the empty sums equals $0$, which means that $P_{0, l}$ is constant and equals $1$.
\end{theorem}

\begin{proof}
Combining Theorem~\ref{thm:exactMomentFormula} and~\eqref{eq:exactMomentFormula_pi}, we have
\begin{equation}
\label{eq:MGFexact_pi}
\mgfExact{\lambda}{N}(z) = \sum_{\pi \in \mathcal{NC}^{2+}} \frac{\card{ \vlNoncr{N}(\pi) }}{N^{\card{\pi}}} \lambda^{\legs{\pi} - 2 \card{\pi}} z^{\legs{\pi}} \text{,}
\end{equation}

where $\legs{\pi}$ is the number of all legs in $\pi$ and the set $\vlNoncr{N}(\pi)$ was defined in the proof of Theorem~\ref{thm:exactMomentFormula}. We now must show that
\begin{equation}
\label{eq:mgfV}
\mgfV{N}{l}{z}{\lambda}{} \coloneqq \sum_{\pi \in \mathcal{NC}^{2+}} \frac{\cardV{N}{l}{\pi}}{N^{\card{\pi}}} \lambda^{\legs{\pi} - 2 \card{\pi}} z^{\legs{\pi}} \text{}
\end{equation}
is the solution of~\eqref{eq:mgfExactRecurrence}, where for a noncrossing partition $\pi \in \noncr$, we denote by $\cardV{N}{l}{\pi}$, $N, l \in \NN$, the number of all V-monotone labelings $\mathrm{L} \colon \pi \to [N]$, such that the partition remains V-monotonically labeled if we add the ``imaginary block'' $\{0, \legs{\pi}+1\}$ with the label $l$ to the labeled partition $(\pi, \mathrm{L})$. Note that $\cardM{N}{\pi}$ is the number of all monotone labelings of $\pi$, i.e., $\cardV{N}{0}{\pi}$ equals the cardinality of
\begin{equation*}
\left\{ \mathrm{L} \colon \pi \to [N] : \mathrm{L}(B') < \mathrm{L}(B) \text{ whenever $B' \prec_{\pi} B$, for all $B, B' \in \pi$} \right\} \text{.}
\end{equation*}
Moreover, $\cardV{N}{l}{\pi} = \card{ \vlNoncr{N}(\pi) }$ if $l > N$, which leads to~\eqref{eq:MGFexact}. Let us first show that the radius of convergence of~\eqref{eq:mgfV} with respect to $z$ satisfies
\begin{equation}
\label{eq:radiusOfConvergenceExact}
R(\lambda) \geq \min \left( \frac{1}{4}, \frac{1}{4 \lvert \lambda \rvert} \right) \text{.}
\end{equation}
Indeed,
\begin{equation*}
\left \lvert \sum_{\pi \in \noncrk{2+}(n)} \frac{ \cardV{N}{l}{\pi} }{N^{\card{\pi}}} \lambda^{n - 2 \card{\pi}} \right \rvert \leq \max(1, \lvert \lambda \rvert^n) \cdot \left( \sum_{\pi \in \noncrk{2+}(n)} 1 \right) \leq \max(1, \lvert \lambda \rvert^n) \cdot C_n \text{,}
\end{equation*}
where $(C_n)_{n=0}^{\infty}$ is the sequence of Catalan numbers, whose radius of convergence is $\lfrac{1}{4}$. In order to show that~\eqref{eq:mgfExactRecurrence} is fulfilled, let us find a recurrence relation for the numbers $\cardV{N}{l}{\pi}$.

Obviously, $\cardV{N}{l}{\emptyset} = 1$. Let us prove that for $l \in [0, N+1]$ and $\pi = \ncPartitionRecurrenceD{\pi'_1}{\cdots}{\pi'_p}{\pi''}$, we have
\begin{equation*}
\cardV{N}{l}{\pi} = \left( \sum_{i=1}^{l-1} \prod_{q=1}^p \cardV{N}{i}{\pi'_q} + \sum_{i=l+1}^N \prod_{q=1}^p \cardM{N-i}{\pi'_q} \right) \cdot \cardV{N}{l}{\pi''} \text{}
\end{equation*}
with $p \in \NNp$ and $\pi'_1, \ldots, \pi'_q, \pi'' \in \noncrk{2+}$. Indeed, let us assign the label $i \in [N]$ to the block $B \ni 1$.
Consider three cases:
\begin{itemize}
	\item $i < l$: in this case, there are $\cardV{N}{i}{\pi'_q}$ labelings of $\pi'_q$, for each $q \in [p]$,
	\item $i = l$: this case is impossible,
	\item $i > l$: in this case, for all $q \in [p]$, $\pi'_q$ must be monotonically labeled. There are $\cardM{N-i}{\pi'_q}$ such labelings.
\end{itemize}

To prove~\eqref{eq:mgfExactRecurrence}, first note that
\begin{equation*}
\noncrk{2+} = \{ \emptyset \} \cup \bigcup_{p=1}^{\infty} \bigcup_{\pi_1', \pi_2', \ldots, \pi_p', \pi'' \in \noncrk{2+}} \left\{ \ncPartitionRecurrenceD{\pi'_1}{\cdots}{\pi'_p}{\pi''} \right\} \text{}
\end{equation*}
and, for $\pi = \ncPartitionRecurrenceD{\pi'_1}{\cdots}{\pi'_p}{\pi''} \in \noncrk{2+}(n, k)$, such that $\pi'_{q} \in \noncrk{2+}(n'_{q}, k'_{q})$ for $q \in [p]$ and $\pi'' \in \noncrk{2+}(n'', k'')$, we have
\begin{itemize}
    \item $n = n'_1 + \cdots + n'_p + (p+1) + n''$,
    \item $k = k'_1 + \cdots + k'_p + 1 + k''$,
\end{itemize}

Now, fix $N \in \NN$, $l \in [0, N+1]$ and a real $\lambda \neq 0$. For a complex $z$ from a neighborhood of $0$, we have
\begin{align*}
\mgfV{N}{l}{z}{\lambda}{}
= & \sum_{\pi \in \noncrk{2+}} \frac{ \cardV{N}{l}{\pi} }{N^k} z^n \lambda^{n - 2 k} \\
= & \, 1 + \frac{z}{\lambda N} \cdot \sum_{p=1}^{\infty} (\lambda z)^p \sum_{\pi'_1, \ldots, \pi'_p, \pi''} \bigg( \sum_{i=1}^{l-1} \prod_{q=1}^p \frac{ \cardV{N}{i}{\pi'_q} }{N^{k'_q}} z^{n'_q} \lambda^{n'_q - 2 k'_q} \\
& + \sum_{i=l+1}^N \prod_{q=1}^p \frac{ \cardM{N-i}{\pi'_q} }{N^{k'_q}} z^{n'_q} \lambda^{n'_q - 2 k'_q} \bigg) \cdot \frac{ \cardV{N}{l}{\pi''} }{N^{k''}} z^{n''} \lambda^{n'' - 2 k''} \\
= & \, 1 + \frac{z}{\lambda N} \cdot \sum_{p=1}^{\infty} (\lambda z)^p \bigg( \sum_{i=1}^{l-1} \prod_{q=1}^p \sum_{\pi'_q \in \noncrk{2+}} \frac{ \cardV{N}{i}{\pi'_q} }{N^{k'_q}} z^{n'_q} \lambda^{n'_q - 2 k'_q} \\
& + \sum_{i=l+1}^N \prod_{q=1}^p \sum_{\pi'_q \in \noncrk{2+}} \frac{ \cardM{N-i}{\pi'_q} }{N^{k'_q}} z^{n'_q} \lambda^{n'_q - 2 k'_q} \bigg) \bigg( \sum_{\pi'' \in \noncrk{2+}} \frac{ \cardV{N}{l}{\pi''} }{N^{k''}} z^{n''} \lambda^{n'' - 2 k''} \bigg) \\
= & \, 1 + \frac{z}{\lambda N} \cdot \sum_{p=1}^{\infty} (\lambda z)^p \bigg( \sum_{i=1}^{l-1} \mgfV{N}{i}{z}{\lambda}{p} + \sum_{i=l+1}^N \mgfM{N-i}{z}{\lambda}{p} \bigg) \mgfV{N}{l}{z}{\lambda}{} \\
= & \, 1 + \frac{z \mgfV{N}{l}{z}{\lambda}{}}{\lambda N} \bigg( \sum_{i=1}^{l-1} \Big( \frac{1}{1 - \lambda z \mgfV{N}{i}{z}{\lambda}{}} - 1 \Big) \\
& + \sum_{i=l+1}^N \Big( \frac{1}{1 - \lambda z \mgfM{N-i}{z}{\lambda}{}} - 1 \Big) \bigg) \text{,}
\end{align*}
which directly leads to~\eqref{eq:mgfExactRecurrence} for $\lambda \neq 0$. The proof for $\lambda = 0$ is similar and we leave it to the reader.
\end{proof}

\begin{remark}
The function $z \mapsto \mgfV{N}{0}{z}{\lambda}{}$, $N \in \NNp$, $\lambda \in \RR$, is the moment generating function of the monotone counterpart of $S_{N}(\lambda)$. Its exact distribution for $\lambda = 0$, up to constant $\sqrt{N}$, was studied in~\cite{CL2020}.
\end{remark}
\let\mgfM\undefined
\let\mgfV\undefined
\let\cardM\undefined
\let\cardV\undefined

\begin{example}
We have $\mgfExact{\lambda}{1}(z) = \dfrac{\lambda z - 1}{z^2 + \lambda z 
- 1}$ (this is the same as in the free and monotone cases). Using Python's \texttt{sympy} module, we calculate that $\mgfExact{\lambda}{2}(z)$ equals
\begin{center}
\scalebox{1.15}{$
\frac{z^{5} \left(8 \lambda^{5} - 12 \lambda^{3} + 4 \lambda\right) + z^{4} \left(- 40 \lambda^{4} + 36 \lambda^{2} - 4\right) + z^{3} \left(80 \lambda^{3} - 36 \lambda\right) + z^{2} \left(12 - 80 \lambda^{2}\right) + 40 \lambda z - 8}{z^{6} \left(8 \lambda^{4} - 8 \lambda^{2} + 1\right) + z^{5} \left(8 \lambda^{5} - 44 \lambda^{3} + 20 \lambda\right) + z^{4} \left(- 40 \lambda^{4} + 84 \lambda^{2} - 12\right) + z^{3} \left(80 \lambda^{3} - 68 \lambda\right) + z^{2} \left(20 - 80 \lambda^{2}\right) + 40 \lambda z - 8} \text{.}
$}

\end{center}
\end{example}

Now, we are ready to find the moment generating function for the limit measure of $S_{N}(\lambda)$ as $N \to \infty$. From now on, this limit measure will be denoted by $\mu_{\lambda}$. For that, we introduce two families of polynomials indexed by noncrossing partitions.
\begin{definition}
\label{GPQPolynom}
For a noncrossing partition $\pi \in \noncrk{2+}$, we define the polynomials $Q_{\pi} \colon [0, 1] \to \RR$ and $P_{\pi} \colon [0, 1] \to \RR$ recursively as follows:
\begin{equation*}
Q_{ \pi }(s) =
\begin{cases}
1 & \text{if $\pi = \emptyset$,} \\
\Big( \int_{s}^{1} Q_{\pi'_1}(t) \cdots Q_{\pi'_k}(t) \, dt\Big)Q_{\pi''}(s) & \text{if $\pi = \ncPartitionRecurrenceD{\pi'_1}{\ldots}{\pi'_k}{\pi''}$,}
\end{cases}
\end{equation*}
and
\begin{equation*}
P_{ \pi }(s) =
\begin{cases}
1 & \text{if $\pi = \emptyset$,} \\
\Big( \int_{0}^{s} P_{\pi'_1}(t) \cdots P_{\pi'_k}(t) \, dt + \int_{s}^{1} Q_{\pi'_1}(t) \cdots Q_{\pi'_k}(t) \, dt \Big) P_{\pi''}(s) & \text{if $\pi = \ncPartitionRecurrenceD{\pi'_1}{\ldots}{\pi'_k}{\pi''}$.}
\end{cases}
\end{equation*}
These polynomials are generalizations of those introduced in~\cite{AD2020}.
\end{definition}

\begin{remark}

For all $\pi \in \noncrk{2+}$, it follows that $Q_{\pi}(s)=Q_{\tilde{\pi}}(s)$ and $P_{\pi}(s)=P_{\tilde{\pi}}(s)$, where $\tilde{\pi} \in \noncrk{2}$ is obtained from $\pi$ by removing all middle legs.
\end{remark}

\begin{example}
Some examples of $P_{\pi}$ and $Q_{\pi}$ for the smallest noncrossing pair partitions are presented in Fig.~\ref{figure:pairPartitionPolynomialsOfTypePandQ}.
\begin{figure}
\centering



\begin{tikzpicture}
    \tikzstyle{leg} = [circle, draw=black, fill=black!100, text=black!100, thin, inner sep=0pt, minimum size=2.0]

    \pgfmathsetmacro {\dx}{0.25}
    \pgfmathsetmacro {\dy}{0.35}

    \pgfmathsetmacro {\x}{2*\dx}
    \pgfmathsetmacro {\y}{0*\dy}

    \draw[-] (\x+0*\dx,\y+0*\dy) -- (\x+0*\dx,\y+1*\dy);
    \node () at (\x+0*\dx,\y+0*\dy) [leg] {.};
    \draw[-] (\x+1*\dx,\y+0*\dy) -- (\x+1*\dx,\y+1*\dy);
    \node () at (\x+1*\dx,\y+0*\dy) [leg] {.};
    \draw[-] (\x+0*\dx,\y+1*\dy) -- (\x+1*\dx,\y+1*\dy);

    \node () at (\x+3.5*\dx,\y+0*\dy) [label={[label distance = -1.5mm]right:{$Q_{\pi} = (1-s)$}}] {};    
    \node () at (\x+3.5*\dx,\y+2*\dy) [label={[label distance = -1.5mm]right:{$P_{\pi} = 1$}}] {};

    \pgfmathsetmacro {\x}{19*\dx}
    \pgfmathsetmacro {\y}{0*\dy}

    \draw[-] (\x+0*\dx,\y+0*\dy) -- (\x+0*\dx,\y+1*\dy);
    \node () at (\x+0*\dx,\y+0*\dy) [leg] {.};
    \draw[-] (\x+1*\dx,\y+0*\dy) -- (\x+1*\dx,\y+1*\dy);
    \node () at (\x+1*\dx,\y+0*\dy) [leg] {.};
    \draw[-] (\x+0*\dx,\y+1*\dy) -- (\x+1*\dx,\y+1*\dy);
    
    \draw[-] (\x+2*\dx,\y+0*\dy) -- (\x+2*\dx,\y+1*\dy);
    \node () at (\x+2*\dx,\y+0*\dy) [leg] {.};
    \draw[-] (\x+3*\dx,\y+0*\dy) -- (\x+3*\dx,\y+1*\dy);
    \node () at (\x+3*\dx,\y+0*\dy) [leg] {.};
    \draw[-] (\x+2*\dx,\y+1*\dy) -- (\x+3*\dx,\y+1*\dy);

    \node () at (\x+4.5*\dx,\y+0*\dy) [label={[label distance = -1.5mm]right:{$Q_{\pi} = (1-s)^2$}}] {};    
    \node () at (\x+4.5*\dx,\y+2*\dy) [label={[label distance = -1.5mm]right:{$P_{\pi} = 1$}}] {};

    \pgfmathsetmacro {\x}{37*\dx}
    \pgfmathsetmacro {\y}{0*\dy}

    \draw[-] (\x+0*\dx,\y+0*\dy) -- (\x+0*\dx,\y+2*\dy);
    \node () at (\x+0*\dx,\y+0*\dy) [leg] {.};
    \draw[-] (\x+3*\dx,\y+0*\dy) -- (\x+3*\dx,\y+2*\dy);
    \node () at (\x+3*\dx,\y+0*\dy) [leg] {.};
    \draw[-] (\x+0*\dx,\y+2*\dy) -- (\x+3*\dx,\y+2*\dy);
    
    \draw[-] (\x+1*\dx,\y+0*\dy) -- (\x+1*\dx,\y+1*\dy);
    \node () at (\x+1*\dx,\y+0*\dy) [leg] {.};
    \draw[-] (\x+2*\dx,\y+0*\dy) -- (\x+2*\dx,\y+1*\dy);
    \node () at (\x+2*\dx,\y+0*\dy) [leg] {.};
    \draw[-] (\x+1*\dx,\y+1*\dy) -- (\x+2*\dx,\y+1*\dy);
    
    \node () at (\x+4.5*\dx,\y+0*\dy) [label={[label distance = -1.5mm]right:{$Q_{\pi} = \tfrac{(1-s)^2}{2}$}}] {};
    \node () at (\x+4.5*\dx,\y+2*\dy) [label={[label distance = -1.5mm]right:{$P_{\pi} = \tfrac{s^2+1}{2}$}}] {};

    \pgfmathsetmacro {\x}{0*\dx}
    \pgfmathsetmacro {\y}{-6*\dy}

    \draw[-] (\x+0*\dx,\y+0*\dy) -- (\x+0*\dx,\y+1*\dy);
    \node () at (\x+0*\dx,\y+0*\dy) [leg] {.};
    \draw[-] (\x+1*\dx,\y+0*\dy) -- (\x+1*\dx,\y+1*\dy);
    \node () at (\x+1*\dx,\y+0*\dy) [leg] {.};
    \draw[-] (\x+0*\dx,\y+1*\dy) -- (\x+1*\dx,\y+1*\dy);
    
    \draw[-] (\x+2*\dx,\y+0*\dy) -- (\x+2*\dx,\y+1*\dy);
    \node () at (\x+2*\dx,\y+0*\dy) [leg] {.};
    \draw[-] (\x+3*\dx,\y+0*\dy) -- (\x+3*\dx,\y+1*\dy);
    \node () at (\x+3*\dx,\y+0*\dy) [leg] {.};
    \draw[-] (\x+2*\dx,\y+1*\dy) -- (\x+3*\dx,\y+1*\dy);
    
    \draw[-] (\x+4*\dx,\y+0*\dy) -- (\x+4*\dx,\y+1*\dy);
    \node () at (\x+4*\dx,\y+0*\dy) [leg] {.};
    \draw[-] (\x+5*\dx,\y+0*\dy) -- (\x+5*\dx,\y+1*\dy);
    \node () at (\x+5*\dx,\y+0*\dy) [leg] {.};
    \draw[-] (\x+4*\dx,\y+1*\dy) -- (\x+5*\dx,\y+1*\dy);
    
    \node () at (\x+5.5*\dx,\y+0*\dy) [label={[label distance = -1.5mm]right:{$Q_{\pi} = (1-s)^3$}}] {};
    \node () at (\x+5.5*\dx,\y+2*\dy) [label={[label distance = -1.5mm]right:{$P_{\pi} = 1$}}] {};

    \pgfmathsetmacro {\x}{18*\dx}
    \pgfmathsetmacro {\y}{-6*\dy}

    \draw[-] (\x+0*\dx,\y+0*\dy) -- (\x+0*\dx,\y+1*\dy);
    \node () at (\x+0*\dx,\y+0*\dy) [leg] {.};
    \draw[-] (\x+1*\dx,\y+0*\dy) -- (\x+1*\dx,\y+1*\dy);
    \node () at (\x+1*\dx,\y+0*\dy) [leg] {.};
    \draw[-] (\x+0*\dx,\y+1*\dy) -- (\x+1*\dx,\y+1*\dy);
    
    \draw[-] (\x+2*\dx,\y+0*\dy) -- (\x+2*\dx,\y+2*\dy);
    \node () at (\x+2*\dx,\y+0*\dy) [leg] {.};
    \draw[-] (\x+5*\dx,\y+0*\dy) -- (\x+5*\dx,\y+2*\dy);
    \node () at (\x+5*\dx,\y+0*\dy) [leg] {.};
    \draw[-] (\x+2*\dx,\y+2*\dy) -- (\x+5*\dx,\y+2*\dy);
    
    \draw[-] (\x+3*\dx,\y+0*\dy) -- (\x+3*\dx,\y+1*\dy);
    \node () at (\x+3*\dx,\y+0*\dy) [leg] {.};
    \draw[-] (\x+4*\dx,\y+0*\dy) -- (\x+4*\dx,\y+1*\dy);
    \node () at (\x+4*\dx,\y+0*\dy) [leg] {.};
    \draw[-] (\x+3*\dx,\y+1*\dy) -- (\x+4*\dx,\y+1*\dy);

    \node () at (\x+5.5*\dx,\y+0*\dy) [label={[label distance = -1.5mm]right:{$Q_{\pi} = \tfrac{(1-s)^3}{2}$}}] {};    
    \node () at (\x+5.5*\dx,\y+2*\dy) [label={[label distance = -1.5mm]right:{$P_{\pi} = \frac{s^2+1}{2}$}}] {};

    \pgfmathsetmacro {\x}{36*\dx}
    \pgfmathsetmacro {\y}{-6*\dy}

    \draw[-] (\x+0*\dx,\y+0*\dy) -- (\x+0*\dx,\y+2*\dy);
    \node () at (\x+0*\dx,\y+0*\dy) [leg] {.};
    \draw[-] (\x+3*\dx,\y+0*\dy) -- (\x+3*\dx,\y+2*\dy);
    \node () at (\x+3*\dx,\y+0*\dy) [leg] {.};
    \draw[-] (\x+0*\dx,\y+2*\dy) -- (\x+3*\dx,\y+2*\dy);
    
    \draw[-] (\x+1*\dx,\y+0*\dy) -- (\x+1*\dx,\y+1*\dy);
    \node () at (\x+1*\dx,\y+0*\dy) [leg] {.};
    \draw[-] (\x+2*\dx,\y+0*\dy) -- (\x+2*\dx,\y+1*\dy);
    \node () at (\x+2*\dx,\y+0*\dy) [leg] {.};
    \draw[-] (\x+1*\dx,\y+1*\dy) -- (\x+2*\dx,\y+1*\dy);
    
    \draw[-] (\x+4*\dx,\y+0*\dy) -- (\x+4*\dx,\y+1*\dy);
    \node () at (\x+4*\dx,\y+0*\dy) [leg] {.};
    \draw[-] (\x+5*\dx,\y+0*\dy) -- (\x+5*\dx,\y+1*\dy);
    \node () at (\x+5*\dx,\y+0*\dy) [leg] {.};
    \draw[-] (\x+4*\dx,\y+1*\dy) -- (\x+5*\dx,\y+1*\dy);
    
    \node () at (\x+5.5*\dx,\y+0*\dy) [label={[label distance = -1.5mm]right:{$Q_{\pi} = \tfrac{(1-s)^3}{2}$}}] {};
    \node () at (\x+5.5*\dx,\y+2*\dy) [label={[label distance = -1.5mm]right:{$P_{\pi} = \frac{s^2+1}{2}$}}] {};
    
    \pgfmathsetmacro {\x}{0}
    \pgfmathsetmacro {\y}{-13*\dy}

    \draw[-] (\x+0*\dx,\y+0*\dy) -- (\x+0*\dx,\y+3*\dy);
    \node () at (\x+0*\dx,\y+0*\dy) [leg] {.};
    \draw[-] (\x+5*\dx,\y+0*\dy) -- (\x+5*\dx,\y+3*\dy);
    \node () at (\x+5*\dx,\y+0*\dy) [leg] {.};
    \draw[-] (\x+0*\dx,\y+3*\dy) -- (\x+5*\dx,\y+3*\dy);
    
    \draw[-] (\x+1*\dx,\y+0*\dy) -- (\x+1*\dx,\y+1*\dy);
    \node () at (\x+1*\dx,\y+0*\dy) [leg] {.};
    \draw[-] (\x+2*\dx,\y+0*\dy) -- (\x+2*\dx,\y+1*\dy);
    \node () at (\x+2*\dx,\y+0*\dy) [leg] {.};
    \draw[-] (\x+1*\dx,\y+1*\dy) -- (\x+2*\dx,\y+1*\dy);
    
    \draw[-] (\x+3*\dx,\y+0*\dy) -- (\x+3*\dx,\y+1*\dy);
    \node () at (\x+3*\dx,\y+0*\dy) [leg] {.};
    \draw[-] (\x+4*\dx,\y+0*\dy) -- (\x+4*\dx,\y+1*\dy);
    \node () at (\x+4*\dx,\y+0*\dy) [leg] {.};
    \draw[-] (\x+3*\dx,\y+1*\dy) -- (\x+4*\dx,\y+1*\dy);
    
    \node () at (\x+5.5*\dx,\y+0*\dy) [label={[label distance = -1.5mm]right:{$Q_{\pi} = \tfrac{(1-s)^3}{3}$}}] {};
    \node () at (\x+5.5*\dx,\y+2*\dy) [label={[label distance = -1.5mm]right:{$P_{\pi} = \tfrac{-s^3 + 3s^2 + 1}{3}$}}] {};
    
    \pgfmathsetmacro {\x}{18*\dx}
    \pgfmathsetmacro {\y}{-13*\dy}

    \draw[-] (\x+0*\dx,\y+0*\dy) -- (\x+0*\dx,\y+3*\dy);
    \node () at (\x+0*\dx,\y+0*\dy) [leg] {.};
    \draw[-] (\x+5*\dx,\y+0*\dy) -- (\x+5*\dx,\y+3*\dy);
    \node () at (\x+5*\dx,\y+0*\dy) [leg] {.};
    \draw[-] (\x+0*\dx,\y+3*\dy) -- (\x+5*\dx,\y+3*\dy);
    
    \draw[-] (\x+1*\dx,\y+0*\dy) -- (\x+1*\dx,\y+2*\dy);
    \node () at (\x+1*\dx,\y+0*\dy) [leg] {.};
    \draw[-] (\x+4*\dx,\y+0*\dy) -- (\x+4*\dx,\y+2*\dy);
    \node () at (\x+4*\dx,\y+0*\dy) [leg] {.};
    \draw[-] (\x+1*\dx,\y+2*\dy) -- (\x+4*\dx,\y+2*\dy);
    
    \draw[-] (\x+2*\dx,\y+0*\dy) -- (\x+2*\dx,\y+1*\dy);
    \node () at (\x+2*\dx,\y+0*\dy) [leg] {.};
    \draw[-] (\x+3*\dx,\y+0*\dy) -- (\x+3*\dx,\y+1*\dy);
    \node () at (\x+3*\dx,\y+0*\dy) [leg] {.};
    \draw[-] (\x+2*\dx,\y+1*\dy) -- (\x+3*\dx,\y+1*\dy);
    
    \node () at (\x+5.5*\dx,\y+0*\dy) [label={[label distance = -1.5mm]right:{$Q_{\pi} = \tfrac{(1-s)^3}{6}$}}] {};
    \node () at (\x+5.5*\dx,\y+2*\dy) [label={[label distance = -1.5mm]right:{$P_{\pi} = \tfrac{3s^2 + 1}{6}$}}] {};
    
\end{tikzpicture}
    \caption{$Q_{\pi}$ and $P_{\pi}$ for the smallest noncrossing pair partitions}
\label{figure:pairPartitionPolynomialsOfTypePandQ}
\end{figure}
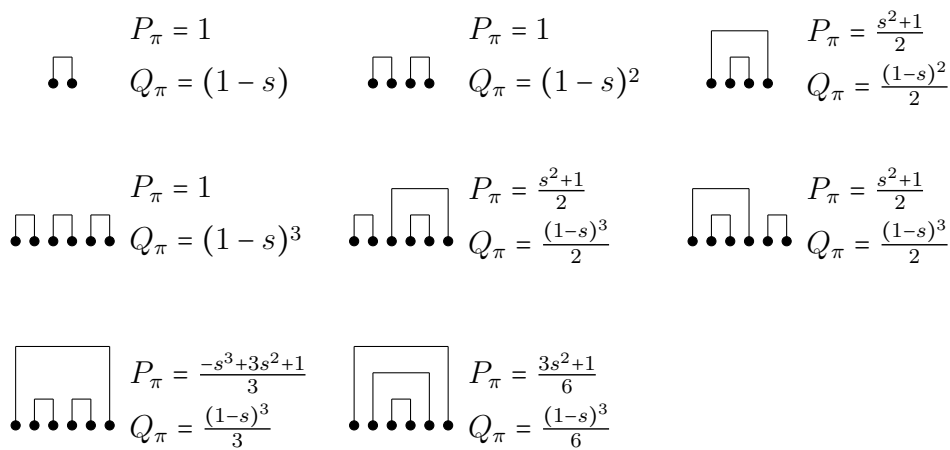
\noindent For instance,
$$
P_{\pi}(s) = \left( \int_{0}^{s} \frac{t^2+1}{2} \, dt + \int_{s}^{1} \frac{(1-t)^2}{2} \, dt \right) \cdot 1 = \dfrac{3s^2 + 1}{6} \text{,}
$$
where $\pi = \big\{ \{ 1,6 \}, \{ 2,5 \}, \{ 3,4 \} \big\}$.
\end{example}

The first step in the process of finding $\mgfLimit{\lambda}$, $\lambda \in \RR$, is to express it by means of the polynomials $P_{\pi}$, $\pi \in \noncrk{2+}$.
\begin{lemma}
The moment generating function of $\mu_{\lambda}$ has the form
\begin{equation}
\label{eq:MGFpartialPoisson}
\mgfLimit{\lambda}(z) =
\sum_{\pi \in \noncrk{2+}} P_{\pi}(1) \cdot \lambda^{\legs{\pi} - 2 \card{\pi}} z^{\legs{\pi}} \text{}
\end{equation}
for each $\lambda \in \RR$ and $z \in \CC$ from some neighborhood of $0$, where $\legs{\pi}$ is the number of all legs in $\pi$.
\end{lemma}

\begin{proof}
We use notation from the proof of Theorem~\ref{thm:limitMomentFormula}. According to~\eqref{eq:limitMomentFormula_pi1}, the moment generating function of $\mu_{\lambda}$, $\lambda \in \RR$, is equal to
\begin{equation}
\label{eq:MGFpartialPoisson0}
\mgfLimit{\lambda}(z) =
\sum_{\pi \in \noncrk{2+}} \frac{\card{ \vNoncr(\pi) }}{\card{\pi}!} \cdot \lambda^{\legs{\pi} - 2 \card{\pi}} z^{\legs{\pi}} \text{.}
\end{equation}
Define $A^{+}(N)=\frac{1}{\sqrt{N}}\sum_{i=1}^{N}A_i^+, A^-(N)=\frac{1}{\sqrt{N}}\sum_{i=1}^{N}A_i^-$. For a noncrossing pair partition $\pi$, consider the operator
\begin{equation*}
A_{\pi}(N)=
\begin{cases}
\id & \text{if $\pi = \emptyset$,} \\
A^-(N)A_{\pi'}(N)A^{+}(N)A_{\pi''}(N) & \text{if $\pi = \ncPartitionRecurrenceB{\pi'}{\pi''}$.}
\end{cases}
\end{equation*}
By multilinearity of $\varphi$, Propositions~\ref{proposition:Riordan},~\ref{proposition:bijRiordan} and Lemma~\ref{lemLS}, we obtain
\begin{align*}
\varphi\left(A_{\pi}(N)\right)=
N^{-\card{\pi}} \cdot \sum_{\mathbf{i} \in [N]^{\legs{\pi}}} \varphi(\partOpDiscA{\pi}{\mathbf{i}}) = N^{-\card{\pi}} \sum_{(\pi, \mathrm{L}) \in \vlNoncrk{2+}{N}(\pi)} 1 = \frac{\card{ \vlNoncr{N}(\pi) }}{N^{\card{\pi}}} \text{.}
\end{align*}
Using~\eqref{equation:vlNoncr_asymptotics} and Lemmas~6.12,~7.3 from~\cite{AD2020}, we get
\begin{equation}
\frac{\card{ \vNoncr(\pi) }}{\card{\pi}!} = \lim \limits_{N \to \infty} \frac{\card{ \vlNoncr{N}(\pi) }}{N^{\card{\pi}}}=\lim\limits_{N\to\infty}\varphi(A_{\pi}(N))=P_{\pi}(1)
\end{equation}
for all $\pi \in \noncrk{2}$. Moreover,
\begin{equation*}
\frac{\card{ \vNoncr(\pi) }}{\card{\pi}!} = \frac{\card{ \vNoncr(\tilde{\pi}) }}{\card{\tilde{\pi}}!} = P_{\tilde{\pi}}(1) = P_{\pi}(1) \text{,}
\end{equation*}
where $\tilde{\pi} \in \noncrk{2}$ is obtained from $\pi$ by removing all middle legs. Our assertion follows then from~\eqref{eq:MGFpartialPoisson0}.
\end{proof}

In order to find $M_{\lambda}$, let us introduce an auxiliary function
\begin{equation}
P(z, \lambda; s) = \sum_{\pi \in \noncrk{2+}} P_{\pi}(s) \cdot \lambda^{\legs{\pi} - 2 \cdot \card{\pi}} z^{\legs{\pi}} \text{,} \qquad s \in [0, 1] \text{.}
\end{equation}
Of course,
\begin{equation*}
M_{\lambda}(x) = P(x, \lambda; 1) \text{}
\end{equation*}
for a real $x$ small enough, therefore it suffices to find the integral equation for $P$, which is done in the following lemma.
\begin{lemma}
The function $s \mapsto P(x, \lambda; s)$ fulfills the following integral equation:
\begin{equation}
\label{eq:MGFpartialPoissonP}
\left( \sqrt{ (1 - \lambda x)^2 - 2 (1-s) x^2 } - \frac{x}{\lambda} \left( \int_{0}^{s} \frac{\diff t}{1 - \lambda x P(x, \lambda; t)} -s \right) + \lambda x \right) P(x, \lambda; s) = 1 \text{,}
\end{equation}
where $s \in [0, 1]$, $\lambda \in \RR \setminus \{ 0 \}$, and $0 < x < \lfrac{1}{\lvert \lambda \rvert}$.
\end{lemma}

\begin{proof}
Let $\legs{\pi}$ and $\ml{\pi}$ be the numbers of all legs and all middle legs in $\pi$, respectively. For $\pi \in \noncrk{2+}$, we have $\ml{\pi} = \legs{\pi} - 2 \cdot \card{\pi}$. Define 
\begin{equation*}
Q(x, \lambda; s) = \sum_{\pi\in NC^{2+}} Q_{\pi}(s) \cdot \lambda^{\ml{\pi}} x^{\legs{\pi}}
\end{equation*}
and let us first prove that the following system of integral equations is fulfilled:
\begin{equation}
\label{eq:MGFpartialPoissonP0}
P(x, \lambda; s) = 1 + \frac{x}{\lambda} \left( \int_{0}^{s} \frac{\diff t}{1 - \lambda x P(x, \lambda; t)} + \int_{s}^{1} \frac{\diff t}{1 - \lambda x Q(x, \lambda; t)} - 1 \right) P(x, \lambda; s) \text{,}
\end{equation}
and
\begin{equation}
\label{eq:MGFpartialPoissonQ}
Q(x, \lambda; s) = 1 + \frac{x}{\lambda} \left( \int_{s}^{1} \frac{\diff t}{1 - \lambda x Q(x, \lambda; t)} - (1-s) \right) Q(x, \lambda; s) \text{.}
\end{equation}
According to Definition \ref{GPQPolynom}, we have 
\begin{align*}
P(x, \lambda; s)
= & \, \sum_{\pi\in \noncrk{2+}} P_{\pi}(s) x^{\legs{\pi}} \lambda^{\ml{\pi}} = 1 + \frac{x}{\lambda} \sum_{p=1}^{\infty} (x\lambda)^p \\
& \cdot \sum_{\pi_1', \ldots, \pi_p', \pi''\in \noncrk{2+}} \Big( \int_{0}^{s} P_{\pi_1'}(t) x^{\legs{\pi_1'}} \lambda^{\ml{\pi_1'}} \cdots P_{\pi_p'}(t) x^{\legs{\pi_p'}} \lambda^{\ml{\pi_p'}} \diff{t} \\
& + \int_{s}^{1} Q_{\pi_1'}(t) x^{\legs{\pi_1'}} \lambda^{\ml{\pi_1'}} \cdots Q_{\pi_p'}(t) x^{\legs{\pi_p'}} \lambda^{\ml{\pi_p'}} \diff{t} \Big) P_{\pi''}(s) x^{\legs{\pi''}} \lambda^{\ml{\pi''}} \text{.}
\end{align*}
Applying Fubini's theorem, we obtain 
\begin{align*}
P(x, \lambda; s)
= & \, 1 + \frac{x}{\lambda} \sum_{p=1}^{\infty}(x \lambda)^{p} \Bigg( \int_{0}^{s} \prod_{q=1}^{p} \left( \sum_{\pi_q'} P_{\pi'_q}(t) x^{ \legs{\pi'_q} } \lambda^{ \ml{\pi'_q} } \right) \diff{t} \\
& + \int_{s}^{1} \prod_{q=1}^{p} \left( \sum_{\pi_q'} Q_{\pi'_q}(t) x^{ \legs{\pi'_q} } \lambda^{ \ml{\pi'_q} } \right) \diff{t} \Bigg) \cdot \sum_{\pi''} P_{\pi''}(s) x^{ \legs{\pi''} } \lambda^{ \ml{\pi''} } \\
= & \, 1 + \frac{x}{\lambda} \sum_{p=1}^{\infty}(x \lambda)^{p} \left( \int_0^s P^{p}(x, \lambda; t) \diff{t} + \int_s^1 Q^{p}(x, \lambda; t) \diff{t} \right) \cdot P(x, \lambda; s) \\
= & \, 1 + \frac{x}{\lambda} \left( \int_0^s \frac{\diff{t}}{1 - x \lambda P(x, \lambda; t)} + \int_s^1 \frac{\diff{t}}{1 - x \lambda Q(x, \lambda; t)} - 1 \right) \cdot P(x, \lambda; s) \text{.}
\end{align*}
The prove of~\eqref{eq:MGFpartialPoissonQ} is similar, so we leave it to the reader.

To obtain~\eqref{eq:MGFpartialPoissonP}, let us substitute
\begin{equation*}
y(s)=\int_{s}^{1}\frac{dt}{1-x\lambda Q(x,\lambda;t)}-(1-s)
\end{equation*}

in~\eqref{eq:MGFpartialPoissonQ}. Multiplying it by $\lambda x (y'(s) - 1)$, we get
\begin{equation*}
y'(s) = \lambda x \cdot (y'(s) - 1) + \frac{x}{\lambda} y(s) y'(s) \text{,}
\end{equation*}
obtaining the following initial value problem:
\begin{equation*}
\begin{cases}
(1 - \lambda x - \dfrac{x}{\lambda} y) y' = - \lambda x \text{,} \\
y(1) = 0 \text{,}
\end{cases} \qquad
y = y(s) \text{, } s \in [0, 1] \text{.}
\end{equation*}
This initial value problem leads to the quadratic equation
\begin{equation*}
- \frac{x}{2 \lambda} y^2 + (1 - \lambda x)y - \lambda x (1 - s) = 0 \text{,}
\end{equation*}
which, according to the initial condition and the fact that $0 < x < \lfrac{1}{\lvert \lambda \rvert}$, has the following solution
\begin{equation*}
y(s) = - \frac{\lambda}{x} \sqrt{ (1 - \lambda x)^2 - 2 (1-s) x^2 } + \frac{\lambda}{x} - \lambda^2 \text{.}
\end{equation*}
Replacing
\begin{equation*}
\int_{s}^{1} \frac{\diff t}{1 - \lambda x Q(x, \lambda; t)} = y(s) + (1-s) \text{}
\end{equation*}
in~\eqref{eq:MGFpartialPoissonP0}, after elementary transformations, we get~\eqref{eq:MGFpartialPoissonP}, which finishes the proof.
\end{proof}

\begin{remark}
Using~\eqref{eq:MGFpartialPoissonQ}, we can also get
\begin{equation*}
Q(x, \lambda; s) = \frac{1}{\lambda x + \sqrt{ (1 - \lambda x)^2 - 2 (1-s) x^2 }} \text{,}
\end{equation*}
which satisfies
\begin{equation*}
Q(x, \lambda, s) = Q \left( x \sqrt{1-s}, \frac{\lambda}{\sqrt{1-s}}; 0 \right) \text{.}
\end{equation*}
It is worth mentioning that $x \mapsto Q(x, \lambda; 0)$ is the moment generating function of the limit distribution of the monotone counterpart of the operator $S_N(\lambda)$ (see~\cite[Section~5]{Mur0} with a different notation).
\end{remark}

Now, we are in a position to establish the moment generating function of the operator $S_N(\lambda)$. It is expressed in terms of the function $T_0 \colon [0, \infty) \mapsto \RR$ given by
\begin{equation}
\label{equation:theRealVersionOfT}
T_0(s) = -\frac{\sqrt{3}}{3} \bigg( \arctan \bigg( \frac{2s-1}{\sqrt{3}} \bigg) - \frac{\pi}{6} \bigg) - \frac{\ln( s^2 - s + 1 )}{2} \text{,}
\end{equation}
which can be written as the following integral:
\begin{equation}
\label{equation:theRealVersionOfTInTheIntegralForm}
T_0(s) = \int_{s}^{1} \dfrac{t}{t^{2} - t + 1} \, \diff{t} \text{.}
\end{equation}
From that, we can easily see that $T_0$ is decreasing and its range is $(-\infty, T_0(0)]$, where $T_0(0) = \lfrac{\sqrt{3} \pi}{9}$ and the limit of $T_0(s)$ as $s \to \infty$ is $-\infty$. We define $T_0^{-1} \colon ( -\infty, \lfrac{\sqrt{3} \pi}{9} ] \to [0, \infty)$ as the inverse of $T_0$. 

\begin{theorem}
The moment generating function $M_{\lambda}$ of the partial Poisson distribution with the parameter $\lambda \in \RR$ is given by
\begin{equation}
\label{eq:MGF}
\frac{1}{M_{\lambda}(x)} = \lambda x + (1 - \lambda x) \cdot T_0^{-1} \left( \frac{1}{2} \ln \left( \frac{(1 - \lambda x)^2}{(1 - \lambda x)^2 - 2x^2} \right) \right) \text{}
\end{equation}
for $x > 0$ small enough.
\end{theorem}

\begin{proof}
Let us first reduce the integral equation~\eqref{eq:MGFpartialPoissonP} to an initial value problem, proceeding similarly while solving~\eqref{eq:MGFpartialPoissonQ}. Namely, we substitute
\begin{equation}
\label{eq:MGFy}
y(s) = \int_{0}^{s} \frac{dt}{1-x\lambda P(x, \lambda; t)} - s \text{,}
\end{equation}
we multiply both sides by $\lambda x (y'(s) + 1)$, and we subtract $\lambda x y'(s)$, obtaining
\begin{equation*}
\begin{cases}
\left( \sqrt{ (1 - \lambda x)^2 - 2 (1-s) x^2 } - \dfrac{x}{\lambda} y \right) y' = \lambda x \text{,} \\
y(0) = 0 \text{,}
\end{cases} \qquad
y = y(s) \text{, } s \in [0, 1] \text{.}
\end{equation*}
We now change coordinates
\begin{equation*}
\xi = \frac{\lambda}{x} \sqrt{ (1 - \lambda x)^2 - 2 (1-s) x^2 } \text{,} \qquad u = \frac{\lambda}{x} \sqrt{ (1 - \lambda x)^2 - 2 (1-s) x^2 } - y(s) \text{.}
\end{equation*}
The inverse mapping is given by
\begin{equation*}
s = \dfrac{\xi^2}{2 \lambda^2} - \dfrac{(1-\lambda x)^2}{2 x^2} + 1 \text{,} \qquad y = \xi - u(\xi) \text{,}
\end{equation*}
with $\xi \geq 0$. We differentiate with respect to $\xi$, obtaining the following ODE:
\begin{equation*}
\frac{x}{\lambda} \cdot u \cdot \frac{\lambda^2}{\xi} \cdot (1-u') = \lambda x \text{,}
\end{equation*}
where $u' = u'(\xi) = \frac{\diff}{\diff \xi} u(\xi)$. Therefore, the initial value problem takes the form
\begin{equation}
\label{eq:AbelODE}
\begin{cases}
u u' - u = - \xi \text{,} \\
u(\xi_0) = \xi_0 \text{,}
\end{cases}
\end{equation}
where $\xi_0 = \lfrac{\lambda}{x} \cdot \sqrt{ (1 - \lambda x)^2 - 2 x^2 }$, with $u = u(\xi)$ and $\xi$ between $\xi_0$ and $\lfrac{\lambda}{x} - \lambda^2$. The ODE here is an Abel ordinary differential equation of the second kind, whose general solution, according to~\cite[eq.~1.3.1.2.]{PZ2002}, has the parametric form
\begin{equation*}
\xi = C e^{T_0(t)} \text{,} \quad u = C t e^{T_0(t)} \text{,} \quad t \in \RR \text{.}
\end{equation*}
The curve describing our solution must pass through the point $(\xi_0, \xi_0)$. Since $u = t \xi$, this can only hold for $t=1$ for the suitable curve. Therefore, $C = \xi_0$ since $T_0(1) = 0$ (see~\eqref{equation:theRealVersionOfTInTheIntegralForm}) and the solution of~\eqref{eq:AbelODE} is given by
\begin{equation*}
u(\xi) = \xi T_{0}^{-1}\left(\ln\left(\frac{x\xi}{\lambda\sqrt{(1-\lambda x)^2-2x^2}}\right)\right) \text{.}
\end{equation*}
In fact, we need only $M_{\lambda}(x) = P(x, \lambda; 1)$, which, due to~\eqref{eq:MGFy}, can be obtained from
\begin{equation*}
\frac{1}{P(x, \lambda; 1)} = \lambda x \left( 1 + \frac{1}{y'(1)} \right) \text{,}
\end{equation*}
where, using~\eqref{eq:AbelODE}, the derivative of $y$ can be expressed by
\begin{equation*}
y'(s) = \frac{\diff \xi}{\diff s} - \frac{\diff u}{\diff \xi} \cdot \frac{\diff \xi}{\diff s} = \frac{\lambda^2}{\xi(s)} ( 1 - u '(\xi(s)) ) = \frac{\lambda^2}{u(\xi(s))} \text{.}
\end{equation*}
That yields
\begin{equation*}
\frac{1}{P(x, \lambda; 1)} = \lambda x \left( 1 + \frac{u(\xi(1))}{\lambda^2} \right) = \lambda x + (1 - \lambda x) \cdot T_{0}^{-1}\left(\ln\left(\frac{1-\lambda x}{\sqrt{(1-\lambda x)^2-2x^2}}\right)\right) \text{,}
\end{equation*}
which equals to~\eqref{eq:MGF} for $x > 0$ small enough. This proves our theorem.
\end{proof}

\begin{remark}
For $\lambda = 0$, we obtain
\begin{equation*}
\frac{1}{M_{0}(x)} = T_0^{-1} \left( \frac{1}{2} \ln \left( \frac{1}{1 - 2x^2} \right) \right) \text{,}
\end{equation*}
the moment generating function for the standard V-monotone Gaussian distribution~\cite{AD2020}.
\end{remark}

\section{Limit measure}
\label{section:limitMeasure}
In this section, we obtain the limit measure $\mu_{\lambda}$ of the nonsymmetric V-monotone position operator $S_N(\lambda)$ using the Stieltjes inversion formula. Let $G_{\lambda} \colon \CC \setminus \text{supp}(\mu_{\lambda}) \to \CC$ be the Cauchy--Stieltjes transform of the limit measure. Of course, $\mu_{\lambda}$ has a compact support since the radius of convergence of its moment generating function fulfills~\eqref{eq:radiusOfConvergenceExact}, which can be proven similarly as the aforementioned inequality. Using the formula
\begin{equation*}
G_{\lambda}(x) = \frac{1}{x} M_{\lambda} \left( \frac{1}{x} \right) \text{,}
\end{equation*}
we get
\begin{equation}
\label{eq:CSTreal}
\frac{1}{G_{\lambda}(x)} = \lambda + (x - \lambda) \cdot T_0^{-1} \left( \frac{1}{2} \ln \left( 1 + \frac{2}{(x - \lambda)^2 - 2} \right) \right) \text{}
\end{equation}
for $x > 0$ big enough. Note that
\begin{equation}
\label{eq:CST0}
G_{\lambda}(x) = \frac{G_0(x-\lambda)}{1 + \lambda G_{0}(x-\lambda)} \text{,} \quad x > 0 \text{ big enough,}
\end{equation}
where $G_{0}$ is the Cauchy--Stieltjes transform of the standard V-monotone Gaussian distribution obtained and analytically extended in~\cite{AD2025}.
Let us now recall all the necessary definitions and statements required to write the aforementioned transform. In the definitions below, we use the principal value of the logarithm and argument, i.e.,
\begin{equation}
\label{equation:theSuitableLogarithm}
\principalLogarithm z = \ln R + i \varphi \text{,} \qquad \text{for $z = R e^{i \varphi}$ with $R > 0$ and $-\pi < \varphi \leq \pi$,}
\end{equation}
which is analytic on $\CC \setminus (-\infty, 0 ]$. Let $\principalArgument \, z = \Im \principalLogarithm z$.

The suitable analytic extension of $T_0^{-1}$ plays a crucial role in the analytic extension of $G_0$.
\begin{definition}
Consider a strip $D = \{ w \in \CC : -\lfrac{\pi}{2} \leq \Im w \leq 0 \}$ and define $T^{-1} \colon D \to \CC$ by
\begin{equation}
\label{equation:theInverseOfT}
T^{-1}(w) =
\begin{cases}
T_0^{-1}(w) & \text{ if $w \in (-\infty, \lfrac{\sqrt{3} \pi}{9}]$, } \\
g(\eta, \xi) & \text{ if $w = H(\eta, \xi) \in D \setminus (-\infty, \lfrac{\sqrt{3} \pi}{9}]$, for $(\eta, \xi) \in \Theta$, }
\end{cases}
\end{equation}
where $\Theta = \{ (\eta, \xi) \in \RR^2 : -\pi \leq \eta \leq 0 \text{ and } \xi < -\eta \}$ is an infinite trapezoid and the functions $g, H \colon \Theta \to \RR$ are defined by
\begin{equation*}
g(\eta, \xi) = \frac{\sqrt{3}i(e^{-\sqrt{3}(\eta + \xi) + i \xi} - 1)}{1 - e^{-2\sqrt{3}(\eta + \xi)}} + \frac{1+\sqrt{3}i}{2}
\end{equation*}
and
\begin{equation*}
\label{equation:theExactFormOfTheCurvedH}
\begin{split}
H(\eta, \xi) =
& - \frac{\sqrt{3}}{6} \left( \principalArgument \left(e^{-\sqrt{3}(\eta+\xi) + i \xi} - 1 \right) - \principalArgument \left( g(\eta, \xi) - \frac{1-\sqrt{3}i}{2} \right) + 2 \left \lceil \frac{\xi - \pi}{2 \pi} \right \rceil \pi + \frac{ \pi}{6} \right) \\
& \ - \frac{\ln ( \lvert g^2(\eta, \xi) - g(\eta, \xi) + 1 \rvert )}{2} + \frac{i \eta}{2} \text{.}
\end{split}
\end{equation*}

\end{definition}
It was showed~\cite[Theorem 3.4]{AD2025} that $T^{-1}$ is a continuous extension of $T_0^{-1}$ which is analytic on the topological interior of $D$. This is the inverse of some function $T$ which is some extension of $T_0$. This $T$ is not relevant here, however.

The following proposition was given in~\cite[Proposition~2.6]{AD2025}.
\begin{proposition}
\label{proposition:domainCutStripANDhalfOfLogarithm}
The function $W \colon (\CC_{+} \cup \RR) \setminus \{ -\sqrt{2}, 0, \sqrt{2} \} \mapsto \CC$, given by
\begin{equation}
\label{equation:halfOfLogarithm}
W(z) = 
\begin{cases}
\dfrac{1}{2} \principalLogarithm \bigg( 1 + \dfrac{2}{z^{2} - 2} \bigg) & \text{if $z \notin (0, \sqrt{2})$,} \\
\dfrac{1}{2} \ln \bigg( \dfrac{2}{2 - z^{2}} - 1 \bigg) - \dfrac{i \pi}{2} & \text{if $z \in (0, \sqrt{2})$,}
\end{cases}
\end{equation}
is a continuous extension of the function
\begin{equation*}
x \mapsto \frac{1}{2} \ln \left( 1 + \frac{2}{x^2 - 2} \right)
\end{equation*}
which is analytic on $\CC_{+}$.
\end{proposition}

To apply the Stieltjes inversion formula, we first use the continuous extension of $G_0$ (which is analytic on $\CC_+$) to the maximal possible subset of $\CC_+ \cup \RR$~\cite[Definition~5.3]{AD2025}.
\begin{definition}
\label{def:theExtensionOfGmu}
Let 
\begin{equation*}
\gamma_0 = \frac{2}{e^{\lfrac{2 \sqrt{3} \pi}{9}} - 1}
\end{equation*}
and let $G \colon \CC_{+} \cup \big( \RR \setminus \{ \mp \sqrt{2+\gamma_{0}} \} \big) \to \CC$ be given by
\begin{equation}
\label{equation:theExtensionOfGmu}
G(z) =
\begin{cases}
\dfrac{1}{z T^{-1}(W(z))} & \text{for $\Re z \geq 0$, $z \neq 0$, $z \neq \sqrt{2}$,} \\
-\dfrac{\sqrt{2} i}{2} e^{\sqrt{3} \pi / 9} & \text{for $z = 0$,} \\
\dfrac{\sqrt{2} - \sqrt{6}i}{4} & \text{for $z = \sqrt{2}$,}
\end{cases}
\end{equation}
and by $G(z) = -\complexAdjoint{G(-\complexAdjoint{z})}$, for $\Re z < 0$.
\end{definition}

The result below was also proven in~\cite{AD2025} (see the proof of Theorem~5.4).
\begin{theorem}
\label{thm:propertiesOfG}
The function $G$ has the following properties:
\begin{enumerate}
    \item it is a continuous extension of $x \mapsto G_{0}(x)$ ($\lambda = 0$ in~\eqref{eq:CSTreal}),
    \item its restriction to $\CC_{+}$ is analytic and is the Cauchy--Stieltjes transform of the standard V-monotone Gaussian distribution, which is supported on $[-\sqrt{2+\gamma_0}, \sqrt{2+\gamma_0}]$,
    \item its limit as $z \to \mp \sqrt{2+\gamma_0}$ is the complex $\infty$.
\end{enumerate}
\end{theorem}

We now find the Cauchy--Stieltjes transform of the limit distribution of $S_N(\lambda)$, $N \to \infty$.
\begin{proposition}
\label{proposition:CST}
The Cauchy--Stieltjes transform of the measure $\mu_{\lambda}$, $\lambda \in \RR$, has the form
\begin{equation}
\label{eq:CST}
G_{\lambda}(z) = \frac{G(z-\lambda)}{1 + \lambda G(z-\lambda)} \text{}
\end{equation}
for $z \in \CC_{+}$.
\end{proposition}

\begin{proof}
Since $G$ restricted to $\CC_{+}$ is the Cauchy--Stieltjes transform of $\mu_0$, its imaginary part is negative and therefore, since $\lambda \in \RR$, the RHS of~\eqref{eq:CST} is well-defined. Combining Theorem~\ref{thm:propertiesOfG} with~\eqref{eq:CST0}, we get our assertion.
\end{proof}

\begin{corollary}
\label{col:lambdasRestriction}
The family of measures $\mu_{\lambda}$, $\lambda \in \RR$, fulfills the following condition
\begin{equation*}
\mu_{-\lambda}(A) = \mu_{\lambda}(-A) \text{,}
\end{equation*}
for any Borel $A \subseteq \RR$, where $-A = \{ -a: a \in A \}$.
\end{corollary}

\begin{proof}
We use the Stieltjes inversion formula. Fix $\lambda$ and let $f_1, f_2 \colon \RR \times (0, \infty) \to \RR$ be the real and the imaginary part of $(x, y) \mapsto G(x+iy)$, $x \in \RR$, $y > 0$, respectively. By the definition of $G$, we have
\begin{equation*}
f_1(-x, y) = - f_1(x,y) \quad \text{and} \quad f_2(-x,y) = f_2(x,y) \text{}
\end{equation*}
for any $x + iy \in \CC_{+}$. Using this fact, we get
\begin{equation*}
\begin{split}
\Im G_{-\lambda}(x+iy)
& = \frac{f_2(x+\lambda,y)}{(1-\lambda f_1(x+\lambda,y))^2 + \lambda^2 f_2^2(x+\lambda,y)} \\
& = \frac{f_2(-x-\lambda,y)}{(1+\lambda f_1(-x-\lambda,y))^2 + \lambda^2 f_2^2(-x-\lambda,y)} = \Im G_{\lambda}(-x+iy) \text{.}
\end{split}
\end{equation*}
From the above equality, for any continuous bounded $\phi \colon \RR \to \RR$, we have
\begin{equation*}
\begin{split}
\int_{-\infty}^\infty \phi(x) \mu_{-\lambda}(\diff x)
& = -\frac{1}{\pi} \lim_{y \to 0^+} \int_{-\infty}^\infty \phi(x) \cdot \Im G_{-\lambda}(x+iy) \diff x \\
& = -\frac{1}{\pi} \lim_{y \to 0^+} \int_{-\infty}^\infty \phi(-x) \cdot \Im G_{\lambda}(x+iy) = \int_{-\infty}^\infty \phi(-x) \mu_{\lambda}(\diff x) \text{,}
\end{split}
\end{equation*}
which leads to the desired conclusion.
\end{proof}

Due to the above corollary, from now on we will consider $\lambda > 0$ only. Recall that $\mu_0$ is the standard V-monotone Gaussian distribution, established in~\cite{AD2025}. In order to use the Stieltjes inversion formula, using~\eqref{eq:CST}, we extend $G_{\lambda}$ continuously to the maximal possible subset of $\CC_{+} \cup \RR$.
\begin{proposition}
\label{proposition:CSTextended}
For $\lambda > 0$, the equation
\begin{equation}
\label{eq:CSTdenominator}
1 + \lambda G(x) = 0 \text{,} \qquad x \in \RR \setminus \{ \mp \sqrt{2+\gamma_0} \} \text{}
\end{equation}
has exactly one solution $a(\lambda) < -\sqrt{2+\gamma_0}$.

\end{proposition}

\begin{proof}
From Definition~\ref{def:theExtensionOfGmu}, we see that $x \mapsto \Re G(x + iy)$ is an odd function for any $y \geq 0$, whereas $x \mapsto \Im G(x + iy)$ is even. Let us first show that the function $\Im G$ is negative on $(-\sqrt{2+\gamma_0}, \sqrt{2+\gamma_0})$, which implies that~\eqref{eq:CSTdenominator} has no solutions on this interval.

By~\eqref{equation:theExtensionOfGmu}, we have $\Im G(0) < 0$ and therefore we consider $0 < x < \sqrt{2+\gamma_0}$ only. The reader can check that the image of this interval under $W$ is $(\RR - i \pi /2) \cup (\sqrt{3}\pi/9, \infty)$, which follows directly from~\eqref{equation:halfOfLogarithm}. That means $T(W(x)) = g(\eta, \xi)$ for some $(\eta, \xi) \in \Theta$, see~\eqref{equation:theInverseOfT}. From~\cite[Lemma~3.8]{AD2025}, we have $\Im g(\eta, \xi) > 0$, hence $\Im G(x) < 0$ for $x \in (-\sqrt{2+\gamma_0}, \sqrt{2+\gamma_0})$.

Let us first show that $G$ is strictly decreasing for $x > \sqrt{2+\gamma_0}$ and takes all the values from $(0, \infty)$. First, note that the image of $(\sqrt{2+\gamma_0}, \infty)$ under $W$ is $(0, \lfrac{\sqrt{3} \pi}{9})$, which follows directly from the definitions of $W$ and $\gamma_0$, and $T^{-1}(W(x)) = T_{0}^{-1}(W(x))$. Moreover, $x \mapsto W(x)$ is decreasing for $x > \sqrt{2}$. Since $T_{0}$ is decreasing by the definition, so is $T_{0}^{-1}$, therefore $x \mapsto G(x)$ is decreasing by~\eqref{equation:theExtensionOfGmu}. We have $T_0(0) = \lfrac{\sqrt{3} \pi}{9}$ and $T_0(1) = 0$, $T_0$ is decreasing and continuous, therefore the image of the increasing function $x \mapsto T_0^{-1}(W(x))$ under $(\sqrt{2+\gamma_0}, \infty)$ is $(0, 1)$, which implies that
\begin{equation}\label{eq:limitG}
G(x) \to 
\begin{cases}
\infty & \text{as $x \to (\sqrt{2+\gamma_0})^{+}$,} \\
0 & \text{as $x \to \infty$,}
\end{cases}
\end{equation}
i.e., the image of $(\sqrt{2+\gamma_0}, \infty)$ under $G$ is $(0, \infty)$. This implies two facts. First,~\eqref{eq:CSTdenominator} has no solutions for $x > \sqrt{2+\gamma_0}$, since $\lambda > 0$. Moreover, the function $G$ restricted to $\RR \setminus [-\sqrt{2+\gamma_0}, \sqrt{2+\gamma_0}]$ is a real-valued odd function, therefore it is strictly decreasing for $x < - \sqrt{2+\gamma_0}$ and the image of this interval is $(-\infty, 0)$. From this, we have $G(x) = -\lfrac{1}{\lambda}$ has exactly one solution $x = a(\lambda) < -\sqrt{2 + \gamma_0}$ for $\lambda > 0$, which completes the proof.
\end{proof}

\begin{corollary}
The injective function $a \colon (0, \infty) \to (-\infty, -\sqrt{2+\gamma_0})$ which is defined as the solution of~\eqref{eq:CSTdenominator} fulfills the following equation:
\begin{equation}
\label{eq:equationFor_a}
\frac{\sqrt{3}}{2} \ln \left( \frac{\lambda^2 + \lambda a + a^2}{a^2-2} \right) = \arctan \left( \frac{a+2\lambda}{\sqrt{3}a} \right) + \frac{\pi}{6} \text{,} \quad a < - \sqrt{2+\gamma_0} \text{,}
\end{equation}
where $a \coloneqq a(\lambda)$.
\end{corollary}

\begin{proof}
Fix $\lambda > 0$. The implicit form~\eqref{eq:equationFor_a} of $a$ follows directly from~\eqref{eq:CSTdenominator}. Since $G$ restricted to $\RR \setminus [-\sqrt{2+\gamma_0}, \sqrt{2+\gamma_0}]$ is an odd function and $a < -\sqrt{2+\gamma_0}$, we have
\begin{equation*}
-\frac{\lambda}{a}= \frac{1}{a G(a)} = -\frac{1}{a G(-a)} = T_0^{-1}(W(-a)) \text{,}
\end{equation*}
i.e., $T_0 ( -\lfrac{\lambda}{a} ) = W(-a)$, which is equivalent to~\eqref{eq:equationFor_a}.
\end{proof}

\begin{lemma}
\label{lemma:residueValue}
The Cauchy--Stjeltjes transform of $\mu_{\lambda}$ defined on $\CC \setminus \supp \mu_{\lambda}$ has one simple pole $z = a + \lambda$. Moreover, the residue of the transform at the pole is
\begin{equation}
\label{eq:atomWeight}
C_{\lambda} = 
\frac{\lambda \left(2 - a^{2}(\lambda) \right)}{\lambda^2a+2\lambda+2a} \text{}
\end{equation}

with $0 < C_{\lambda} < 1$.
\end{lemma}

\begin{proof}
Consider an open disk $D\subseteq \CC_{+}$ centered at $a + \lambda$ with the radius $0 < R \leq -(a+\sqrt{2+\gamma_0})$. Let $f$ be the analytic continuation of
\begin{equation*}
G_{\lambda} \restriction \{ z \in \CC : \Im z \geq 0 \text{ and } 0 < \lvert z - (a+\lambda) \rvert < R \} \text{}
\end{equation*}
to $D \setminus \{ a+\lambda \}$. We can use the Riemann--Schwarz reflection formula to define $f$, since $x \mapsto G(x-\lambda)$ is real-valued and negative for $x<\lambda-\sqrt{2+\gamma_0}$. Since $\lfrac{1}{f}$ is analytic on $D$, $f$ is meromorphic. The function $f$ is of course a restriction of the Cauchy--Stieltjes transform of $\mu_{\lambda}$ (defined on the completion of $\supp \mu_{\lambda}$) and has a pole at $z = a+\lambda$. Let us compute the limit of $(z-a-\lambda) f(z)$ as $z \to a+\lambda$. Due to the meromorphicity of $f$, we can consider only real $z$ and we substitute $x = \lambda-z$. The limit is equal to
\begin{equation}
\label{eq:atomWeight1}
\begin{split}
\lim_{x \to -a} \frac{-(x+a) G(-x)}{1 + \lambda G(-x)} = \lim_{x \to -a} \frac{(x+a) G(x)}{1 - \lambda G(x)} = G(-a) \lim_{x \to -a} -\frac{1}{\lambda G'(x)} = -\frac{1}{\lambda^2 G'(-a)} \text{}
\end{split}\end{equation}
Let us compute the derivative of $G$ for real $x > \sqrt{2 + \gamma_0}$:
\begin{equation*}\begin{split}
G'(x) = \left( \frac{1}{x T^{-1}(W(x))} \right)' = -G^2(x) \left( T^{-1}(W(x)) + \frac{x W'(x)}{T'(T^{-1}(W(x)))} \right) \text{.}
\end{split}\end{equation*}
Using the identity $T^{-1}(W(-a)) = -\lfrac{\lambda}{a}$, we have
\begin{equation*}
\lim_{z \to a+\lambda} (z-a-\lambda) G_{\lambda}(z) = \frac{1}{-\lfrac{\lambda}{a} - a \lfrac{W'(-a)}{T'(-\lfrac{\lambda}{a})}} \text{.}
\end{equation*}
Since
\begin{equation*}
T'(u) = -\frac{u}{u^2 - u + 1} \qquad \text{and} \qquad W'(x) = \frac{2}{x(2-x^2)} \text{,}
\end{equation*}
we have
\begin{equation}
\label{eq:atomWeight2}
\lim_{z \to a+\lambda} (z-a-\lambda) G_{\lambda}(z) = \frac{\lambda \left(2 - a^{2}\right)}{\lambda^2a+2\lambda+2a} = C_{\lambda} \text{.}
\end{equation}
Let $N$ and $D$ be the numerator and denominator of $C_{\lambda}$, respectively. Since $a < -\sqrt{2+\gamma_0}$, we have $N<0$. If it held
\begin{equation*}
N-D = -a ( \lambda^2 + \lambda a + 2 ) \leq 0 \text{,}
\end{equation*}
we would have $\lambda^2 + \lambda a + 2 \leq 0$, i.e., $\lambda^2 + \lambda a + a^2 \leq a^2 - 2$. Since $\lambda^2 + \lambda a + a^2$ and $a^2 - 2$ are always positive, the LHS of~\eqref{eq:equationFor_a} is nonpositive and thus the nonpositivity of the RHS implies that $\lfrac{(a+2\lambda)}{a} \leq -1$, i.e., $a+\lambda \geq 0$. This contradicts the fact that $\lambda (\lambda+a) + 2 \leq 0$. We thus have $0>N>D$, which means $0 < C_{\lambda} < 1$. Therefore $f$ has a simple pole at $a+\lambda$ and the residue $C_{\lambda}$ at this point.
\end{proof}

Define $G_{\lambda} \colon \CC_{+} \cup (\RR \setminus \{ a + \lambda \}) \to \CC$ by
\begin{equation}
\label{eq:CSTextended}
G_{\lambda}(z) =
\begin{cases}
\dfrac{G(z-\lambda)}{1 + \lambda G(z-\lambda)} & \text{for $z \neq \lambda \mp \sqrt{2+\gamma_0}$,} \\
\lfrac{1}{\lambda} & \text{otherwise,}
\end{cases}
\end{equation}
which is the continuous extension of $G_{\lambda} \colon \CC_{+} \to \CC_{-}$, also denoted by $G_{\lambda}$ with a slight abuse of notation.

We now use the Stieltjes inversion formula to find the exact form of $\mu_{\lambda}$.
\begin{theorem}
\label{thm:limitMeasure}
The measure $\mu_{\lambda}$, $\lambda > 0$, has the form
\begin{equation*}
\mu_{\lambda}( \diff x ) = \frac{\lambda \left(2 - a^{2}\right)}{\lambda^2a+2\lambda+2a} \cdot \delta_{a +\lambda}( \diff{x} ) - \frac{\Im G_{\lambda}(x)}{\pi} \, \diff{x} \text{,}
\end{equation*}
where $\delta_x$ is the Dirac measure at $x \in \RR$. Moreover, the absolutely continuous part of the measure is supported on $[\lambda - \sqrt{2+\gamma_0}, \lambda + \sqrt{2+\gamma_0}]$.
\end{theorem}

\begin{proof}
Fix $\lambda > 0$. The continuity of $G$, needed to apply the Stieltjes inversion formula, follows from the fact that
\begin{equation*}
\lim_{z \to \sqrt{2+\gamma_0}} \lfrac{1}{G(z)} = 0 \text{,} \qquad \Im z \geq 0 \text{,}
\end{equation*}
which was shown in the proof of~\cite[Theorem~5.4]{AD2025}. Due to Lemma~\ref{lemma:residueValue}, the point $a+\lambda$ is an isolated point of $\supp \mu_{\lambda}$ and it is an atom of the measure with the weight $C_{\lambda}$ (see~\cite[Proposition~1.104]{HO}).

Let $\phi: \RR \to \RR$ be a continuous function with a compact support such that $a+\lambda \notin \supp \phi$. For such function we have
\begin{equation*}
\lim \limits_{y \to 0^{+}} \phi(x) \Im G_{\lambda}(x+iy) = \phi(x) \Im G_{\lambda}(x)
\end{equation*}
for any $x \in \RR \setminus \{ a + \lambda \}$, which implies that
\begin{equation*}
\lim \limits_{y \to 0^{+}} \int_{-\infty}^{\infty} \phi(x) \Im G_{\lambda}(x+iy) \diff{x} = \int_{-\infty}^{\infty} \phi(x) \Im G_{\lambda}(x) \diff{x} \text{.}
\end{equation*}
The convergence above is guaranteed by the Lebegue's dominant convergence theorem with the majorant
\begin{equation*}
\lvert \phi(x) \rvert \cdot \max_z \lvert G_{\lambda}(z) \rvert \text{,}
\end{equation*}
where the maximum is taken over $\supp \phi + i \cdot [0, 1]$ from which it follows that $\mu_{\lambda}$ restricted to $\RR \setminus \{ a + \lambda \}$ equals $-\lfrac{1}{\pi} \Im G_{\lambda}(x) \diff{x}$. This finishes the proof.
\end{proof}

We present the plots of $\mu_{\lambda}$ for several values of $\lambda$ in Fig.~\ref{figure:measures}. We now prove several facts about the atom of a $\mu_{\lambda}$.
\begin{figure}
\centering
    \includegraphics[scale=0.5]{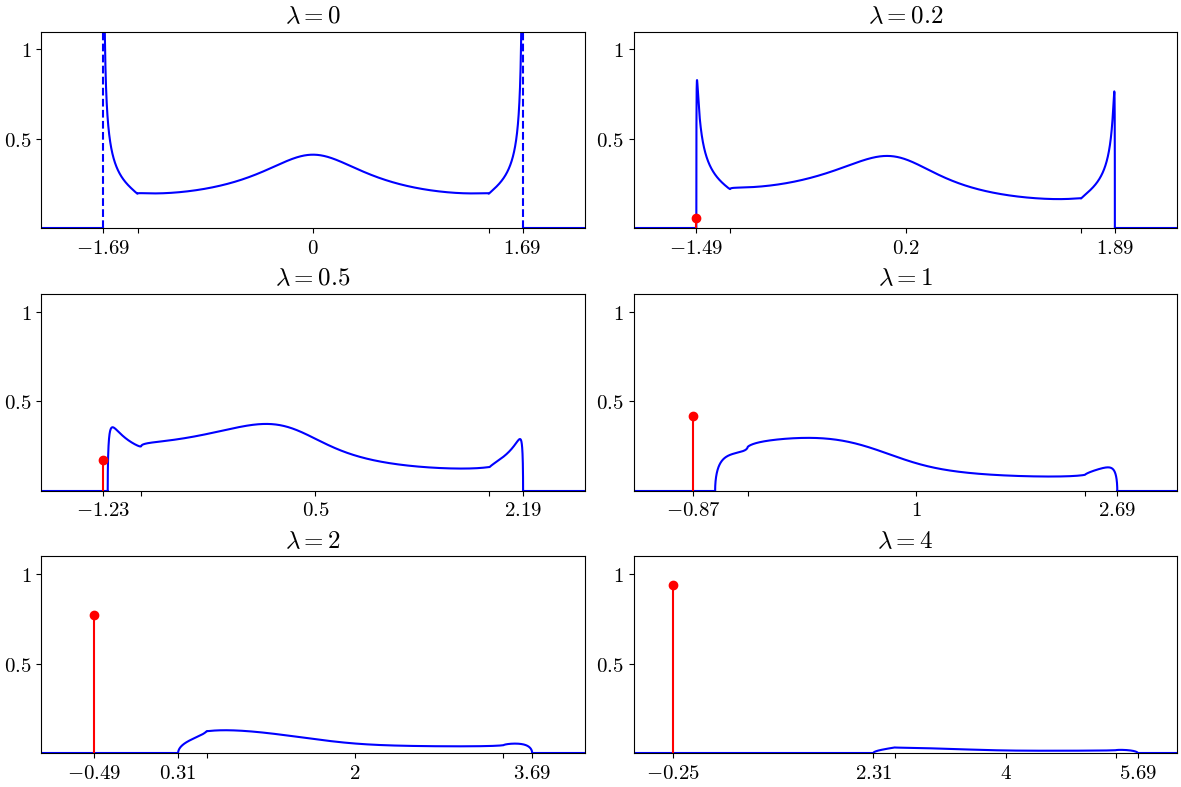}
    \caption{Limit distribution for $\lambda \in \{ 0, \lfrac{1}{5}, \lfrac{1}{2}, 1, 2, 4 \}$.}
\label{figure:measures}
\end{figure}

\begin{theorem}
\label{thm:behaviorOfAtom}
The function $\lambda \mapsto a(\lambda)$ has the following properties:
\begin{enumerate}
    \setlength\itemsep{1mm}
    \item\label{test} $a'(\lambda) = -C_{\lambda}$ and thus $a$ is decreasing,
    \item $\lim \limits_{\lambda \to 0^{+}} a(\lambda) = -\sqrt{2+\gamma_0}$ and $\lim \limits_{\lambda \to \infty} a(\lambda) = -\infty$,
    \item $ -\lambda -\min(\sqrt{2+\gamma_0}, \lfrac{2}{\lambda})< a(\lambda) < -\max(\sqrt{2+\gamma_0}, \lambda)$.
\end{enumerate}
\end{theorem}

\begin{proof}
From the Proposition \ref{proposition:CSTextended}, $G\restriction \RR \setminus [-\sqrt{2+\lambda_0}, \sqrt{2+\lambda_0}]$ is an odd function. Then
\begin{equation*}
  a'(\lambda) = \left(G^{-1}\left(-\lfrac{1}{\lambda}\right)\right)'= \frac{1}{\lambda^2} \cdot \frac{1}{G'\left(G^{-1}\left(\lfrac{1}{\lambda}\right)\right)} 
               =\frac{1}{\lambda^2} \cdot \frac{1}{G'(-a)}=-C_{\lambda}, 
\end{equation*}
where the last equality follows from~\eqref{eq:atomWeight1}. Moreover, since $0<C_{\lambda}<1$, $a(\lambda)$ is decreasing, and this proves~\eqref{test}. The second statement follows from~\eqref{eq:limitG} since $a(\lambda)=G^{-1}\left(-\lfrac{1}{\lambda}\right)$. Now, since $a(0^+)=\sqrt{2+\gamma_0}$ and $a'(\lambda)\in(-1, 0)$, we have
$$-\lambda-\sqrt{2+\gamma_0}<a(\lambda)<-\sqrt{2
+\gamma_0}.$$
From the proof of Lemma~\ref{lemma:residueValue}, we obtain $\lambda^2+\lambda a+2>0$. Combining it with~\eqref{eq:equationFor_a} as in the proof, we get $a+\lambda<0$, which leads to 
$$-\lambda-\frac{2}{\lambda}<a(\lambda)<-\lambda.$$
This proves the last statement.
\end{proof}

In Fig.~\ref{figure:atom}, we present the graphs of the position (left) and the weight (right) of the unique atom as the functions of $\lambda \geq 0$.
\begin{figure}
\centering
    \includegraphics[scale=0.5]{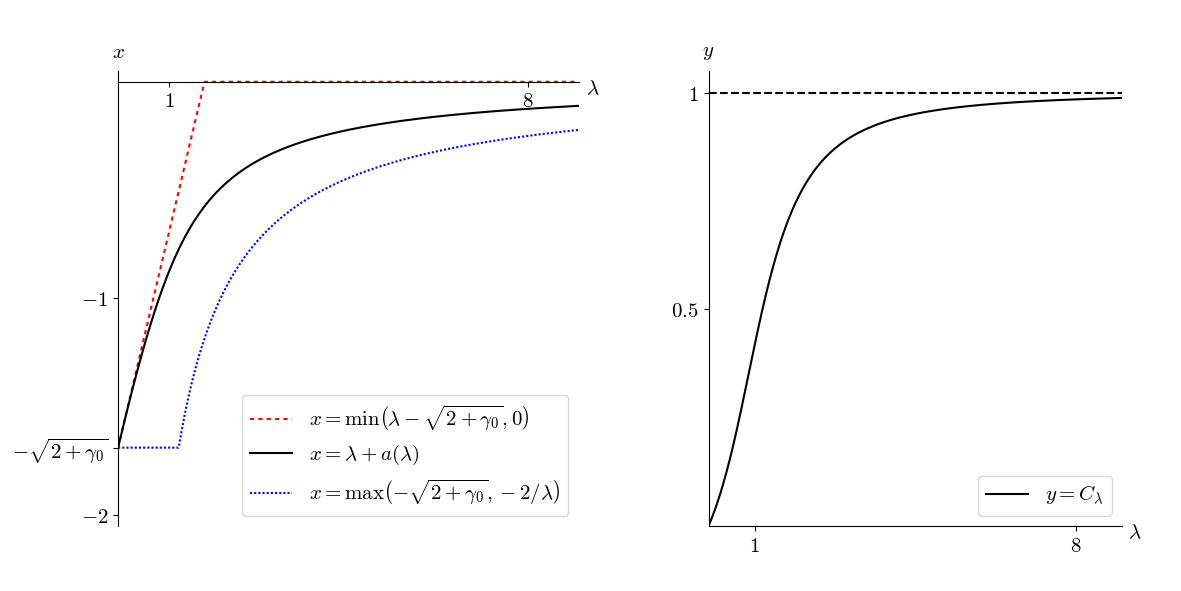}
    \caption{Position and weight of the atom.}
\label{figure:atom}
\end{figure}

\section*{Acknowledgements}
The first named author has been supported by Wroc\l{}aw University of Environmental and Life Sciences (Grant number B010/0019/24).
\newpage
\bibliographystyle{plain}
\bibliography{main.bib}

\end{document}